\documentclass{article}

\usepackage{geometry}
\usepackage{bm}
\usepackage{easybmat}
\usepackage{etex}
\usepackage{dsfont}
\usepackage{tabularx}
\usepackage{stmaryrd}
\usepackage{amsthm}
\usepackage{amsmath}%
\usepackage{amsfonts}%
\usepackage{amssymb}%
\usepackage{graphicx}
\usepackage{array}
\usepackage{booktabs}
\usepackage[usenames]{color}
\usepackage{epstopdf}
\usepackage{caption}
\usepackage{subcaption}
\usepackage{float}
\usepackage{algorithm}
\usepackage{algpseudocode}
\usepackage{nomencl}
\usepackage[font={small}]{caption}

\newtheorem{theorem}{Theorem}[section]

\newtheorem{corollary}[theorem]{Corollary}

\newtheorem{definition}[theorem]{Definition}

\newtheorem{lemma}[theorem]{Lemma}

\newtheorem{proposition}[theorem]{Proposition}
\newtheorem{remark}[theorem]{Remark}

\numberwithin{equation}{section}
\algrenewcommand\alglinenumber[1]{\tiny #1:}
\definecolor{gosia}{RGB}{50,200,50}
\definecolor{damian}{RGB}{255,0,0}
\definecolor{wjs}{RGB}{0,0,255}
\newcommand{\T}[1]{^\mathsf{#1}}

\newcommand{\argmax}[1]{\underset{#1}{\operatorname{arg}\,\operatorname{max}}\;}
\newcommand{\argmin}[1]{\underset{#1}{\operatorname{arg}\,\operatorname{min}}\;}
\newcommand{\minimize}[1]{\underset{#1}{\operatorname{minimize}}\;}
\newcommand{\maximize}[1]{\underset{#1}{\operatorname{maximize}}\;}
\newcommand{\I}[1]{\llbracket {#1}\rrbracket_I}
\newcommand{\iI}[1]{\llbracket {#1}\rrbracket_{\mathbb{I}}}
\newcommand{\XI}[1]{\llbracket {#1}\rrbracket_{X,I}}
\newcommand{\II}{\mathbb{I}}
\newcommand{\y}{\widetilde{y}}
\newcommand{\X}{\widetilde{X}}
\newcommand{\Beta}{\widetilde{\beta}}

\newcommand{\ES}[1]{\textnormal{\tiny {#1}}}
\renewcommand{\arraystretch}{1.2}
\newgeometry{tmargin=3cm, bmargin=3cm} 

\newcommand{\Chi}{\raisebox{2pt}{$\chi$}}

\makenomenclature

\floatname{algorithm}{Procedure}

\begin{document}
\begin{center}{\large\bf  Group SLOPE - adaptive selection of groups of predictors}
\vspace{0.2in}

\noindent{Damian Brzyski$^{a,b}$,  Weijie Su$^{c}$, Ma\l{}gorzata Bogdan$^{a,d}$,}
\vspace{0.2in}

\noindent $\mbox{}^{a}$ {\it \small Faculty of Mathematics, Wroclaw University of
Technology, Poland}

\noindent $\mbox{}^{b}$ {\it \small  Institute of Mathematics, Jagiellonian University, Cracow, Poland}
 
\noindent  $\mbox{}^{c}$ {\it \small  Department of Statistics, Stanford University, Stanford, CA 94305, USA}

\noindent $\mbox{}^{d}$ {\it \small  Institute of Mathematics, Wroclaw University, Poland}

\vspace{0.2in}

\noindent Key words:  Multiple Regression, Model Selection, Group selection, SLOPE, False Discovery Rate, Asymptotic Minimax

\end{center}

\begin{abstract}
Sorted L-One Penalized Estimation (SLOPE, \cite{SLOPE}) is a relatively new convex optimization procedure which allows for adaptive selection of regressors under sparse high dimensional designs. Here we extend the idea of SLOPE to deal with the situation when one aims at selecting whole groups of explanatory variables instead of single regressors. This approach is particularly useful when variables in the same group are strongly correlated and thus true predictors are difficult to distinguish from their correlated "neighbors"'. We formulate the respective convex optimization problem, gSLOPE (group SLOPE), and propose an efficient algorithm for its solution. We also define a notion of the group false discovery rate (gFDR) and provide a choice of the sequence of tuning parameters for gSLOPE so that gFDR is provably controlled at a prespecified level if the groups of variables are orthogonal to each other. Moreover, we prove that the resulting procedure adapts to unknown sparsity and is asymptotically minimax with respect to the estimation of the proportions of variance of the response variable explained by regressors from different groups. We also provide a method for the choice of the regularizing sequence  when variables in different groups are not orthogonal but statistically independent and illustrate its good properties  with computer simulations. 

\end{abstract}



\section{Introduction}

Consider the classical multiple regression model of the form
\begin{equation}
\label{mainmodel}
y=X\beta+z,\end{equation} where $y$ is the $n$ dimensional vector of values of the response variable, $X$ is the $n$ by $p$ experiment (design) matrix and $z\sim\mathcal{N}(0,\sigma^2 \mathbf{I}_n)$. We assume that $y$ and $X$ are known, while $\beta$ is unknown. 
In many cases of data mining the purpose of the statistical analysis is to recover the support of $\beta$, which identifies the set of important regressors. Here, the true support corresponds to truly relevant variables (i.e. variables which have impact on observations).  Common procedures to solve this model selection problem rely on minimization of some objective function consisting of the weighted sum of two components: first term responsible for the goodness of fit and second term penalizing the model complexity. 
Among such procedures one can mention classical model selection criteria like the Akaike Information Criterion (AIC) \cite{Aka} and the Bayesian Information Criterion (BIC) \cite{Schw}, where the  penalty depends on the number of variables included in the model,  or  LASSO \cite{LASSO}, where the penalty depends on the $\ell_1$ norm of regression coefficients. The main advantage of LASSO over classical model selection criteria is that it is a convex optimization problem and, as such, it can be easily solved even for very large design matrices.  

LASSO solution is obtained by solving the optimization problem 
\begin{equation}
\label{LASSO}
\argmin b\ \ \bigg\{\frac 12\big\|y-Xb\big\|^2+\lambda_L\|b\|_1\bigg\},
\end{equation}
where $\lambda_L$ is a tuning parameter defining the trade-off between the model fit and the sparsity of solution. 
In practical applications the selection of good $\lambda_L$ might be very challenging. For example it has been reported that in high dimensional settings  the popular cross-validation  typically leads  to detection of  a large number of false regressors (see e.g. \cite{SLOPE}). 
The general rule is that when one reduces $\lambda_L$, then LASSO can identify more elements from the true support (true discoveries) but at the same time it produces more false discoveries.  In general the numbers of true and false discoveries for a given  $\lambda_L$ depend on unknown properties on the data generating mechanism, like the number of true regressors and the magnitude of their effects. A very similar problem occurs when selecting  thresholds for individual tests in the context of multiple testing. Here it was found that the popular Benjamini-Hochberg rule (BH, \cite{BH}), aimed at control of the False Discovery Rate (FDR), adapts to the unknown data generating mechanism and has some desirable optimality properties under a variety of statistical settings (see e.g. \cite{ABDJ, ABOS, ABOS2, ABOS3}). The main property of this rule is that it relaxes the thresholds along the sequence of test statistics, sorted in the decreased order of magnitude. Recently the same idea was used in a new generalization of LASSO, named SLOPE (Sorted L-One Penalized Estimation, \cite{SLOPE2, SLOPE}). Instead of the $\ell_1$ norm (as in LASSO case), the method uses FDR control properties of $J_{\lambda}$ norm, defined as follows; for sequence $\{\lambda\}_{i=1}^p$ satisfying $\lambda_1\geq\ldots\geq\lambda_p\geq0$ and $b\in\mathbb{R}^p$, $J_{\lambda}(b):=\sum_{i=1}^p\lambda_i|b|_{(i)},$ where $|b|_{(1)}\geq\ldots\geq |b|_{(p)}$ is the vector of sorted absolute values of coordinates of $b$. SLOPE is the solution to a convex optimization problem
\begin{equation}
\label{SLOPE}
\argmin b\ \ \bigg\{\frac 12\big\|y-Xb\big\|^2+J_{\lambda}(b)\bigg\},
\end{equation}
which clearly reduces to LASSO  for $\lambda_1=\ldots=\lambda_p=:\lambda_L$.
Similarly as in classical model selection, the support of solution defines the subset of variables estimated as relevant. It was shown in \cite{SLOPE} that SLOPE is strongly connected with BH procedure under orthogonal case, i.e. when $X^TX=\mathbf{I}_n$. The main theoretical result presented in \cite{SLOPE} states that under such assumption, the sequence of tuning parameters could be specifically selected, such that the FDR control is guaranteed. Moreover, in \cite{WE} it is proved that SLOPE with this sequence of tuning parameters adapts to unknown sparsity and is asymptotically minimax under orthogonal and random Gaussian designs.  

In the sequence of examples presented in \cite{SLOPE2} and \cite{SLOPE} it was shown that SLOPE has very desirable properties in terms of FDR control in case when the  regressor variables are weakly correlated. While there exist other interesting approaches which allow to control FDR also under correlated designs (e.g. \cite{ko}),  the efforts to prevent detection of false regressors which are strongly correlated with true ones inevitably lead to a huge loss of power. 
An alternative approach to deal with strongly correlated predictors is to simply give up the idea of distinguishing between them and include all of them into the selected model as a group.
This leads to the problem of group selection in linear regression, extensively investigated and applied in many fields of science.
In many of these applications the groups are selected not only due to the strong correlations but also taking into account the problem specific scientific knowledge.


Probably the most known convex optimization method for selection of gropus of explanatory variables is the group LASSO (gLASSO) \cite{Bakin}. For a fixed tuning parameter, $\lambda_{gL}>0$, the gLASSO estimate is most frequently (e.g. \cite{gLASSO1}, \cite{gLASSO5}) defined as a solution to optimization problem
\begin{equation}
\label{gLASSO}
\argmin b\ \ \bigg\{\frac 12\Big\|y-\sum_{i=1}^mX_{I_i}b_{I_i}\Big\|_2^2+\sigma\lambda_{gL}\sum_{i=1}^m\sqrt{|I_i|}\|b_{I_i}\|_2\bigg\},
\end{equation}
where the sets $I_1,\ldots,I_m$ form the partition of the set $\{1,\ldots,p\}$, $|I_i|$ denotes the number of elements in set $I_i$, $X_{I_i}$ is the submatrix of $X$ composed of columns indexed by $I_i$ and $\beta_{I_i}$ is the restriction of $\beta$ to indices from $I_i$. The method introduced in this article is, however, closer to the alternative version of gLASSO, in which penalties are imposed on $\|X_{I_i}b_{I_i}\|_2$ rather than $\|b_{I_i}\|_2$. This method was formulated in $\cite{gLASSO6}$, where authors defined estimate of $\beta$ as
\begin{equation}
\label{gLASSO2}
\beta^\ES{gL}:=\argmin b\ \ \bigg\{\frac 12\Big\|y-\sum_{i=1}^mX_{I_i}b_{I_i}\Big\|_2^2+\sigma\lambda_{gL}\sum_{i=1}^m\sqrt{|I_i|}\|X_{I_i}b_{I_i}\|_2\bigg\},
\end{equation}
with the condition $\|X_{I_i}\beta^\ES{gL}_{I_i}\|_2>0$ serving as a group relevance indicator.

Similarly as in the context of regular model selection, the properties of gLASSO strongly depend on the smoothing parameter $\lambda_{gL}$, whose optimal value is the function of unknown parameters of true data generating mechanism. Thus, a natural question arises if the idea of SLOPE can be used for construction of the similar adaptive procedure  for the group selection. To answer this query in this paper we define and investigate the properties of the group SLOPE (gSLOPE). We formulate the respective optimization problem  and provide the algorithm for its solution. We also define the notion of the group FDR (gFDR), and provide the theoretical choice of the sequence of smoothing parameters, which guarantees that SLOPE controls gFDR in the situation when variables in different groups are orthogonal to each other. Moreover, we prove that the resulting procedure adapts to unknown sparsity and is asymptotically minimax with respect to the estimation of the proportions of variance of the response variable explained by regressors from different groups. Additionally, we provide the way of constructing the sequence of smoothing parameters under the assumption that the regressors from distinct groups are independent and use computer simulation to show that it allows to control gFDR.
\section{Group SLOPE}
\subsection{Formulation of the optimization problem}
\label{subsec:gs2711944}
Let the design matrix $X$ belong to the space $M(n,p)$ of matrices with $n$ rows and $p$ columns. Furthermore, suppose that $I=\{I_1,\ldots,I_m\}$ is some partition of the set $\{1,\ldots,p\}$, i.e. $I_i$'s are nonempty sets, $I_i\cap I_j=\emptyset$ for $i\neq j$ and $\bigcup I_i = \{1,\ldots,p\}$. We will consider the linear regression model with $m$ groups of the form
\begin{equation}
\label{gmodel}
y=\sum_{i=1}^mX_{I_i}\beta_{I_i}+z,
\end{equation}
where $X_{I_i}$ is the submatrix of $X$ composed of columns indexed by $I_i$ and $\beta_{I_i}$ is the restriction of $\beta$ to indices from the set $I_i$. We will use notations $l_1,\ldots,l_m$ to refer to the ranks of submatrices $X_{I_1},\ldots,X_{I_m}$. To simplify notations in further part, we will assume that $l_i>0$ (i.e. there is at least one nonzero entry of $X_{I_i}$ for all $i$). Besides this, $X$ may be absolutely arbitrary matrix, in particular any linear dependencies inside submatrices $X_{I_i}$ are allowed.

In this article we will treat the value $\|X_{I_i}\beta_{I_i}\|_2$ as a measure of an impact of $i$th group on the response and we will say that the group $i$ is truly relevant if and only if $\|X_{I_i}\beta_{I_i}\|_2>0$. Thus our task of the identification of the relevant groups is equivalent with finding the support of the vector $\XI{\beta}:= \big(\|X_{I_1} \beta_{I_1}\|_2, \ldots, \|X_{I_m} \beta_{I_m}\|_2\big)^\mathsf{T}$.

To estimate the nonzero coefficients of $\XI{\beta}$, we will use a new penalized method, namely group SLOPE (gSLOPE). For a given sequence of nonincreasing, nonnegative tuning parameters, $\lambda_1,\ldots,\lambda_m$, given sequence of positive weights, $w_1,\ldots,w_m$, and design matrix, $X$, the gSLOPE, $\beta^\ES{gS}$, is defined as any solution to 
\begin{equation}
\label{gSLOPE}
\beta^\ES{gS}: = \argmin b\ \ \bigg\{\frac 12\Big\|y-Xb\Big\|_2^2+\sigma J_{\lambda}\Big( W\XI{b}\Big)\bigg\},
\end{equation}
where $W$ is diagonal matrix defined by equations $W_{i,i}:=w_i,$ for $i=1,\ldots,m$. The estimate of $\XI{\beta}$ support is simply defined by the indices corresponding to nonzeros of $\XI{\beta^\ES{gS}}$.

It is easy to see that when one considers $p$ groups containing only one variable (i.e. single-groups situation), then taking all weights equal to one reduces (\ref{gSLOPE}) to SLOPE (\ref{SLOPE}). On the other hand, taking $w_i=\sqrt{|I_i|}$ and putting $\lambda_1=\ldots=\lambda_m=:\lambda_{gL}$, immediately gives gLASSO problem (\ref{gLASSO2}) with the smoothing parameter $\lambda_{gL}$. The gSLOPE could be therefore treated both: as the extension to SLOPE, and the extension to group LASSO.

Now, let us define $\widetilde{p}=l_1+\ldots+l_m$ and consider the following partition, $\II=\{\II_1,\ldots,\II_m\}$, of the set $\{1,\ldots,\widetilde{p}\}$
\begin{equation}
\II_1:=\big\{1,\ldots, l_1\big\},\quad \II_2:=\big\{l_1+1,\quad \ldots, l_1+l_2\big\},\ \ldots,\quad \II_m:=\Big\{\sum_{j=1}^{m-1}l_i+1,\ldots, \sum_{j=1}^{m}l_i\Big\}.
\end{equation}
Observe that each $X_{I_i}$ can be represented as $X_{I_i}=U_iR_i$, where $U_i$ is any matrix with $l_i$ orthogonal columns of a unit $l_2$ norm, whose span coincides with the space spanned by the columns of $X_{I_i}$ , and $R_i$ is the corresponding matrix of a full row rank. Therefore, for the purpose of estimating the group effects, we can reduce the design matrix
$X$ to the matrix $\widetilde{X}\in M(m,\widetilde{p})$ such that $\widetilde{p}=l_1+\ldots+l_m$ and $\widetilde{X}_{\II_i}=U_i$ for all $i$.

Now observe that denoting  $c_{\II_i}:=R_ib_{I_i}$ for $i\in\{1,\ldots,m\}$ we immediately obtain
\begin{equation}
\arraycolsep=1.4pt\def\arraystretch{1.8}
\begin{array}{c}
Xb=\sum\nolimits_{i=1}^mX_{I_i}b_{I_i} = \sum\nolimits_{i=1}^mU_iR_ib_{I_i} = \sum\nolimits_{i=1}^m\widetilde{X}_{\II_i}c_{\II_i}=\widetilde{X}c,\\
\Big(\XI{b}\Big)_i=\|X_{I_i}b_{I_i}\|_2 = \|R_ib_{I_i}\|_2= \|c_{\II_i}\|_2	
\end{array}
\end{equation}
and the problem (\ref{gSLOPE}) can be equivalently presented in the form
\begin{equation}
\label{24111029}
\left\{
\begin{array}{l}
c^\ES{gS}: = \argmin {c}\ \ \bigg\{\frac12\big\|y-\widetilde{X}c\big\|_2^2+\sigma J_{\lambda}\Big( W\iI{c}\Big)\bigg\} \\
c^\ES{gS}_{\II_i}:=R_i\beta^\ES{gS}_{I_i},\ i=1,\ldots,m
\end{array}
\right.,
\end{equation}
for $\iI{c}:= \big(\|c_{\II_1}\|_2, \ldots, \|c_{\II_m}\|_2\big)^\mathsf{T}.$ 
Therefore to identify the relevant groups and estimate their group effects it is enough to solve the optimization problem \eqref{24111029}. We will say that (\ref{24111029}) is the standardized version of the problem \eqref{gSLOPE}.

\begin{remark} 
At the time of finishing this article, we have been informed that the similar formulation of the group SLOPE was proposed in an independent work of Alexej Gossmann et al. \cite{Gossmann2015}. However \cite{Gossmann2015} considers only the case when the weights $w_i$ are equal to the square root of the group size and penalties are imposed directly on $\|\beta_{I_i}\|_2$ rather than on group effects $\|X_{I_i} \beta_{I_i}\|_2$. This makes the method of \cite{Gossmann2015} dependent on scaling or rotations of variables in a given group.  In comparison to \cite{Gossmann2015}, who propose a Monte Carlo approach for estimating the regularizing sequence, our article provides the choice of the smoothing parameters which provably allows for FDR control in case where the regressors in different groups are orthogonal to each other and, according to our simulation study, allows for FDR control where regressors in different groups are independent. 
\end{remark}  

\subsection{Group FDR}

Group SLOPE is designed to select groups of variables, which might be very strongly correlated within a group or even linearly dependent. In this context we do not intend to identify single important predictors but rather want to point at the groups which contain at least one true regressor. To theoretically investigate the properties of gSLOPE in this context we now introduce the respective notion of group FDR (gFDR).
\begin{definition}
Consider model (\ref{gmodel}) with some design matrix $X$. Let $\beta^\ES{gS}$ be an estimate given by (\ref{gSLOPE}). We define two random variables: the number of all groups selected by gSLOPE (Rg) and the number of groups falsely discovered by gSLOPE (Vg), as
$$Rg:=\big | \big\{i:\ \|X_{I_i}\beta^\ES{gS}_{I_i}\|_2\neq 0\big\}\big |, \qquad Vg:=\big |\big\{i:\ \|X_{I_i}\beta_{I_i}\|_2=0,\ \|X_{I_i}\beta^\ES{gS}_{I_i}\|_2 \neq 0\big\}\big |.$$
\end{definition}
\begin{definition}
We define the false discovery rate for groups (gFDR) as
\begin{equation}
gFDR: = \mathbb{E}\left [ \frac{Vg}{\max\{Rg, 1\}} \right].
\end{equation}
\end{definition}
Our goal is the identification of the regularizing sequence for SLOPE such that gFDR can be controlled at any given level $q\in(0,1)$. In the next section we will provide such a sequence, which provably controls gFDR in case when variables in different groups are orthogonal to each other. In subsequent sections we will replace this condition with the weaker assumption of the stochastic independence of regressors in different groups.
 

\subsection{Control of gFDR when variables from different groups are orthogonal}

In this section we will present the sequence of smoothing parameters $\lambda_i$ for gSLOPE, which guarantees gFDR control at the assumed level if variables   from different groups are orthogonal to each other. Before the statement of the respective theorem, we will recall the definition of $\chi$ distribution and define a scaled $\chi$ distribution.
\begin{definition}
We will say that a random variable $X_1$ has a $\Chi$ distribution with $l$ degrees of freedom, and write $X_1 \sim \Chi_l$, when $X_1$ could be expressed as $X_1 = \sqrt{X_2},$ for $X_2$ having a $\Chi^2$ distribution with $l$ degrees of freedom.
We will say that a random variable $X_1$ has a scaled $\Chi$ distribution with $l$ degrees of freedom and scale $\mathcal{S}$, when $X_1$ could be expressed as $X_1 = \mathcal{S}\cdot X_2,$ for $X_2$ having a $\Chi$ distribution with $l$ degrees of freedom. We will use the notation $X_1 \sim \mathcal{S}\Chi_l$. 
\end{definition}
\begin{theorem}[gFDR control under orthogonal case]
\label{gFDRcontrol}
Consider model (\ref{gmodel}) with the design matrix $X$ satisfying $X_{I_i}^\mathsf{T}X_{I_j}=0$, for any $i\neq j$. Denote the number of zero coefficients in $\XI{\beta}$ by $m_0$ and let $w_1,\ldots,w_m$ be positive numbers. Moreover,
\begin{itemize}
\setlength\itemsep{1pt}
\item let $W$ be diagonal matrix such as $W_{i,i}=w_i$, for $i=1,\ldots, m$,
\item for each $i$ denote by $l_i$ the rank of submatrix $X_{I_i}$,
\item define $\lambda=(\lambda_1,\ldots,\lambda_m)^\mathsf{T}$, with $\lambda_i:=\max\limits_{j=1,\ldots,m}\left\{\frac1{w_j}F^{-1}_{\Chi_{l_j}}\left (1-\frac{q\cdot i}{m}\right )\right\}$, where $F_{\Chi_{l_j}}$ is a cumulative distribution function of $\Chi$ distribution with $l_j$ degrees of freedom,
\end{itemize}
Then any solution, $\beta^\ES{gS}$, to problem gSLOPE \eqref{gSLOPE} generates the same vector $\XI{\beta^\ES{gS}}$ and it holds
\begin{equation*}
gFDR =\mathbb{E}\left [ \frac{Vg}{\max\{Rg, 1\}} \right]\leq q \cdot \frac{m_0}{m}.
\end{equation*}
\end{theorem}

\begin{proof}
We will start with the standardized version of the gSLOPE problem, given by \eqref{24111029}. Based on results discussed in Appendix \ref{Sec:alt_rep}, we can consider problem 
\begin{equation}
\label{17022353}
\left\{
\begin{array}{l}
c^*=\argmin{c}\left\{\frac12\sum_{i=1}^m\big(\|\widetilde{y}_{\II_i}\|_2-w^{-1}_ic_i\big)^2+J_{\sigma\lambda}(c)\right\}\\
\|X_{I_i}\beta^\ES{gS}_{I_i}\|_2=c^*_i \big(w_i\|\widetilde{y}_{\II_i}\|_2\big)^{-1}\widetilde{y}_{\II_i} ,\quad i=1,\ldots,m
\end{array}
\right.
\end{equation}
as an equivalent formulation of \eqref{24111029} and work with the model $\widetilde{y}\sim \mathcal{N}\big(\widetilde{\beta},\ \sigma \mathbf{I}_{\widetilde{p}}\big)$, with $\y=\X^T y$ and $\widetilde{\beta}_{\II_i}=R_i\beta_{I_i}$. The uniqueness of $\XI{\beta^\ES{gS}}$ follows simply from the uniqueness of $c^*$ in \eqref{17022353}. Define random variables ${R:=\big | \big\{i:\ c_i^*\neq 0\big\}\big |}$ and ${V:=\big |\big\{i:\ \|\widetilde{\beta}_{\II_i}\|_2=0,\quad c^*_i \neq 0\big\}\big |}$. Clearly, then $Rg=R$ and $Vg=V$. Consequently, it is enough to show that
$$\mathbb{E}\left[\frac{V}{\max\{R, 1\}}\right]\leq q \cdot \frac{m_0}{m}.$$
Without loss of generality we can assume that groups $I_1,\dots,I_{m_0}$ are truly irrelevant, which gives $\|\widetilde{\beta}_{\II_1}\|_2=\ldots=\|\widetilde{\beta}_{\II_{m_0}}\|_2=0$ and $\|\widetilde{\beta}_{\II_j}\|_2>0$ for $j>m_0.$ Suppose now that $r,i$ are some fixed indices from $\{1,\ldots,m\}$. From definition of $\lambda_r$
\begin{equation}
\lambda_r\geq \frac1{w_i}F^{-1}_{\chi_{l_i}}\left(1-\frac{qr}m\right)\ \Longrightarrow\ 1-F_{\chi_{l_i}}\left(\lambda_rw_i\right)\leq\frac{qr}m.
\end{equation}
Since $\sigma^{-1}\|\widetilde{y}_{\II_i}\|_2\sim\chi_{l_i}$, for $i\leq m_0$ we have
\begin{equation}
\label{09022303}
\mathbb{P}\left(w_i^{-1}\|\widetilde{y}_{\II_i}\|_2\geq \sigma\lambda_r\right)=\mathbb{P}\left(\sigma^{-1}\|\widetilde{y}_{\II_i}\|_2\geq \lambda_rw_i\right)=1-F_{\chi_{l_i}}\left(\lambda_rw_i\right)\leq \frac{qr}m.
\end{equation}
Now, denote by $\widetilde R^i$ the number of nonzero coefficients in SLOPE estimate of \eqref{17022353} after reducing sample by excluding $i$ variable, as it was described in the statement of \ref{lemma3}. Thanks to lemmas \ref{lemma2} and \ref{lemma3}, we immediately get
\begin{equation}
\label{09022302}
\big\{\iI{\widetilde{y}}:\ c^*_i\neq 0\textrm{ and }R=r\big\}\subset\big\{\iI{\widetilde{y}}:\ w_i^{-1}\|\widetilde{y}_{\II_i}\|_2>\sigma\lambda_r\textrm{ and }\widetilde{R}^i=r-1\big\},
\end{equation}
which together with (\ref{09022302}) and (\ref{09022303}) raises
\begin{equation}
\begin{split}
\mathbb{P}(c_i^*\neq0\textrm{ and }R=r)\ \leq\ &\mathbb{P}\left(w_i^{-1}\|\widetilde{y}_{\II_i}\|_2>\sigma\lambda_r\textrm{ and }\widetilde{R}^i=r-1\right)=\\
&\mathbb{P}\left(w_i^{-1}\|\widetilde{y}_{\II_i}\|_2>\sigma\lambda_r\right) \mathbb{P}\left(\widetilde{R}^i=r-1\right)\ \leq\\
&\frac{qr}m\mathbb{P}\left(\widetilde{R}^i=r-1\right).
\end{split}
\end{equation}
Therefore
\begin{equation}\label{eq:thm_last}
\begin{split}
&\mathbb{E}\left[\frac{V}{\max\{R, 1\}} \right]=\sum_{r=1}^m\mathbb{E}\left[\frac{V}{r}\mathds{1}_{\{R=r\}}\right] = \sum_{r=1}^m\frac1r\mathbb{E}\left[\sum_{i=1}^{m_0}\mathds{1}_{\{c^*_i\neq0\}}\mathds{1}_{\{R=r\}}\right]=\\
&\sum_{r=1}^m\frac1r\sum_{i=1}^{m_0}\mathbb{P}\left(c^*_i\neq0\textrm{ and }R=r\right)\ \leq\ \sum_{i=1}^{m_0}\frac qm\sum_{r=1}^m\mathbb{P}\big(\widetilde{R}^i=r-1\big) = \frac{qm_0}m,
\end{split}
\end{equation}
which finishes the proof.
\end{proof}
In further part of this article, we will use the term basic lambdas and use the notation $\lambda^{max}$ to refer to the sequence of tuning parameters defined in Theorem \ref{gFDRcontrol}. Figure \ref{19091846}(a) illustrates the gFDR achieved by gSLOPE with design matrix $X=\mathbf{I}_p$ (hence the rank of $i$ group, $l_i$, coincides with its size) and for the sequence $\lambda^{max}$. In simulation we have fixed $5$ groups sizes from the set $\{3, 4, 5, 6, 7\}$, and for each size $200$ groups were considered, which gave $p = n = 5000$ and $m=1000$. Signal sizes were generated such that $\big(\XI{\beta}\big)_i=a\sqrt{l_i}$ for truly relevant groups, which were randomly chosen in each iteration. Parameter $a$ was selected to satisfy the condition $\frac1m\sum_{i=1}^ma\sqrt{l_i} = \frac1m\sum_{i=1}^mB(m,l_i)$, where $B(m_i,l)$ is defined in (\ref{05031641}).
\begin{figure}[ht]
\centering
\begin{subfigure}{.33\textwidth}
  \centering
	\includegraphics[width=1\linewidth]{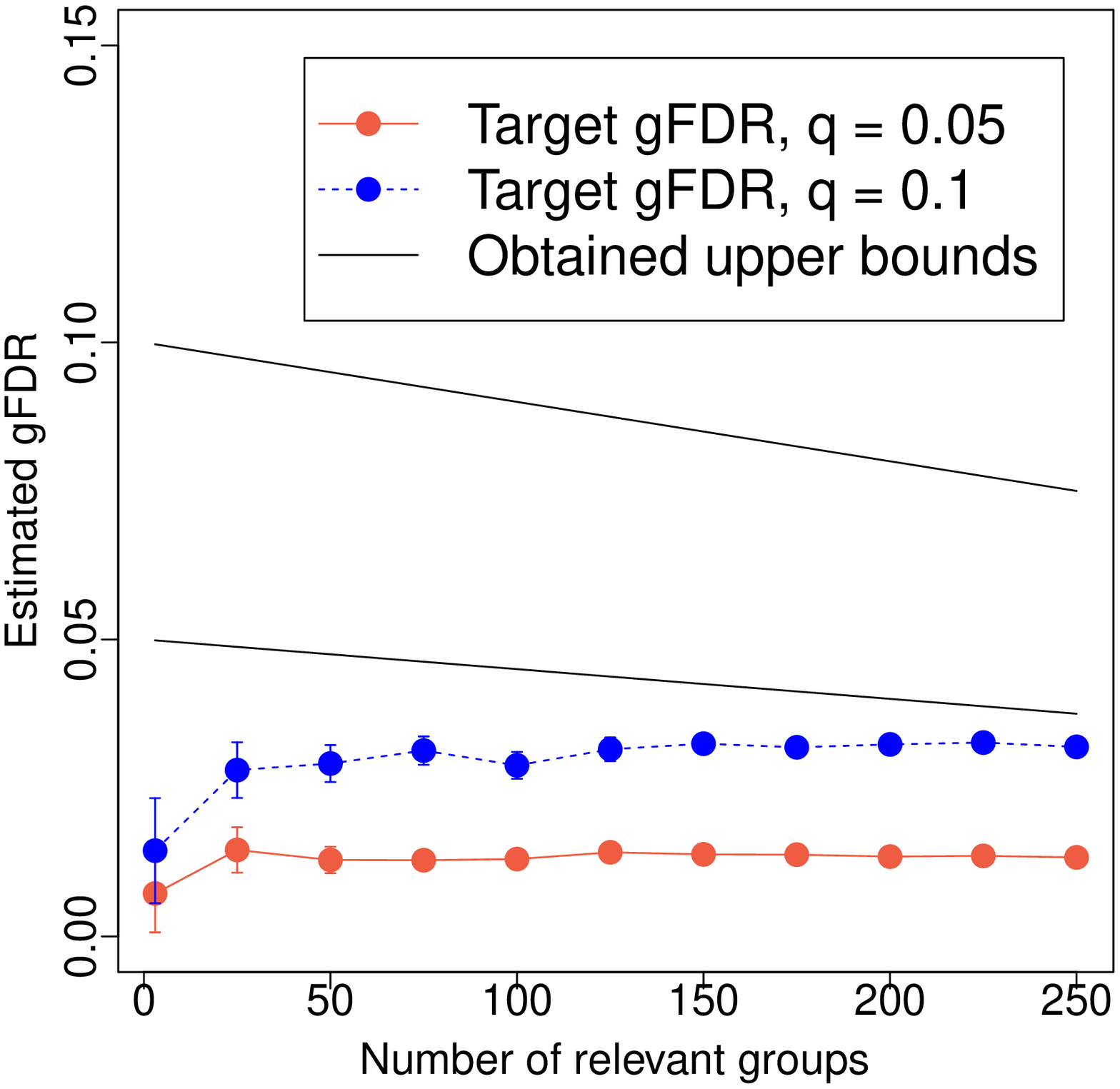}
  \caption{gFDR for $\lambda^{max}$}
\end{subfigure}%
\begin{subfigure}{.33\textwidth}
  \centering
  \includegraphics[width=1\linewidth]{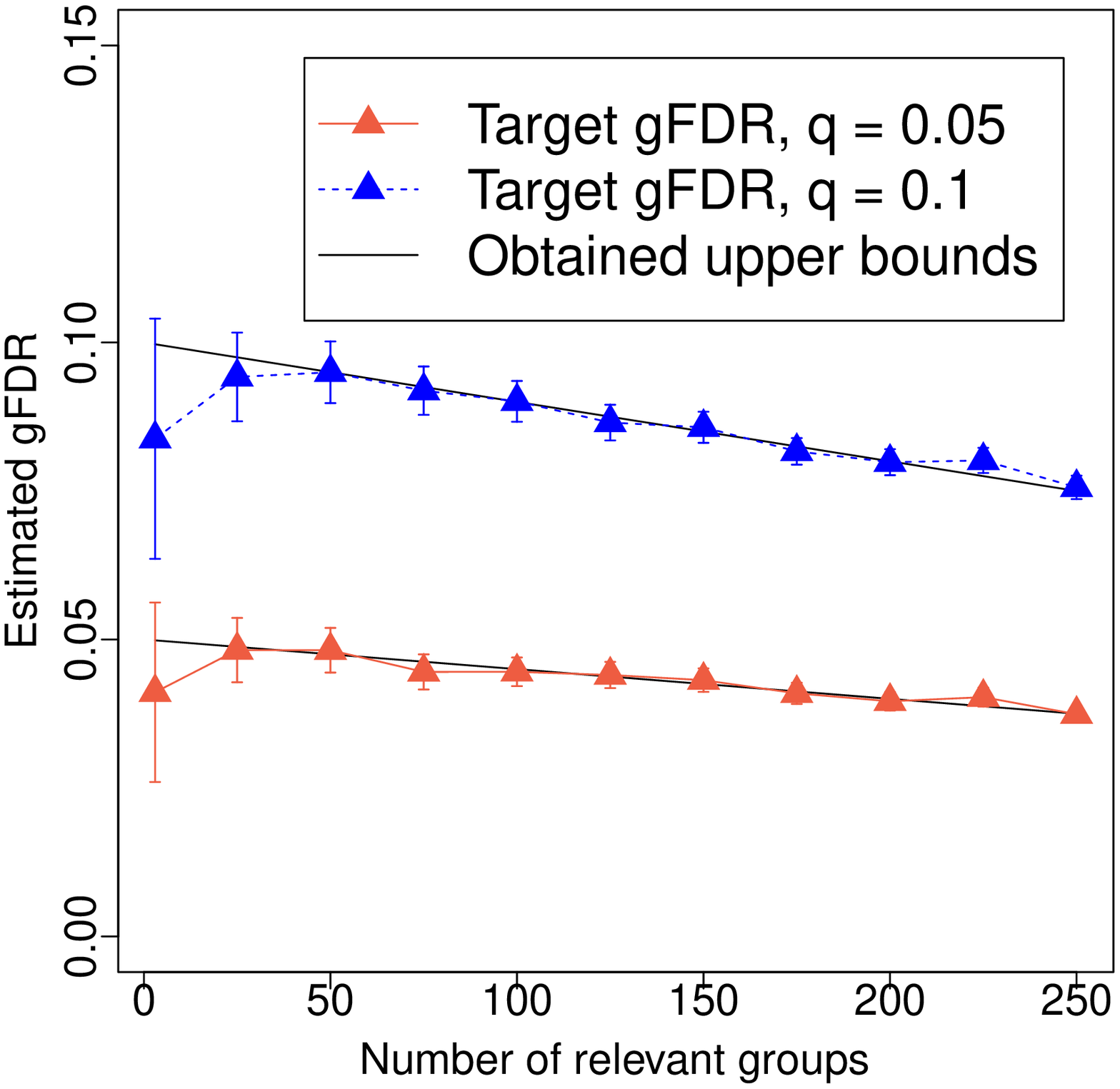}
  \caption{gFDR for $\lambda^{mean}$}
\end{subfigure}%
\begin{subfigure}{.33\textwidth}
  \centering
  \includegraphics[width=1\linewidth]{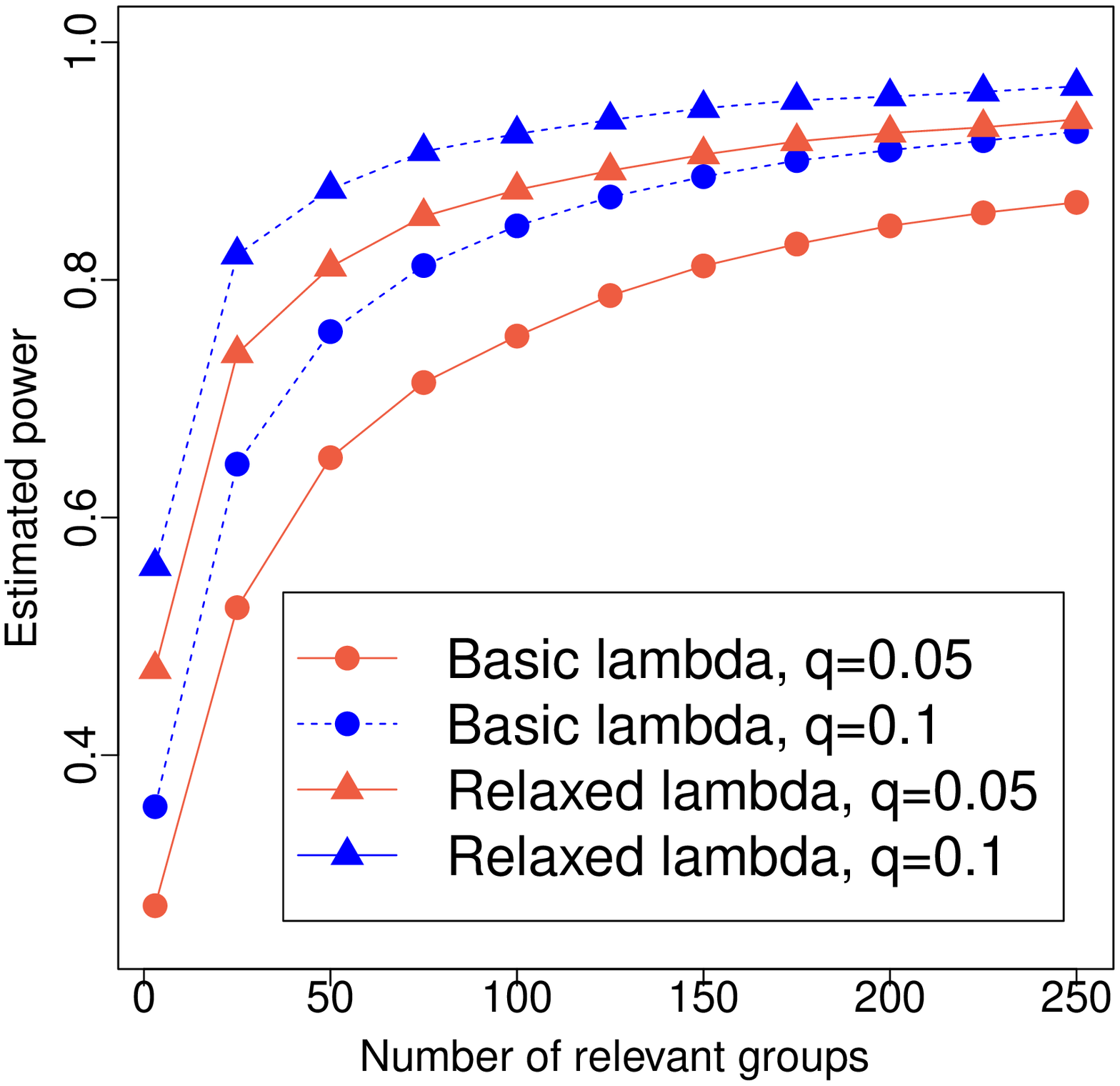}
  \caption{Power}
\end{subfigure}%
\caption{Orthogonal situation with various sizes of groups. For each target gFDR level and true support size, $300$ iterations were performed, bars correspond to $\pm 2$SE. Here, black straight lines are given by $q\cdot\big((m-k)/m\big)$, for $k$ being true support size. Weights were defined as $w_i:=\sqrt{l_i}$.}
\label{19091846}
\end{figure}
It could be observed, that the selected tuning parameters are rather conservative, i.e. the achieved gFDR is significantly lower than assumed. This suggests, that penalties (dictated by lambdas) could be slightly decreased, such as the method gets more power and still achieves the gFDR below the assumed level. Returning to the proof of Theorem \ref{gFDRcontrol}, we can see that the crucial property of sequence $\{\lambda_i\}_{r=1}^m$ is that  $1-F_{\Chi_{l_i}}\left(\lambda_rw_i\right)\leq\frac{qr}m$ for each $i$. The possible relaxation of this condition is to assume only that
\begin{equation}
\label{08121449}
\sum_{i=1}^m\Big(1-F_{w^{-1}_i\Chi_{l_i}}(\lambda_r)\Big)\leq qr.
\end{equation}
Under equal weights and equal $l_i$'s assumption, the above conditions are equivalent. In general case, however, the second approach (with the inequality replaced by equality) produces smaller penalties compared to tuning parameters given by Theorem \ref{gFDRcontrol}. Most often, such a change results in improving power (and increasing gFDR at the same time). Replacing the inequality in (\ref{08121449}) by equality yields the following strategy of choosing relaxed $\lambda$ sequence (denoted by $\lambda^{mean}$)  
\begin{equation}
\label{08191453}
\lambda^{mean}_r:= \overline{F}^{-1}\left(1-\frac{qr}m\right)\quad\textrm{for}\quad \overline{F}(x): = \frac1m\sum_{i=1}^mF_{w^{-1}_i\Chi_{l_i}}(x),\quad r\in\{1,\ldots,m\},
\end{equation}
where $F_{w^{-1}_i\Chi_{l_i}}$ is the cumulative distribution function of scaled chi distribution with $l_i$ degrees of freedom and scale $\mathcal{S}=w_i^{-1}$. In Figure \ref{19091846}b we present estimated gFDR, for tuning parameters given by (\ref{08191453}). The results suggest that with relaxed version of tuning parameters, we can still achieve the gFDR control while essentially improving the power of gSLOPE. Such a strategy could be especially important in situation, when differences between the smallest and the largest quantiles (among distributions $w_i^{-1}\Chi_{l_i}$) are relatively large. When this is the case, the gSLOPE with lambdas given by Theorem \ref{gFDRcontrol} in orthogonal situation could be considered as too conservative. 

Up until this point, we have only considered the testing properties of gSLOPE. Though originally proposed to control the FDR, surprisingly, SLOPE enjoys appealing estimation properties as well \cite{WE}. It thus would be desirable to extend this link between testing and estimation for gSLOPE. In measuring the deviation of an estimator from the ground truth $\beta$, as earlier, we focus on the group level instead of an individual. Accordingly, here we aim to estimate $\llbracket \beta\rrbracket_{X,I}: = \big(\|X_{I_1} \beta_{I_1}\|_2, \ldots, \|X_{I_m} \beta_{I_m}\|_2\big)^\mathsf{T} $ or $\iI{\widetilde{\beta}}: = \big(\|\widetilde{\beta}_{\II_1}\|_2, \ldots, \|\widetilde{\beta}_{\II_m}\|_2\big)^\mathsf{T} $, equivalently. For illustration purpose, we employ the setting described as follows. Imagine that we have a sequences of problems with the number of groups $m$ growing to infinity: the design $X$ is orthonormal at groups level; ranks of submatrices $X_{I_i}$, $l_i$, are bounded, that is, $\max l_i \le l$ for some constant integer $l$; denoting by $k \ge 1$ the sparsity level (that is, the number of relevant groups), we assume the asymptotic $k/m \rightarrow 0$. Now we state our minimax theorem, where we write $a \sim b$ if $a/b \rightarrow 1$ in the asymptotic limit, and $\|\XI{\beta}\|_0$ denotes the number of nonzero entries of $\XI{\beta}$. The proof makes use of the same techniques for proving Theorem 1.1 in \cite{WE} and is deferred to the Appendix.
\begin{theorem}\label{minimax}
Fix any constant $q \in (0, 1)$, let $w_i = 1$ and $\lambda_i = F_{\chi_l}^{-1}(1 - qi/m)$ for $i=1,\ldots,m$. Under the preceding conditions, gSLOPE is asymptotically minimax over the nearly black object $\big\{\beta: \big\|\XI{\beta}\big\|_0 \le k\big\}$, i.e.,
\[
\sup_{\|\XI{\beta}\|_0 \le k} \mathbb{E} \left( \Big\| \XI{\beta^\ES{gS}} - \XI{\beta} \Big\|_2^2\right) \sim  \inf_{\widehat\beta}\sup_{\|\XI{\beta}\|_0 \le k} \mathbb{E} \left( \Big\| \XI{\widehat\beta} - \XI{\beta} \Big\|_2^2\right),
\] 
where the infimum is taken over all measurable estimators $\widehat\beta(y, X)$.
\end{theorem}
Notably, in this theorem the choice of $\lambda_i$ does not assume the knowledge of sparsity level. Or putting it differently, in stark contrast to gLASSO, gSLOPE is adaptive to a range of sparsity in achieving the exact minimaxity. Combining Theorems~\ref{gFDRcontrol} and \ref{minimax}, we see the remarkable link between FDR control and minimax estimation also applies to gSLOPE \cite{ABDJ, WE}. While it is out of the scope of this paper, it is of great interest to extend this minimax result to general design matrices.

\subsection{The impact of chosen weights}
In this subsection we will discuss the influence of chosen weights, $\{w_i\}_{i=1}^m$, on results. Let $I=\{I_1,\ldots,I_m\}$ be given division into groups and $l_1,\ldots,l_m$ be ranks of submatrices $X_{I_i}$. Assume the orthogonality at groups level, i.e. that it holds $X_{I_i}^\mathsf{T}X_{I_j}=0$, for $i\neq j$, and suppose that $\sigma=1$. The support of $\XI{\beta}$ coincides with the support of vector $c^*$ defined in \eqref{17022353}, namely 
\begin{equation}
c^* = \argmin{c}\ \frac12\Big\|\iI{\y}-W^{-1}c\Big\|_2^2+J_{\lambda}(c),
\end{equation} 
where $W^{-1}$ is diagonal matrix with positive numbers $w_1^{-1},\ldots,w_m^{-1}$ on diagonal. Suppose now, that $c^*$ has exactly $r$ nonzero coefficients. From Corollary \ref{06021230}, these indices are given by $\{\pi(1),\ldots,\pi(r)\}$, where $\pi$ is permutation which orders $W^{-1}\iI{\y}$. Hence, the order of realizations $\big\{w_i^{-1}\|\widetilde{y}_{\II_i}\|_2\big\}_{i=1}^m$ decides about the subset of groups labeled by gSLOPE as relevant. Suppose that groups $I_i$ and $I_j$ are truly relevant, i.e $\|\Beta_{\II_i}\|_2>0$ and $\|\Beta_{\II_j}\|_2>0$. If we want to achieve the situation in which subset of truly discovered groups is not significantly affected by $l_i$, we should choose weights such as $w_i^{-1}\|\y_{\II_i}\|_2$ and $w_j^{-1}\|\y_{\II_j}\|_2$ are ''comparable''. One sensible strategy is to look at this issue from the side of expected values. The distributions of $\|\y_{\II_i}\|_2$ and $\|\y_{\II_j}\|_2$ are noncentral $\Chi$ distributions, with $l_i$ and $l_j$ degrees of freedom, and the noncentrality parameters equal to $\|\Beta_{\II_i}\|_2$ and $\|\Beta_{\II_j}\|_2$, respectively. Now, the expected value of the noncentral $\Chi$ distribution could be well approximated by the square root of the expected value of the noncentral $\Chi^2$ distribution, which gives
$$\mathbb{E}\big(w_i^{-1}\|\y_{\II_i}\|_2\big)\approx w_i^{-1}\sqrt{\mathbb{E}\big(\|\y_{\II_i}\|^2_2\big)}=w_i^{-1}\sqrt{l_i+\|\Beta_{\II_i}\|_2^2}.$$
Therefore, roughly speaking, truly relevant groups $I_i$ and $I_j$ are treated as comparable, when it occurs ${l_i/w_i^2+\|\Beta_{\II_i}\|_2^2/w_i^2\approx l_j/w_j^2+\|\Beta_{\II_j}\|_2^2/w_j^2}$. This gives us the intuition about the behavior of gSLOPE with the choice $w_i=\sqrt{l_i}$ for each $i$: gSLOPE treats two truly relevant groups as comparable, if groups effect sizes satisfy the condition $\big(\XI{\beta}\big)_i/\big(\XI{\beta}\big)_j\approx\sqrt{l_i}/\sqrt{l_j}$. The derived condition could be recast as $\|X_{I_i}\beta_{I_i}\|_2^2/l_i \approx \|X_{I_j}\beta_{I_j}\|_2^2/l_j$. This gives a nice interpretation: with the choice $w_i:=\sqrt{l_i}$, under orthogonality at groups level and with linear independence of columns inside groups, gSLOPE treats two groups as comparable, when these groups have similar squared effect group sizes per coefficient. One possible idealistic situation, when such a property occurs, is when all truly relevant variables have the same impact on the response and were divided into groups containing either all truly relevant or all truly irrelavant regressors.

In Figure \ref{21090909} (results for previously described simulations performed for orthogonal situation and various groups sizes) we see that when the condition $\big(\XI{\beta}\big)_i/\big(\XI{\beta}\big)_j=\sqrt{l_i}/\sqrt{l_j}$ is met, the fractions of groups with different sizes in the selected truly relevant groups (STRG) are approximately equal. To investigate the impact of selected weights on the set of discovered groups, we performed simulations with different settings, namely we used $w_i=1$ and $w_i = l_i$ (without changing other parameters). With the first choice, larger groups are penalized less than before, while the second choice yields the opposite situation. This is reflected in the proportion of each groups in STRG (Figure \ref{21090909}). 
\begin{figure}[ht]
\centering
\begin{subfigure}{.33\textwidth}
  \centering
	\includegraphics[width=1\linewidth]{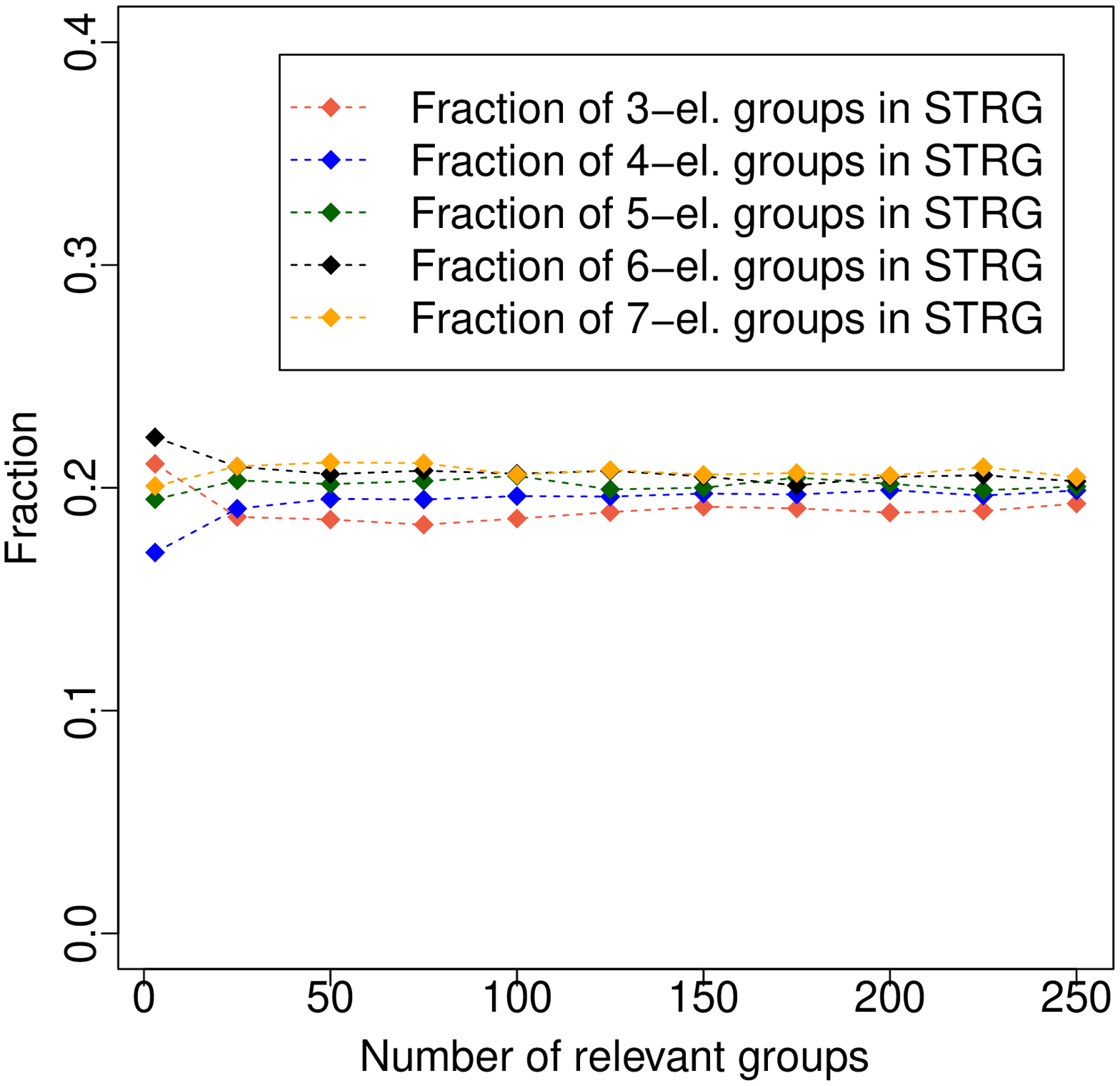}
  \caption{Structure of STRG, $w_i:=\sqrt{l_i}$}
\end{subfigure}%
\begin{subfigure}{.33\textwidth}
  \centering
  \includegraphics[width=1\linewidth]{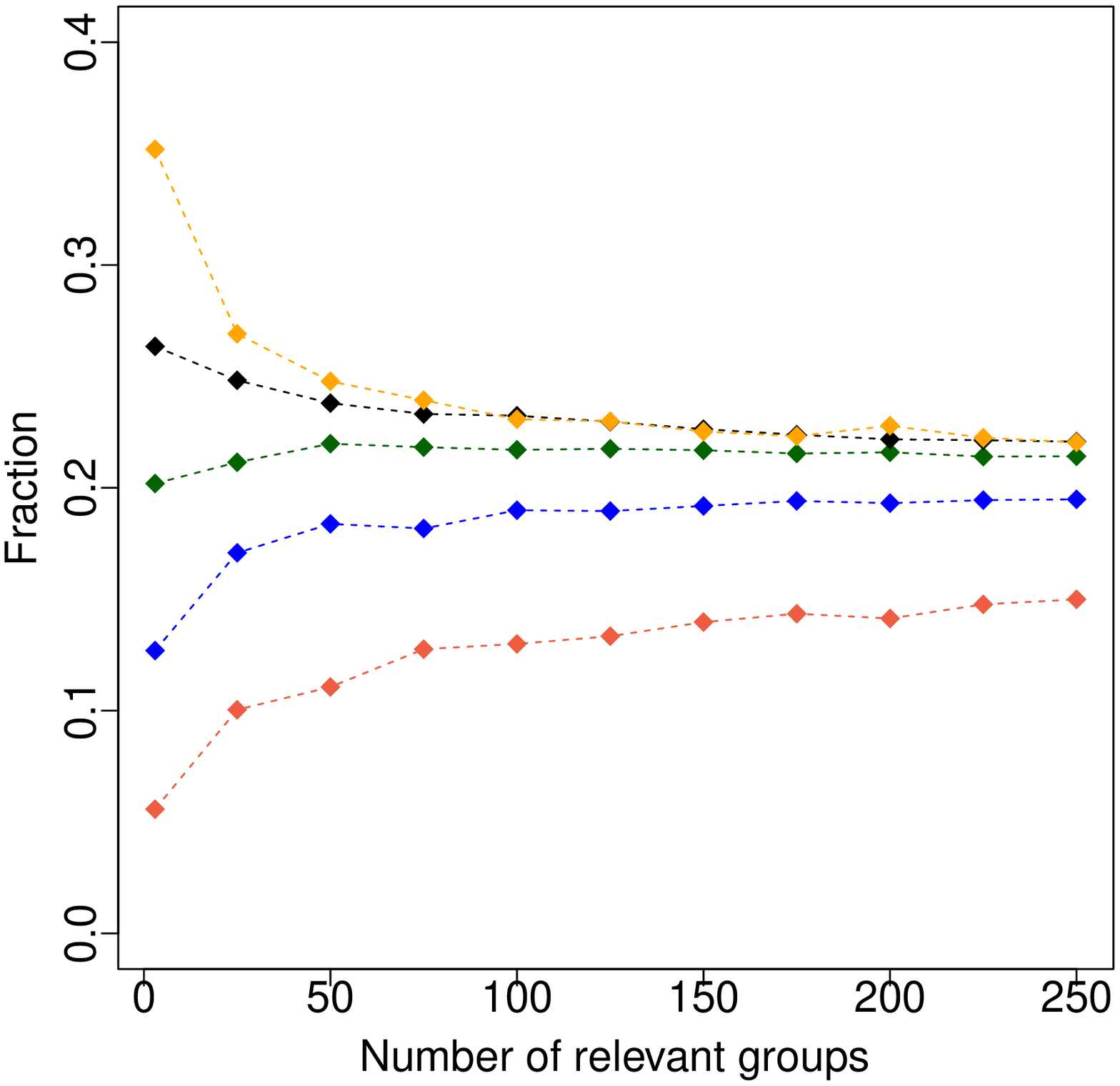}
  \caption{Structure of STRG, $w_i:=1$}
\end{subfigure}%
\begin{subfigure}{.33\textwidth}
  \centering
  \includegraphics[width=1\linewidth]{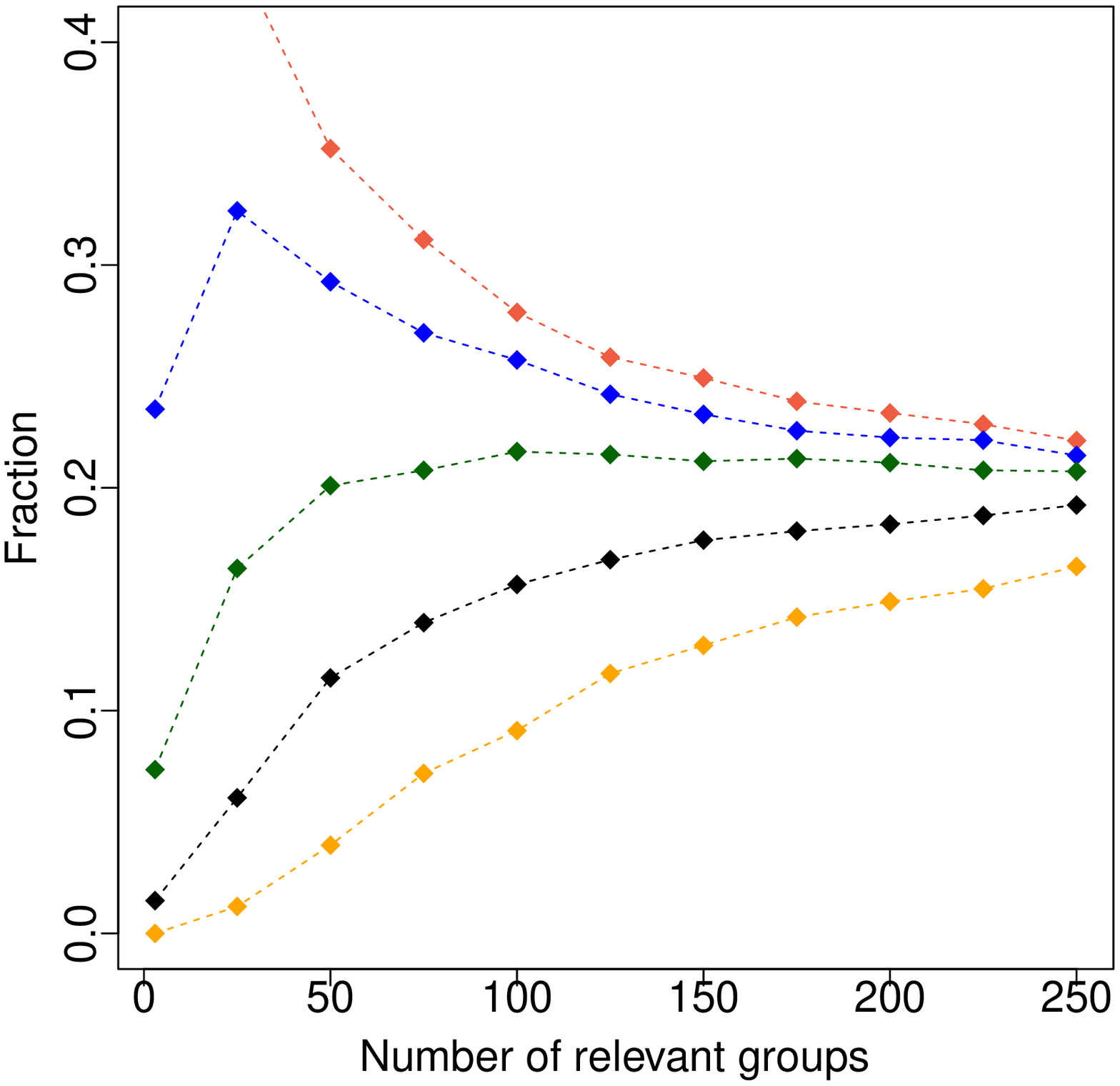}
  \caption{Structure of STRG, $w_i:=l_i$}
\end{subfigure}%
\caption{Fraction of each group sizes in STRG. Beyond the weights, this simulation was conducted with the same setting as in experiments summarized in Figure 3 for $\lambda^{mean}$. In particular, for truly relevant groups $i$ and $j$, it occurs $\big(\XI{\beta}\big)_i/\big(\XI{\beta}\big)_j=\sqrt{l_i}/\sqrt{l_j}$.}
\label{21090909}
\end{figure}
The values of gFDR are very similar under all choices of weights. Consequently, we are equipped in the entire family of settings, according to the rule: fix weights to induce the comparability in the domain of group effect strengths for different group sizes, and select lambdas to control gFDR for given weights. 

\subsection{Near-orthogonal situation}
In this section we will deal with the case in which the columns in the design matrix are only ''almost orthogonal'', which could be valid for real-world applications. To simulate such a situation we will assume that $n$ by $p$ design matrix is a realization of the random matrix with independent entries drawn from the normal distribution, $\mathcal{N}\big(0,\frac 1n\big)$, so as the expected value of $X_i\T{T}X_j$ is equal to $1$ for $i = j$, and equal to $0$ otherwise. The main objective is to derive the lambda sequence, which could be applied to achieve gFDR control under assumption that the $\XI{\beta}$ is sparse. In this subsection we will confine ourselves only to the case $l_1=\ldots=l_m:=l$, $w_1=\ldots=w_m:=w$ and when the number of elements in each group is relatively small as compared to the number of observations ($l<<n$). For simplicity in this subsection we will fix $\sigma = 1$. In case when $\sigma \neq 1$, the proposed sequence lambda should be multiplied by $\sigma$, as in expression (\ref{gSLOPE}).In the heuristic presented in this subsection, we will use the notation $A \approx B$, in order to express that with large probability the differences between corresponding entries of matrices $A$ and $B$ are very small. 

In situation when entries of $X$ comes from $\mathcal{N}\big(0,\frac 1n\big)$ distribution and sizes of groups are relatively small, very good approximation of $\beta^\ES{gS}$ could be obtained by $\hat{\beta}$, defined as
\begin{equation}
\label{gSLOPE_1}
\hat{\beta}: = \argmin b\ \ \bigg\{\frac 12\Big\|y-Xb\Big\|_2^2+\sigma J_{\lambda}\Big( W\I{b}\Big)\bigg\}.
\end{equation}
Assume for simplicity that $\|\beta_{I_1}\|_2>\ldots> \|\beta_{I_s}\|_2 >0$, $\|\beta_{I_j}\|_2=0$ for $j>s$, $\hat{\beta}$ satisfies the same conditions for some $\lambda$ (this implies in particular, that true signals are relatively strong) and the true model is sparse. Divide $I$ into two families of sets $I^s:=\{I_1,\ldots,I_s\}$ and $I^c:=\{I_{s+1},\ldots,I_m\}$. To derive optimality condition for $\hat{\beta}$ we will prove the following
\begin{theorem}
\label{27111155}
Let $b\in\mathbb{R}^p$ be such that $\|b_{I_1}\|_2>\ldots>\|b_{I_s}\|_2>0$, $\|b_{I_j}\|_2=0$ for $j>s$ and denote $\lambda^c:=(\lambda_{s+1},\ldots,\lambda_m)\T{T}$. If $g\in \partial J_{\lambda}\big(w\I{b}\big)$, then it holds:
\begin{equation}
\left\{
\begin{array}{l}
g_{I_i}=w\lambda_i\frac{b_{I_i}}{\|b_{I_i}\|_2},\ i=1,\ldots, s\\
\llbracket g\rrbracket_{I^c}\in C_{w\lambda^c}
\end{array}
\right.,
\end{equation}
where the set $C_{\lambda}$ (here with $w\lambda^c$ instead of $\lambda$) is defined in appendix (\ref{subsec:app26112110}).
\end{theorem}
\begin{proof}
For $b\in \mathbb{R}^p$ define $J_{\lambda,I}(b):=J_{\lambda}\big(\I{b}\big)$ and put
$H:=\big\{h\in\mathbb{R}^p:\ \|(b+h)_{I_1}\|_2>\ldots>\|(b+h)_{I_s}\|_2,\ \|(b+h)_{I_s}\|_2>\|h_{I_j}\|_2,\ j>s\big\}.$
If $g\in\partial J_{\lambda,I}(b)$, then for all $h\in H$ from definition of subgradient it holds
\begin{equation}
\label{06191517}
\sum_{i=1}^s\lambda_i\|(b+h)_{I_i}\|_2+\sum_{i=s+1}^m\lambda_i\big(\I{b+h}\big)_{(i)}\geq\sum_{i=1}^s\lambda_i\|b_{I_i}\|_2 + \sum_{i=1}^sg_{I_i}\T{T}h_{I_i}+(g^c)\T{T}h^c,
\end{equation}
for $g^c:=(g_{I_{s+1}}\T{T},\ldots,g_{I_m}\T{T})\T{T}$ and $h^c:=(h_{I_{s+1}}\T{T},\ldots,h_{I_m}\T{T})\T{T}$. Define $\widetilde{I}:=\big\{\widetilde{I}_1,\ldots, \widetilde{I}_{m-s}\big\}$, with set $\widetilde{I}_i:=\big\{(i-1)\cdot l +1,\ldots, i\cdot l\big\}$. Then $\llbracket g^c\rrbracket_{\widetilde{I}}= \llbracket g\rrbracket_{I^c}$. Consider first case, when $h$ belongs to the set $H^c:=\{h\in H:\ h_{I_i}\equiv0,\ i\leq s\}$. This yields
\begin{equation}
\label{06191400}
\sum_{i=1}^{m-s}\lambda_{s+i}\big(\llbracket h^c\rrbracket_{\widetilde{I}}\big)_{(i)}\geq (g^c)\T{T}h^c.
\end{equation}
Since $\{h^c:\ h\in H^c\}$ is open in $\mathbb{R}^{l(m-s)}$ and contains zero, from Corollary \ref{06182058} we have that $g^c\in \partial J_{\lambda^c,\widetilde{I}}(0)$ and the inequality (\ref{06191400}) is true for any $h^c\in\mathbb{R}^{l(m-s)}$ yielding  
\begin{equation}
0\geq \sup_{h^c}\Big\{(g^c)\T{T}h^c -J_{\lambda^c,\widetilde{I}}(h^c)\Big\} = J_{\lambda^c,\widetilde{I}}^*(g^c) = \left\{\begin{array}{cl}
0,&\llbracket g^c\rrbracket_{\widetilde{I}}\in C_{\lambda^c}\\
\infty,&\textrm{otherwise}
\end{array}\right.,
\end{equation}
see Proposition \ref{07071055}. This result immediately gives condition $\llbracket g^c\rrbracket_{\widetilde{I}}\in C_{\lambda^c}$, which is equivalent with $\llbracket g\rrbracket_{I^c}\in C_{\lambda^c}$. To find conditions for $g_{I_i}$ with $i\leq s$, define sets $H_i:=\{h\in H:\ h_{I_j}\equiv0,\ j\neq i\}$. For $h\in H_i$, (\ref{06191517}) reduces to $\lambda_i\|b_{I_i}+h_{I_i}\|_2\geq \lambda_i\|b_{I_i}\|_2+g_{I_i}\T{T}h_{I_i}$. Since the set $\{h_{I_i}:\ h\in H_i\}$ is open in $\mathbb{R}^l$ and contains zero, from Corollary \ref{06182058} we have $g_{I_i}\in \partial f_i(b_{I_i})$ for $f_i:\mathbb{R}^l\longrightarrow\mathbb{R}$, $f_i(x):=\lambda_i\|x\|_2$. Since $f_i$ is convex and differentiable in $b_{I_i}$, it holds $g_{I_i}=\lambda_i\frac{b_{I_i}}{\|b_{I_i}\|_2}$, which finishes the proof.
\end{proof}
\noindent The above theorem allows to write the optimality condition for $\hat{\beta}$ in form 
\begin{equation}
\label{24031126}
\left\{
\begin{array}{l}
X_{I_i}\T{T}(y-X\hat{\beta})=w\lambda_i\frac{\hat{\beta}_{I_i}}{\|\hat{\beta}_{I_i}\|_2},\ i=1,\ldots, s\\
\llbracket X\T{T}(y-X\hat{\beta})\rrbracket_{I^c}\in C_{w\lambda^c}
\end{array}
\right..
\end{equation}
Since $X_{I_i}\T{T}X_{I_i}\approx\mathbf{I}_l$, for $i\leq s$ we get $X_{I_i}\T{T}\left(y - X_{\backslash I_i}\hat{\beta}_{\backslash I_i}\right) \approx \hat{\beta}_{I_i}\left(1+\frac{w\lambda_i}{\|\hat{\beta}_{I_i}\|_2}\right)$, where $X_{\backslash I_i}$ is matrix $X$ without columns from $I_i$ and $\hat{\beta}_{\backslash I_i}$ denotes vector $\hat{\beta}$ with removed coefficients indexed by $I_i$. This means that, for $i=1,\ldots, s$, vector $v_{I_i}:=X_{I_i}\T{T}\left(y - X_{\backslash I_i}\hat{\beta}_{\backslash I_i}\right)$ is approximately collinear with $\hat{\beta}_{I_i}$. Since $1+\frac{w\lambda_i}{\|\hat{\beta}_{I_i}\|_2}>0$, we have $\frac{v_{I_i}}{\|v_{I_i}\|_2} \approx \frac{\hat{\beta}_{I_i}}{\|\hat{\beta}_{I_i}\|_2}$.
This yields $\hat{\beta}_{I_i}\approx\left(1-\frac{w\lambda_i}{\|v_{I_i}\|_2}\right)v_{I_i}$ and consequently $\|\hat{\beta}_{I_i}\|_2\approx\Big|\|v_{I_i}\|_2-w\lambda_i\Big|$. Therefore (\ref{24031126}) can be written as 
\begin{equation}
\label{23061949}
\left\{
\begin{array}{l}
\Big|\|v_{I_i}\|_2-w\lambda_i\Big|\approx\|\hat{\beta_{I_i}}\|_2,\ i=1,\ldots, s\\
\llbracket v \rrbracket_{I^c}\in C_{w\lambda^c}
\end{array}
\right.,
\end{equation}
for $v:=(v_{I_1}\T{T},\ldots,v_{I_m}\T{T})\T{T}$.

The task now is to select $\lambda_i$'s such that condition $\llbracket v\rrbracket_{I^c}\in C_{w\lambda^c}$ regulates the rate of false discoveries. Denote $I_S:=\bigcup_{i=1}^sI_i$. Putting $y=X_{I_S}\beta_{I_S}+z$, we obtain
\begin{equation}
\label{24061549}
v_{I_i} = X_{I_i}\T{T}X_{I_S}(\beta_{I_S}- \hat{\beta}_{I_S}) + X_{I_i}\T{T}z,
\end{equation}
for $i>s$ (irrelevant groups). Under orthogonal design this expression reduces only to the term $X_{I_i}\T{T}z$, and in such situation $\|v_{I_i}\|_2$ has $\Chi$ distribution with $l$ degrees of freedom which was used in subsection \ref{subsec:06232149} to define the sequence $\lambda$. In the considered near-orthogonal situation, the term $X_{I_i}\T{T}X_{I_S}(\beta_{I_S}- \hat{\beta}_{I_S})$ should be also taken into account. Two following assumptions will be important to derive the appropriate approximation of $v_{I_i}$ distribution:
\begin{itemize}
\setlength\itemsep{1pt}
\item the distribution of $v_{I_i}$ could be well approximated by multivariate normal distribution, 
\item for relatively strong effects it occurs $\frac{\hat{\beta}_{I_i}}{\|\hat{\beta}_{I_i}\|_2}\approx\frac{\beta_{I_i}}{\|\beta_{I_i}\|_2}$ for $i=1,\ldots,s$.
\end{itemize}
The first assumption is justified when one works with large data scenario, based on the Central Limit Theorem. In discussion concerning the second assumption it is important to clarify the effect of penalty imposed on entire groups. The magnitudes of coefficients in $\hat{\beta}_{I_i}$, for truly relevant group $i$, are generally significantly smaller than in $\beta_{I_i}$. This, a so-called shrinking effect, is typical for penalized methods. It turns out, however, that under assumed conditions estimates of coefficients of nonzero $\beta_{I_i}$ are pulled to zero proportionally and after normalizing, $\hat{\beta}_{I_i}$ and $\beta_{I_i}$ are comparable. 

From the upper equation in (\ref{24031126}), we have that $X_{I_S}\T{T}(X_{I_S}\beta_{I_S}-X_{I_S}\hat{\beta}_{I_S}) + X_{I_S}\T{T}z \approx wH_{\lambda,\beta}$, for 
\begin{equation}
\label{27091918}
H_{\lambda,\beta}: = \Big(\lambda_1\frac{\beta_{I_1}\T{T}}{\|\beta_{I_1}\|_2},\ldots, \lambda_s\frac{\beta_{I_s}\T{T}}{\|\beta_{I_s}\|_2}\Big)\T{T}, 
\end{equation}
which gives $X_{I_i}\T{T}X_{I_S}(\beta_{I_S}-\hat{\beta}_{I_S}) \approx X_{I_i}\T{T}X_{I_S}(X_{I_S}\T{T}X_{I_S})^{-1}(wH_{\lambda,\beta} - X_{I_S}\T{T}z)$. Combining the last expression with (\ref{24061549}) yields
\begin{equation}
\label{03251704}
v_{I_i}\approx  X_{I_i}\T{T}X_{I_S}(X_{I_S}\T{T}X_{I_S})^{-1}\Big(wH_{\lambda,\beta}-X_{I_S}\T{T}z\Big)+X_{I_i}\T{T}z.
\end{equation}
To determine the parameters of multivariate normal distribution, which best describes the distribution of $v_{I_i}$, we will derive the exact values of the mean and the covariance matrix of the distribution of the right-hand side expression in (\ref{03251704}) for $i>s$. Since $I_i\cap I_S=\emptyset$ and entries of $X$ matrix are randomized independently with $\mathcal{N}\big(0,\frac 1n\big)$ distribution, the expected value of the random vector in (\ref{03251704}) is $0$ and its covariance matrix is provided in the following theorem.
\begin{theorem}
\label{06241642}
The covariance matrix of $\hat{v}_{I_i}:=X_{I_i}\T{T}X_{I_S}(X_{I_S}\T{T}X_{I_S})^{-1}\Big(wH_{\lambda,\beta}-X_{I_S}\T{T}z\Big)+X_{I_i}\T{T}z$, for $i>s$, is given by the formula
$$Cov(\hat{v}_{I_i}) = \left(\frac{n-ls}n+w^2\frac{\|\lambda^S\|^2_2}{n-ls-1}\right)\mathbf{I}_l,$$ where $\lambda^S:=(\lambda_1,\ldots, \lambda_s)\T{T}.$
\end{theorem}
\noindent Before proving Theorem \ref{06241642}, we will introduce two lemmas, proofs of which can be found in the Appendix.
\begin{lemma}
\label{25031835}
Suppose that entries of a random matrix $X\in M(n,r)$, with $r\leq n$, are independently and identically distributed and have a normal distribution with zero mean. Then, there exists the expected value of a random matrix $A_X = X(X\T{T}X)^{-1}X\T{T}$ and $\mathbb{E}\left(A_X\right) = \frac rn \mathbf{I}_n$.
\end{lemma}
\begin{lemma}
\label{31032152}
Suppose that $X\in M(n,r)$, with $r+1<n$, and entries of $X$ are independent and identically distributed, $X_{ij}\sim\mathcal{N}(0,1/n)$ for all $i$ and $j$. Then, there exists expected value of random matrix, $M_{X,\lambda}: = B_XH_{\lambda,\beta}H_{\lambda,\beta}\T{T}B_X\T{T}$, for $B_X=X(X\T{T}X)^{-1}$ and $H_{\lambda,\beta}$ defined in \eqref{27091918}. Moreover, it holds $\mathbb{E}\left(M_X\right) = \frac {\|\lambda_S\|_2^2}{n-r-1}\, \mathbf{I}_n$.
\end{lemma}
\begin{proof}[Proof of Theorem \ref{06241642}]
We have $\hat{v}_{I_i} =\xi_{X,z} + \zeta_X,$ for $\xi_{X,z}: = X_{I_i}\T{T}\left(\mathbf{I}_n-A_X\right)z$, $\zeta_X:= wX_{I_i}\T{T}B_XH_{\lambda,\beta}$, $A_X: = X_{I_S}(X_{I_S}\T{T}X_{I_S})^{-1}X_{I_S}\T{T}$, $B_X: = X_{I_S}(X_{I_S}\T{T}X_{I_S})^{-1}$. Since $\mathbb{E}(\xi_{X,z}\zeta_X\T{T}) = 0$ and mean of $\hat{v}_{I_i}$ is equal to $0$, it holds $Cov(\hat{v}_{I_i})=Cov(\xi_{X,z})+Cov(\zeta_X)$. Now thanks to Lemma \ref{25031835} and Lemma \ref{31032152}
\begin{equation}
\begin{split}
Cov(\xi_{X,z}) = \ &\mathbb{E}\left[X_{I_i}\T{T}\big(\mathbf{I}_n-A_X\big)zz\T{T}\big(\mathbf{I}_n-A_X\big)\T{T}X_{I_i}\right]=\\
&\mathbb{E}\left[X_{I_i}\T{T}\big(\mathbf{I}_n-A_X\big)\big(\mathbf{I}_n-A_X\big)\T{T}X_{I_i}\right]= \mathbb{E}\left[X_{I_i}\T{T}\big(\mathbf{I}_n-A_X\big)X_{I_i}\right] =\\
&\frac1n\big(n-ls\big)\cdot\mathbb{E}\left[X_{I_i}\T{T}X_{I_i}\right]=\frac1n\big(n-ls\big)\cdot\mathbf{I}_l,
\end{split}
\end{equation}
\begin{equation}
\begin{split}
Cov(\zeta_X) = \ & w^2\mathbb{E}\left[X_{I_i}\T{T}B_XH_{\lambda,\beta}H_{\lambda,\beta}\T{T}B_X\T{T}X_{I_i}\right] = w^2\frac{\|\lambda^S\|_2^2}{n-sl-1}\mathbb{E}\left[X_{I_i}\T{T}X_{I_i}\right] =\\
& w^2\frac{\|\lambda^S\|_2^2}{n-sl-1}\mathbf{I}_l,
\end{split}
\end{equation}
which finishes the proof.
\end{proof}
We have shown that for $i>s$ the distribution of $\|v_{I_i}\|_2$ could be approximated by scaled $\Chi$ distribution with $l$ degrees of freedom and scale parameter $\mathcal{S} = \sqrt{\frac{n-ls}n+\frac{w^2\|\lambda_S\|^2_2}{n-sl-1}}$. Now, analogously to the orthogonal situation, lambdas could be defined as $\lambda_i:=\frac{1}{w_i}F_{\mathcal{S}\Chi_l}^{-1}\left (1-\frac{q\cdot i}m\right )= \frac{\mathcal{S}}{w_i}F_{\Chi_l}^{-1}\left (1-\frac{q\cdot i}m\right )$. Since $s$ is unknown, we will apply the strategy used in $\cite{SLOPE}$: define $\lambda_1$ as in orthogonal case and for $j\geq 2$ define $\lambda_i$ basing on already generated sequence, according to following procedure.
\begin{algorithm}[H]
{\fontsize{9pt}{5pt}\selectfont
  \caption{Selecting lambdas in near-orthogonal situation: equal groups sizes}
	\label{24021313}
  \begin{algorithmic}
		\State \textbf{input:} $q\in (0,1)$,\ \ $w>0$,\ \ $p,\ n,\ m,\ l \in\mathbb{N}$
		\State $\lambda_1:=\frac1wF_{\chi_{\mathnormal{l}}}^{-1}\left (1-\frac{q}m\right )$;
		\State  \textbf{For} $i\in\{2,\ldots,m\}$:
		\State \indent  $\lambda^S: = (\lambda_1,\ldots, \lambda_{i-1})^\mathsf{T}$;
		\State \indent  $\mathcal{S}: = \sqrt{\frac{n-l(i-1)}n+\frac{w^2\|\lambda^S\|^2_2}{n-l(i-1)-1}}$;
		\State \indent  $\lambda^*_i:=\frac{\mathcal{S}}{w_i}F_{\chi_l}^{-1}\left (1-\frac{q\cdot i}m\right )$;
		\State \indent  if $\lambda^*_i\leq \lambda_{i-1}$, then put $\lambda_i:=\lambda^*_i$. Otherwise, stop the procedure and put $\lambda_j:=\lambda_{i-1}$ for $j\geq i$;
		\State  \textbf{end for}
  \end{algorithmic}\label{21091256}
	}
\end{algorithm}

To justify the need for correction, when the columns in the design matrix are realizations of independent random variables, we performed simulations with entries of design matrix generated from $\mathcal{N}(0,1/n)$ distribution. For each target gFDR level and true support size we generated observations according to (\ref{gmodel}) with $\sigma=1$ and $m=1000$ groups, each containing $l=5$ variables. For each index $i$, included to true support, we randomized group effect such as $\|X_{I_i}\beta_{I_i}\|_2: = B(m,l)$, for $B(m,l): = \sqrt{4\log(m)/\big(1-m^{-\frac2l}\big)-l}$.

In Figure \ref{21091237} we show gFDR for target levels $q_1=0.05$ and $q_2=0.1$, when lambdas are chosen based Theorem \ref{gFDRcontrol} (basic lambdas, Figure \ref{21091237}a) and when we apply the correction given by Procedure \ref{21091256} (corrected lambdas, Figure \ref{21091237}b). First $100$ coefficients of the corresponding sequences of smoothing parameters are shown in Figure \ref{21091237}c.
\begin{figure}[ht]
\centering
\begin{subfigure}{.33\textwidth}
  \centering
	\includegraphics[width=1\linewidth]{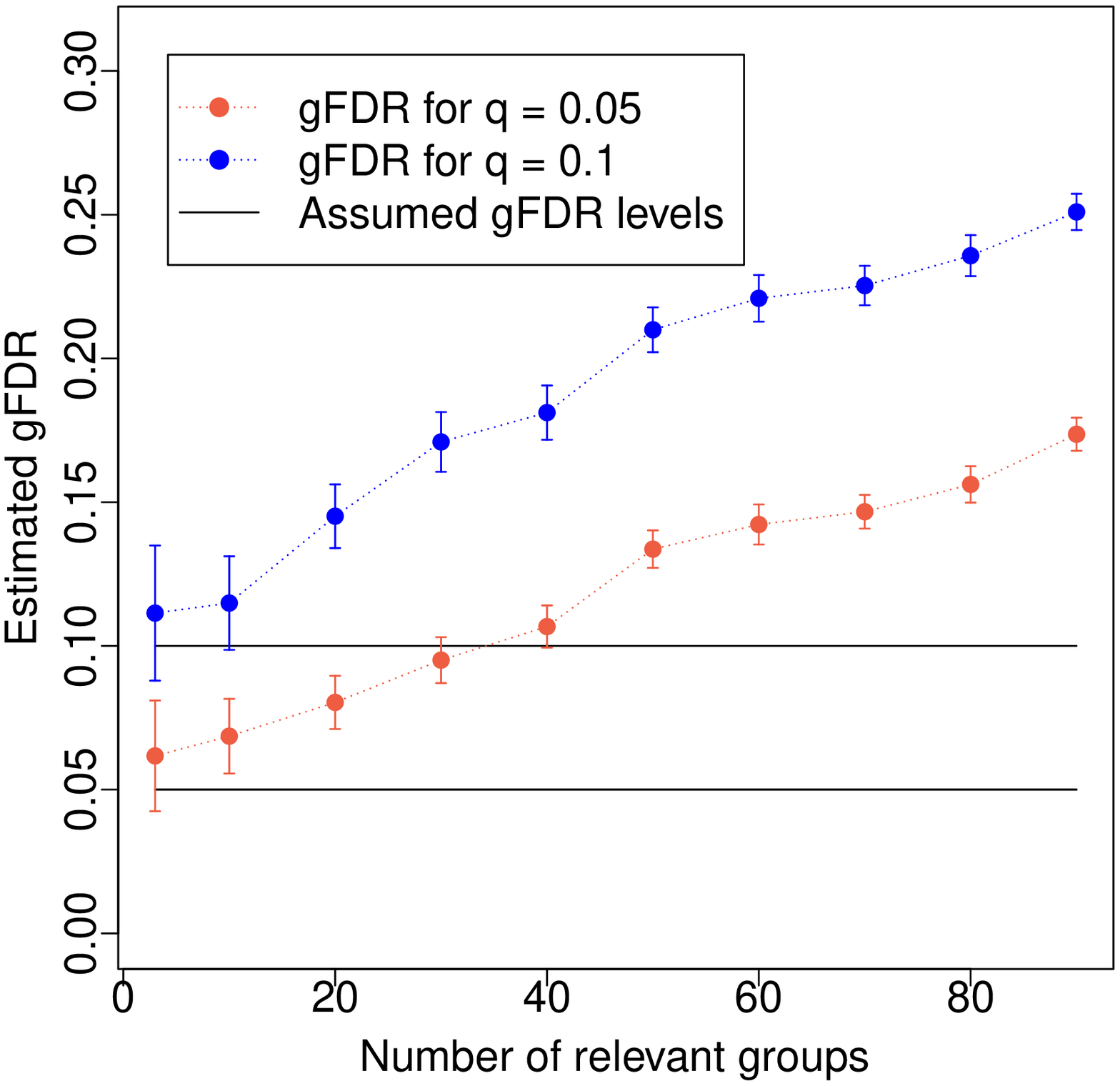}
  \caption{gFDR, basic lambdas}
\end{subfigure}%
\begin{subfigure}{.33\textwidth}
  \centering
  \includegraphics[width=1\linewidth]{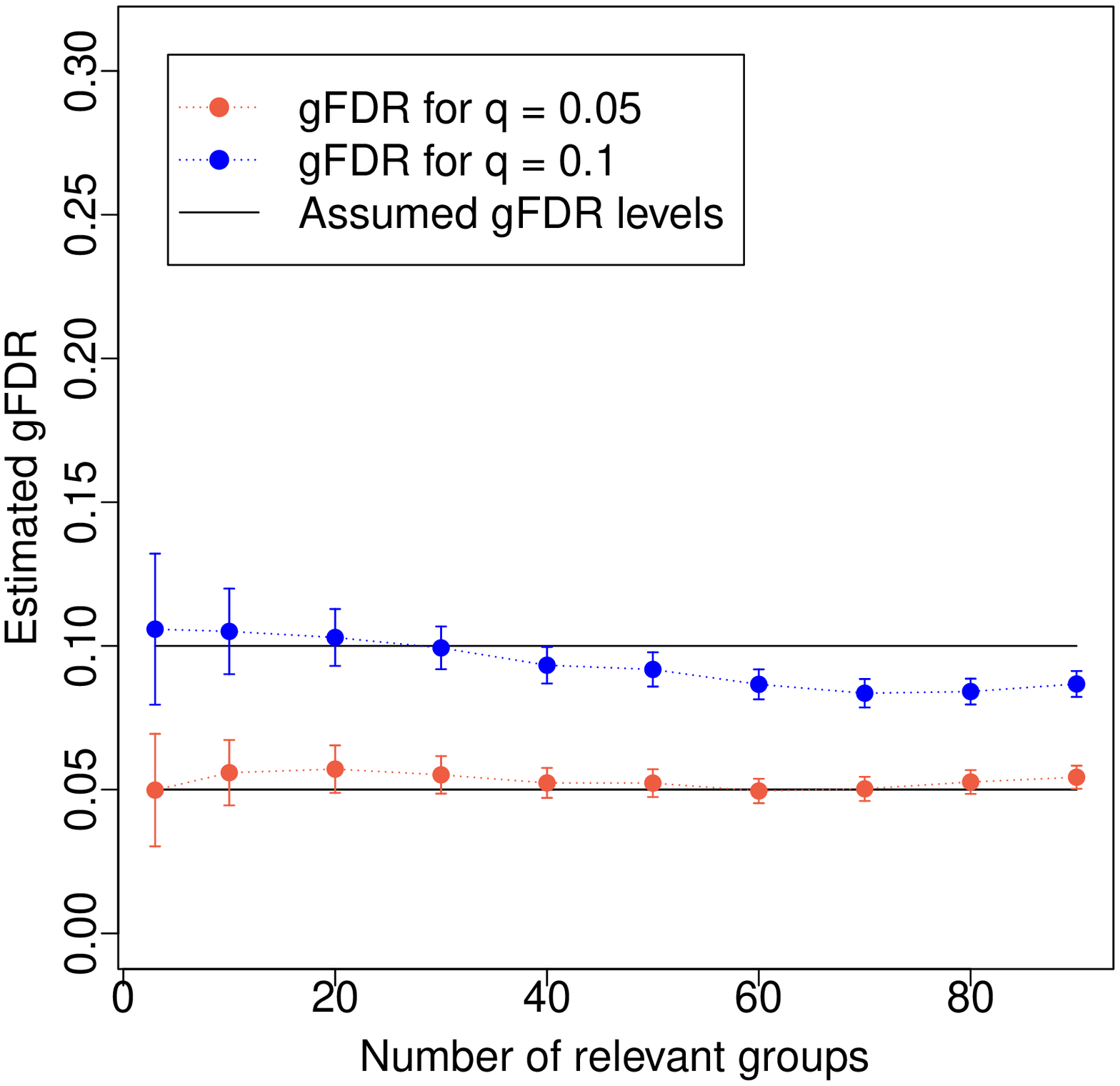}
  \caption{gFDR, corrected lambdas}
\end{subfigure}%
\begin{subfigure}{.33\textwidth}
  \centering
  \includegraphics[width=1\linewidth]{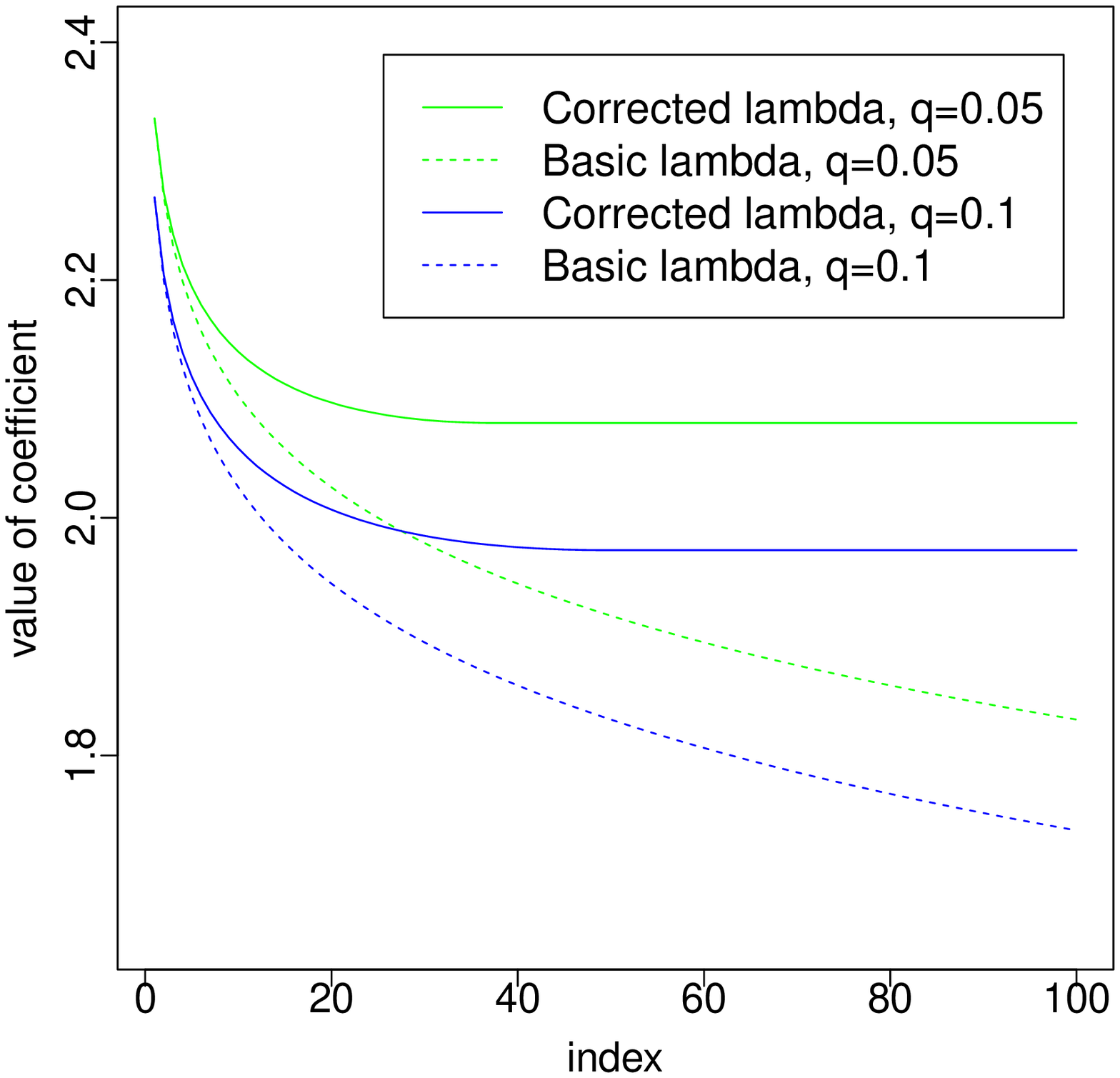}
  \caption{Basic and corrected lambdas}
\end{subfigure}%
\caption{Near-orthogonal situation with equal groups sizes for $l=5$, $m=1000$ and $p=n=5000$. For each target gFDR level and true support size, $200$ iterations were performed, bars correspond to $\pm 2$SE. Entries of design matrix were drawn from $\mathcal{N}(0,1/n)$ distribution and truly relevant signal, $i$, was generated such as $\|\beta_{I_i}\|_2=\sqrt{4\ln(m)/(1-m^{-2/l})-l}$.}
\label{21091237}
\end{figure}
Clearly, basic lambdas, defined in Theorem \ref{gFDRcontrol}, lead to excessive of the target gFDR, even when only relatively few groups are truly relevant. This means that the correction of smoothing parameters is necessary under occurrence of nonzero correlations between groups. The results (Figure 7b) show that our procedure allows to control gFDR for quite wide range of sparsity.

Consider now the Gaussian design with arbitrary groups sizes and sequence of positive weights $w_1,\ldots,w_m$. One possible approach is to construct consecutive $\lambda_i$ as the largest scaled quantiles among all distributions, i.e. as $\max\limits_{j=1,\ldots,m}\left\{\frac{\mathcal{S}_j}{w_j}F^{-1}_{\chi_{l_j}}\left (1-\frac{q\cdot i}{m}\right )\right\}$ for corrections $\mathcal{S}_j$'s adjusted to different $l_i$ values (the conservative strategy). In this article, however, we will stick to the more liberal strategy and we will construct relaxed lambdas basing on the concept used earlier in definition of $\lambda^{mean}$ \eqref{08191453}. Therefore we will generate lambdas according to Procedure \ref{11091537}, 
\begin{algorithm}
{\fontsize{9pt}{5pt}\selectfont
  \caption{Selecting lambdas in near-orthogonal situation: arbitrary groups sizes}
	\label{11091537}
  \begin{algorithmic}
		\State \textbf{input:} $q\in (0,1)$,\ \ $w_1,\ldots,w_m>0$,\ \ $p,\ n,\ m,\ l_1,\ldots, l_m \in\mathbb{N}$
		\State $\lambda_1:=\overline{F}^{-1}\left(1-\frac qm\right),\quad$ for $\quad\overline{F}(x): = \frac1m\sum_{i=1}^mF_{w^{-1}_i\chi_{l_i}}(x)$;
		\State  \textbf{for} $i\in\{2,\ldots,m\}$:
		\State \indent $\lambda^S: = (\lambda_1,\ldots, \lambda_{i-1})^\mathsf{T}$;
		\State \indent $\mathcal{S}_j: = \sqrt{\frac{n-l_j(i-1)}n+\frac{w_j^2\|\lambda^S\|^2_2}{n-l_j(i-1)-1}},\qquad$ for $j\in\{1,\ldots,m\}$;
		\State \indent $\lambda^*_i:=\overline{F}^{-1}_{\mathcal{S}}\left(1-\frac {qi}m\right),\quad$ for $\quad\overline{F}_{\mathcal{S}}(x): = \frac1m\sum_{j=1}^m\frac{\mathcal{S}_j}{w_j}F_{\chi_{l_j}}(x)$;
		\State \indent if $\lambda^*_i\leq \lambda_{i-1}$, then put $\lambda_i:=\lambda^*_i$. Otherwise, stop the procedure and put $\lambda_j:=\lambda_{i-1}$ for $j\geq i$;
		\State  \textbf{end for}
  \end{algorithmic}
	}
\end{algorithm}
where the idea is to use the arithmetic mean of scaled distributions rather than the maximum value, which enables to discover larger number of truly relevant groups as compared to the conservative variant. We will test these lambdas in the next subsection, where we also go one step further and assume additionally that the variance of the stochastic error in unknown.

\subsection{The estimation of the variance of stochastic error, $\sigma^2$}
Up until this moment, we have used $\sigma$ in gSLOPE optimization problem, assuming that this parameter is known . However, in many applications $\sigma$ is unknown and its estimation is an important issue. When $n>p$, the standard procedure is to use the unbiased estimator of $\sigma^2$, $\hat{\sigma}_{\textnormal{\tiny OLS}}^2$, given by
\begin{equation}
\label{11091832}
\hat{\sigma}_{\textnormal{\tiny OLS}}^2:=\big(y-X\beta^{\textnormal{\tiny OLS}}\big)^\mathsf{T}\big(y-X\beta^{\textnormal{\tiny OLS}}\big)/(n-p),\textrm{ for }\beta^{\textnormal{\tiny OLS}}:=(X^\mathsf{T}X)^{-1}X^\mathsf{T}y.
\end{equation}
For the target situation, with $p$ much larger than $n$, such an estimator can not be used. To estimate $\sigma$ we will therefore apply the procedure which was dedicated for this purpose in $\cite{SLOPE}$ in the context of SLOPE. Below we present algorithm adjusted to gSLOPE (Procedure \ref{11091824}).
\begin{algorithm}
{\fontsize{9pt}{5pt}\selectfont
  \caption{gSLOPE with estimation of $\sigma$}
	\label{11091824}
  \begin{algorithmic}
		\State \textbf{input:} $y$,\ $X$ and $\lambda$ (defined for some fixed $q$)
		\State  \textbf{initialize:} $S_+=\emptyset$;
		\State  \textbf{repeat}
		\State  \indent $S=S_+$;
		\State  \indent \textrm{compute RSS obtained by regressing }$y${ onto variables in }$S$;
		\State  \indent set $\hat{\sigma}=RSS/(n-|S|-1)$;
		\State  \indent compute the solution $\beta^{\textnormal{\tiny gS}}$ to gSLOPE with parameters $\hat{\sigma}$ and sequence $\lambda$;
		\State  \indent set $S_+=\operatorname{supp}(\beta^{\textnormal{\tiny gS}})$;
		\State  \textbf{until} $S_+=S$
  \end{algorithmic}
	}
\end{algorithm}
The idea standing behind the procedure is simple. The gSLOPE property of producing sparse estimators is used, and in each iteration columns in design matrix are first restricted to support of $\beta^{\textnormal{\tiny gS}}$, so the numbers of rows exceeds the number of columns and (\ref{11091832}) can be used. Algorithm terminates when gSLOPE finds the same subset of relevant variables as in the preceding iteration.

To investigate the performance of gSLOPE under the Gaussian design and various groups sizes, we performed simulations with $1000$ groups. Their sizes were drawn from the binomial distribution, $Bin(1000; 0.008)$, so as the expected value of the group size was equal to $8$ (Figure \ref{21091752}c). As a result, we obtained $7917$ variables, divided into $1000$ groups (the same division was used in all iterations and scenarios). For each sparsity level, two target gFDR levels, $0.05$ and $0.1$, and each iteration we generated entries of the design matrix using $\mathcal{N}\big(0,\frac1n\big)$ distribution, then $X$ was standardized and observation were generated according to model (\ref{gmodel}) with $\sigma=1$. To obtain estimates of relevant groups we have used the iterative version of gSLOPE, with $\sigma$ estimation (Procedure \ref{11091824}) and lambdas given by Procedure \ref{11091537}. We performed $200$ repetitions for each scenario, $n$ was fixed as $5000$. Results are represented in Figure \ref{21091752}.
\begin{figure}[ht]
\centering
\begin{subfigure}{.33\textwidth}
  \centering
	\includegraphics[width=1\linewidth]{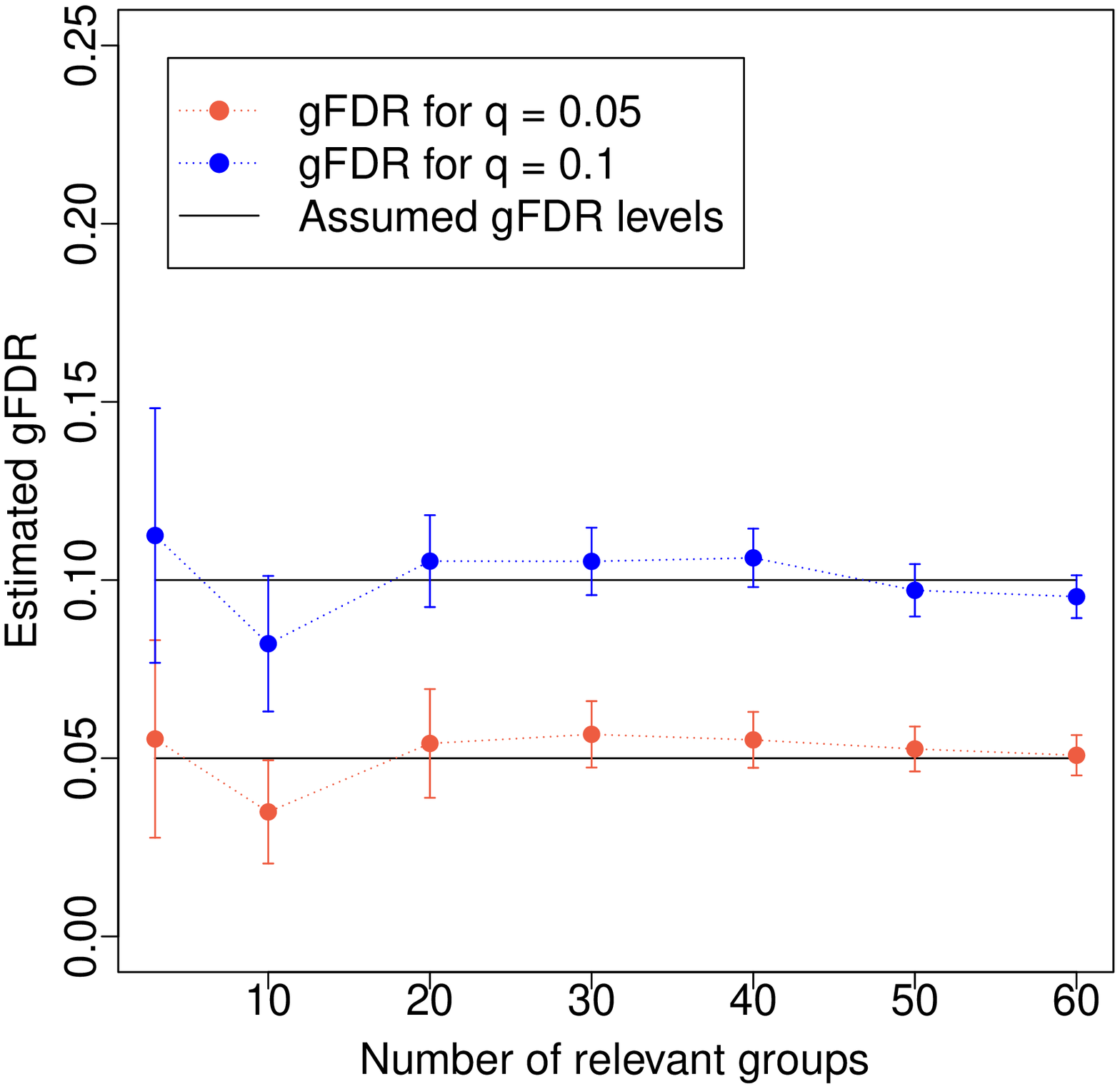}
  \caption{gFDR}
\end{subfigure}%
\begin{subfigure}{.33\textwidth}
  \centering
  \includegraphics[width=1\linewidth]{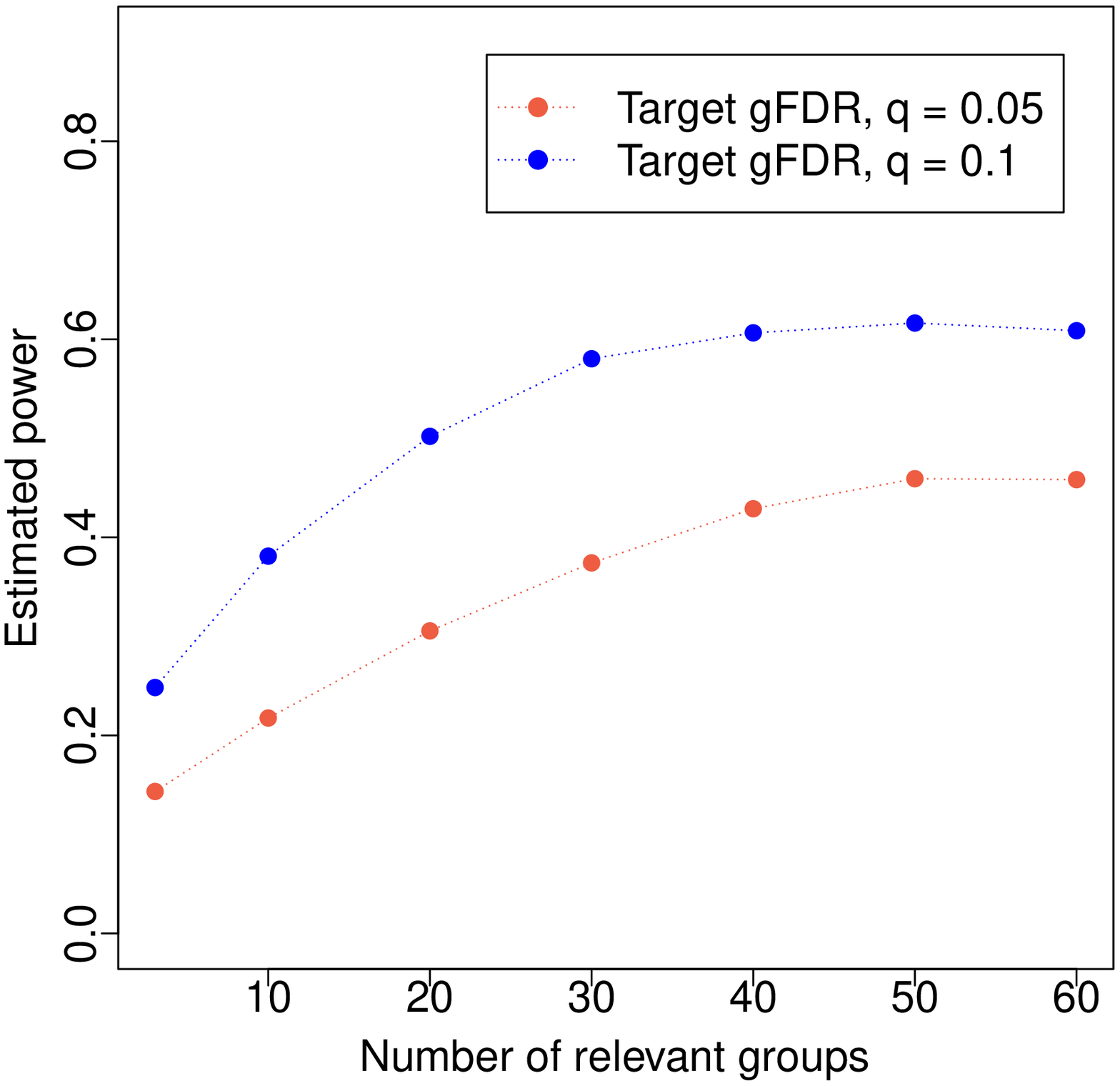}
  \caption{Power}
\end{subfigure}%
\begin{subfigure}{.33\textwidth}
  \centering
  \includegraphics[width=1\linewidth]{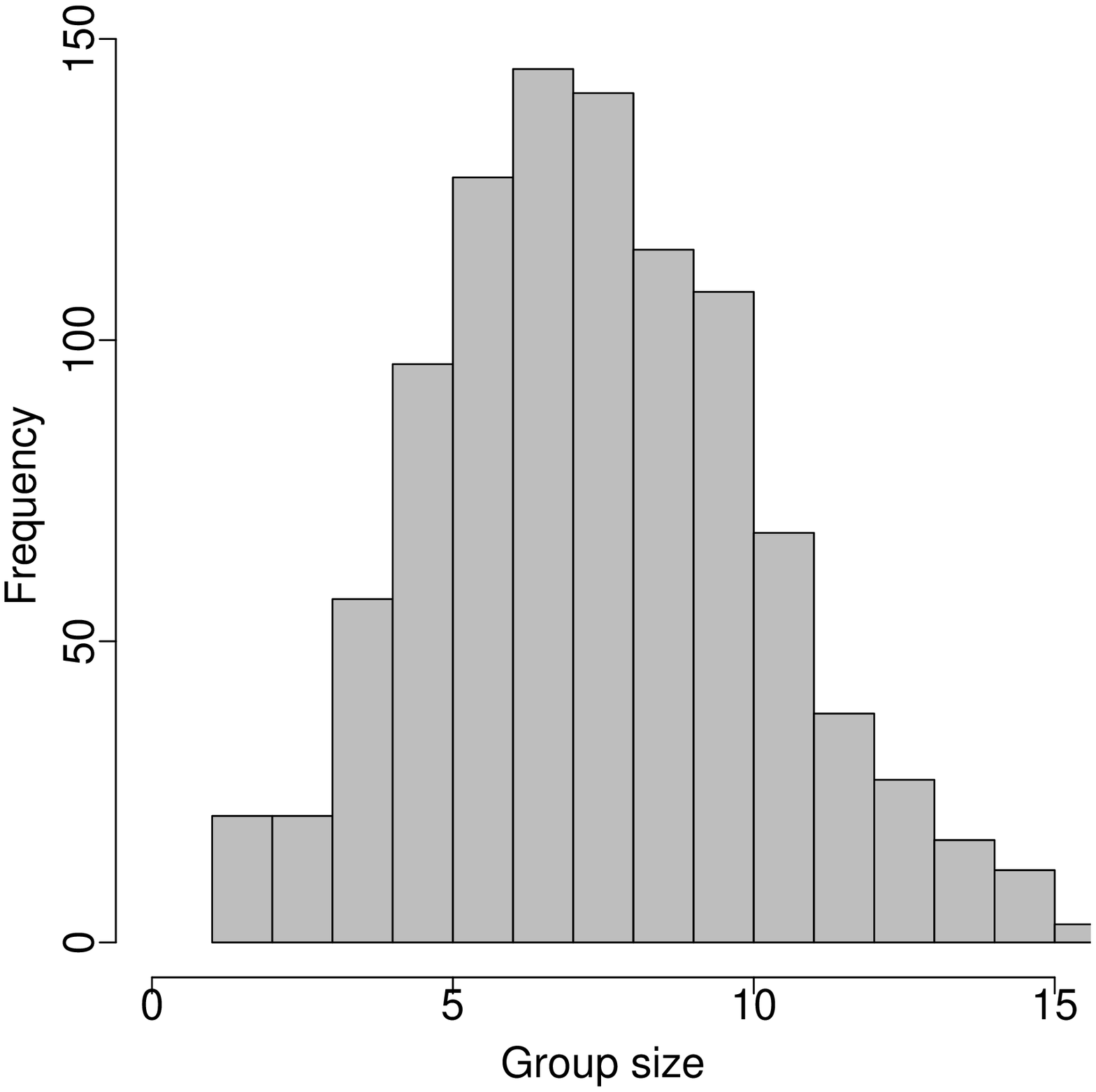}
  \caption{Histogram for group sizes}
\end{subfigure}%
\caption{Near-orthogonal situation for various groups sizes with $m=1000$, $p=7917$ and $n=5000$. Bars correspond to $\pm 2$SE. Entries of design matrix were drawn from $\mathcal{N}(0,1/n)$ distribution and truly relevant signal, $i$, was generated such as $\|X_{I_i}\beta_{I_i}\|_2=\frac1m\sum_{i=1}^mB(m,l_i)$, where $B(m,l)$ is defined in (5.24).}
\label{21091752}
\end{figure}

\section{Numerical algorithm}\label{31071347}
In this section we will discuss the convexity of the objective function and the algorithm for computing the solution to gSLOPE problem (\ref{gSLOPE}). Our optimization method is based on the fast algorithm for evaluation proximity operator (prox) for sorted $\ell_1$ norm, which was derived in \cite{SLOPE}.
\subsection{Convexity of the objective function}
\noindent To show that the objectives in problems \eqref{gSLOPE} and (\ref{24111029}) are a convex functions, we will prove the following propositions
\begin{proposition}
Function $J_{\lambda, W, \II}(b):=J_{\lambda}\Big( W\iI{b}\Big)$ is a norm for any nonnegative, nonincreasing sequence $\{\lambda_i\}_{i=1}^m$ containing at least one nonzero element, partition $\II$ of the set $\{1,\ldots,\widetilde{p}\}$ and diagonal matrix $W$ with positive elements on diagonal.
\end{proposition}
\begin{proof}
It is easy to see that $J_{\lambda, W, \II}(c) = 0$ if and only if $c=0$ and that for any scalar $\alpha\in \mathbb{R}$ it occurs $J_{\lambda, W, \II}(\alpha c) = |\alpha|J_{\lambda, W, \II}(c)$. We will show that $J_{\lambda, W, \II}$ satisfies the triangle inequality. Let $b, c$ be any vectors from $\mathbb{R}^{\widetilde{p}}$. From the positivity of $w_i$'s we have $W\iI{a+b}\preceq W\iI{a} + W\iI{b}$. Therefore, Corollary \ref{17021046} yields
\begin{equation}
\begin{split}
J_{\lambda, W, \II}\big(a+ b\big)= & J_{\lambda}\Big( W\iI{a+b}\Big)\leq\  J_{\lambda}\Big( W\iI{a} + W\iI{b}\Big)\leq \\ & J_{\lambda}\Big( W\iI{a}\Big) + J_{\lambda}\Big(W\iI{b}\Big) = J_{\lambda, W, \II}(a) + J_{\lambda, W, \II}(b).
\end{split}
\end{equation}
since $J_{\lambda}$ is the norm.
\end{proof}

\begin{proposition}
Function $J_{\lambda}\Big( W\XI{b}\Big)$ is a seminorm for any nonnegative, nonincreasing sequence $\{\lambda_i\}_{i=1}^m$, partition $I$ of the set $\{1,\ldots,p\}$, design matrix $X\in M(n,p)$ and diagonal matrix $W$ with positive elements on diagonal.
\end{proposition}
\begin{proof}
Clearly, $J_{\lambda}\Big( W\XI{\alpha b}\Big)= |\alpha|J_{\lambda}\Big( W\XI{b}\Big)$, for any scalar $\alpha\in \mathbb{R}$. Moreover, for any $a,b\in \mathbb{R}^p$, it holds $W\XI{a+b}\preceq W\XI{a} + W\XI{b}$, and the triangle inequality could be proved similarly as in the previous proposition.
\end{proof}

\subsection{Proximal gradient method}
Consider unconstrained optimization problem of form
\begin{equation}
\label{07030001}
\minimize{b} \ f(b) = g(b) + h(b),
\end{equation}
where $g$ and $h$ are convex functions and $g$ is differentiable (for example LASSO and SLOPE are of such form). There exist efficient methods, namely {\itshape proximal gradient algorithms}, which could be applied to find numerical solution for such objective functions. To design efficient algorithms, however, $h$ must be prox-capable, meaning that there is known fast algorithm for computing the proximal operator for $h$, 
\begin{equation}
prox_{th}(y):=\argmin{b}\left\{\frac1{2t}\|y-b\|_2^2+h(b)\right\},
\end{equation}
for each $y\in\mathbb{R}^p$ and $t>0$.
The iterative algorithm work as follows. Suppose that in $k$ step $b^{(k)}$ is the current guess. Then, guess $b^{(k+1)}$ is given by
\begin{equation}
\label{17021824}
b^{(k+1)}:=\argmin{b}\left\{g\Big(b^{(k)}\Big)+\left\langle\nabla g\Big(b^{(k)}\Big),b-b^{(k)}\right\rangle+\frac1{2t}\|b-b^{(k)}\|_2^2+h(b)\right\}.
\end{equation} 
The two first terms in objective function in (\ref{17021824}) are Taylor approximation of $g$, third addend is a proximity term which is responsible for searching an update reasonably close and $t$ can be treated as a step size.

Problem (\ref{17021824}) could be reformulated to  
\begin{equation}
b^{(k+1)}:=\argmin{b}\left\{\frac12\left\|b^{(k)}-t\nabla g\big(b^{(k)}\big)-b\right\|_2^2+th(b)\right\},
\end{equation}
hence $b^{(k+1)}=prox_{th}\Big(b^{(k)}-t\nabla g\big(b^{(k)}\big)\Big)$, which justifies the need for existence of fast algorithm computing values of proximal operator. In each step the value of $t$ could be changed raising the sequence $\{t_i\}_{i=1}^{\infty}$. In situation when $g(b)=\frac12\|y-Xb\|_2^2$, we get following algorithm.
\begin{algorithm}
  \caption{Proximal gradient algorithm}
	\label{24021312}
  \begin{algorithmic}
		\State \textbf{input:} $b^{[0]}\in \mathbb{R}^p$,\ \ k=0
    \While{ ( Stopping criteria are not satisfied) }
		\State 1. $b^{[k+1]}=prox_{t_kh_{\lambda}}\Big(b^{[k]}-t_kX^\mathsf{T}\big(Xb^{[k]}-y\big)\Big)$;
		\State 2. $k\gets k+1$.
    \EndWhile
  \end{algorithmic}
\end{algorithm}

\noindent It is known that $t_i$'s could be selected in different ways to ensure that $f(b^{(k)})$ converges to the optimal value \cite{FISTA}, \cite{FISTA2}.
\subsection{Proximal operator for gSLOPE}
Let $I=\{I_1,\ldots,I_m\}$, $l_i$ be rank of submatrix $X_{I_i}$ for $i=1,\ldots,m$ and $\lambda=(\lambda_1,\ldots,\lambda_m)^\mathsf{T}$ be vector satisfying $\lambda_1 \geq\ldots\geq\lambda_m\geq0$. We will now employ the proximal gradient method to find the numerical solution to $\eqref{gSLOPE}$. As stated in subsection \ref{subsec:gs2711944}, we can focus on the equivalent optimization problem \eqref{24111029}, namely we aim to solve problem
\begin{equation}
\label{27111021}
b^*:=\argmin {b}\ \ \bigg\{\frac12\big\|y-\widetilde{X}b\big\|_2^2+\sigma J_{\lambda}\Big( W\iI{b}\Big)\bigg\},
\end{equation}
with $\II = \{\II_1,\ldots,\II_m\}$ being a partition of the set $\{1,\ldots,\widetilde{p}\}$ where $\widetilde{p}=l_1+\ldots+l_m$.

Without loss of generality we assume that $\sigma=1$. Since considered objective is of form (\ref{07030001}), we can apply proximal gradient algorithm, provided that norm $J_{\lambda, \II, W}$ is prox-capable. To compute the proximal operator for $J_{\lambda, \II, W}$ we we must be able to minimize $\frac1{2t}\|y-b\|_2^2+J_{\lambda, \II, W}(b),$ for any $y\in\mathbb{R}^{\widetilde{p}}$ and $t>0$. Multiplying objective by positive number, $t$, does not change the solution. Such operation leads to new objective function, $\frac12\|y-b\|_2^2+J_{t\lambda, \II, W}(b)$. This shows that it is enough to derive fast algorithm for finding numerical solution to problem
\begin{equation}
\label{07030029}
prox_J(y):=\argmin{b}\left\{\frac12\|y-b\|_2^2+J_{\lambda, \II, W}(b)\right\},
\end{equation}
which could be applicable to arbitrary sequence $\lambda_1\geq\ldots\geq\lambda_m\geq 0.$

We will start with situation when $W$ is identity matrix. Simply, then $prox_J(y)$ is proximal operator for function $J_{\lambda,\II}(b):=J_{\lambda}(\I{b})$. In such a case computing (\ref{07030029}) could be immediately reduced to finding prox for $J_{\lambda}$ norm, since thanks to (\ref{17022353}) we have
\begin{equation}
\left\{
\begin{array}{l}
c^*=\argmin{c}\left\{\frac12\big\|\iI{y}-c\big\|_2^2+J_{\lambda}(c)\right\}\\
\big(prox_J(y)\big)_{\II_i}=c^*_i \big(\|y_{\II_i}\|_2\big)^{-1}y_{\II_i} ,\quad i=1,\ldots,m
\end{array}.
\right.
\end{equation}
Consequently, $prox_J(y)$ could be obtained by applying two steps procedure: find $c^*$ by using fast prox algorithm for $J_{\lambda}$ for vector $\iI{y}$, and compute $prox_J(y)$ by applying simple calculus to $c^*$.

Consider now general situation with fixed positive numbers $w_1,\ldots, w_m$ and define diagonal matrix $M$ by conditions ${M_{\II_i,\II_i}:=w_i^{-1}\mathbf{I}_{l_i}}$, for $i=1,\ldots,m$. Then
\begin{equation}
J_{\lambda,\II,W}(b)=J_{\lambda}\big(W\iI{b}\big) = J_{\lambda}\big(\iI{M^{-1}b}\big)=J_{\lambda,\II}\big(M^{-1}b\big).
\end{equation}
Since $M$ is nonsingular, we can substitute $\eta:=M^{-1}b$ and consider equivalent formulation of \eqref{27111021}
\begin{equation}
\label{07031414}
\left\{
\begin{array}{l}
\eta^*:=\argmin{\eta}\left\{\frac12\|y-\X M\eta\|_2^2+J_{\sigma\lambda,\II}\big(\eta\big)\right\},\\
b^* = M\eta^*
\end{array}.
\right.
\end{equation}
Therefore, after modifying the design matrix, gSLOPE can be always recast as problem with unit weights. Since $J_{\lambda,\II}$ is prox-capable, applying proximal gradient method to (\ref{07031414}) is straightforward. To implement method introduced in this article, we have used developed version of an Algorithm \ref{24021312}, the accelerated proximal gradient method known as FISTA \cite{FISTA}. In particular FISTA gives precise procedure for choosing steps sizes, to achieve fast convergence rate. To derive proper stopping criteria, we have considered dual problem to gSLOPE and employed the strong duality property. Detailed description was placed in the Appendix.

\section*{Acknowledgement}

We would like to thank Emmanuel Cand\`es for helpful remarks and suggestions and to Alexej Gossman for sharing with us his results and helpful discussions.  D.~B. and M.~B. are partially supported by European Union's 7th Framework Programme for research, technological development and demonstration under Grant Agreement no 602552. W.~S. was partially supported by a General Wang Yaowu Stanford Graduate Fellowship.
\vspace{20 pt}

\bibliographystyle{plain}
\bibliography{gSLOPE}
\newpage
\appendix
\section{$J_\lambda$ norm properties}
For nonnegative, nonincreasing sequence $\lambda_1 \geq \ldots \geq \lambda_p \geq 0$ consider function $ \mathbb{R}^p \ni b \longmapsto J_\lambda(b) \in \mathbb{R}$ given by $J_\lambda(b)=\sum_{i=1}^{p}\lambda_i \cdot |b|_{(i)}$, where $|b|_{(1)}\geq \ldots\geq |b|_{(p)}$ is the vector of sorted absolute values.
\begin{proposition} \label{Props1}
If $a$, $b \in \mathbb{R}^p$ are such that $|a| \preceq |b|$, then $|a|_{(\cdot)} \preceq |b|_{(\cdot)}$.
\end{proposition}

\begin{proof}
Without loss of generality we can assume that $a$ and $b$ are nonnegative and that it occurs $a_1\geq \ldots\geq a_p$. We will show that $a_k \leq b_{(k)}$ for $k\in \{1,\ldots,p\}$. Fix such $k$ and consider the set $S_k:=\{b_i:\  b_i \geq a_k\}$. It is enough to show that $|S_k|\geq k$.
For each $j\in \{1,\ldots,k\}$ we have
\begin{equation*}
b_j \geq a_j \geq a_k\ \Longrightarrow\ b_j\in S_k,
\end{equation*} 
what proves the last statement.
\end{proof}

\begin{corollary}
\label{17021046}
Let $a \in \mathbb{R}^p$, $b \in \mathbb{R}^p$ and $|a| \preceq |b|$ then Proposition (\ref{Props1}) instantly gives that ${J_\lambda (a) \leq J_\lambda (b)}$, since $J_\lambda (a)=\lambda^\mathsf{T}|a|_{(\cdot)} \leq \lambda^\mathsf{T}|b|_{(\cdot)}= J_\lambda (b)$.
\end{corollary}


\begin{proposition} \label{Props 3}
For fixed sequence $\lambda_1\geq\ldots\geq\lambda_p\geq0$, let $b\in \mathbb{R}^p$ be such that $b \succeq 0$ and $b_j > b_l$ for some ${j,l\in \{1,\ldots,p\}}$. For $0 < \varepsilon \leq (b_j-b_l)/2$, define $b_\varepsilon\in\mathbb{R}^p$ by conditions $(b_\varepsilon)_l:=b_l+\varepsilon$, $(b_\varepsilon)_j:=b_j-\varepsilon$ and $(b_\varepsilon)_i:=b_i$ for $i\notin\{j,l\}$. Then $J_\lambda(b_\varepsilon) \leq J_\lambda(b)$.
\end{proposition}
\begin{proof}
Let $\pi:\{1,\ldots,p\}\longrightarrow\{1,\ldots,p\}$ be permutation such as $\sum_{i=1}^p\lambda_i (b_\varepsilon)_{(i)}=\sum_{i=1}^p\lambda_{\pi(i)} (b_\varepsilon)_i$ for each $i$ in $\{1,\ldots,p\}$ and $\lambda_{\pi(j)}\geq\lambda_{\pi(l)}$. From the rearrangement inequality (Theorem 368 in \cite{RearIneq}),
\begin{equation}
\label{06021616}
\begin{split}
J_{\lambda}(b)-J_{\lambda}(b_\varepsilon)=&\sum_{i=1}^p\lambda_ib_{(i)}-\sum_{i=1}^p\lambda_i(b_\varepsilon)_{(i)}=\sum_{i=1}^p\lambda_ib_{(i)}-\sum_{i=1}^p\lambda_{\pi(i)} (b_\varepsilon)_i\geq\\
&\sum_{i=1}^p\lambda_{\pi(i)}b_i-\sum_{i=1}^p\lambda_{\pi(i)} (b_\varepsilon)_i=\varepsilon\big(\lambda_{\pi(j)}-\lambda_{\pi(l)}\big)\geq0.
\end{split}
\end{equation}
\end{proof}

\section{Alternative representation of gSLOPE in the orthogonal case}\label{Sec:alt_rep}

\label{subsec:06232149}
Suppose that the experiment matrix is orthogonal at group level, i.e. it holds $X_{I_i}^\mathsf{T}X_{I_j} = \mathbf{0}$, for every $i,j\in \{1,\ldots,m\}$, $i\neq j$. In such a case, $\widetilde{X}$ in problem \eqref{24111029} is orthogonal matrix, i.e. $\widetilde{X}^\mathsf{T}\widetilde{X}=\mathbf{I}_{\widetilde{p}}$. If $n=\widetilde{p}$, i.e. $\widetilde{X}$ is a square and orthogonal matrix, we also have $\widetilde{X}\widetilde{X}^\mathsf{T} = \mathbf{I}_{\widetilde{p}}$ and it obeys $\|\widetilde{X}^\mathsf{T}b\|^2_2=b^\mathsf{T}\widetilde{X}\widetilde{X}^\mathsf{T}b=\|b\|^2_2$ for $b\in\mathbb{R}^{\widetilde{p}}$. For the general case with $n\geq \widetilde{p}$, we can extend $\widetilde{X}$ to a square matrix by adding new orthonormal columns and defining $\widetilde{X}_C:=\big [ \begin{BMAT}(b){c.c}{c}\widetilde{X}&C \end{BMAT} \big ]$, where $C$ is composed of vectors (columns) being some complement to orthogonal basis of $\mathbb{R}^{\widetilde{p}}$. For $y\in\mathbb{R}^n$ and $b \in \mathbb{R}^{\widetilde{p}}$ we get:
\begin{equation}\label{ortog} \Big\|y-\widetilde{X}b\Big\|^2_2=\Big\|\widetilde{X}_C^\mathsf{T}\left (y-\widetilde{X}b \right)\Big\|^2_2=\left\| \left[\begin{BMAT}(b,8pt,15pt){c}{c.c}\widetilde{X}^\mathsf{T}\\ C^\mathsf{T}\end{BMAT}\right]y - \left[\begin{BMAT}(b,8pt,15pt){c}{c.c} b\\ \mathbf{0}\end{BMAT}\right] \right\|^2_2=\Big \|\widetilde{X}^\mathsf{T}y-b\Big \|^2_2+const,\end{equation} 
which implies that under orthogonal situation the optimization problem in $(\ref{24111029})$ could be recast as 
\begin{equation}
\label{gSLOPE_ort}
\argmin b\ \ \left\{\frac 12\big\|\widetilde{y}-b\big\|_2^2+\sigma J_{\lambda}\big( W\iI{b}\big)\right\},
\end{equation}
for $\widetilde{y}: = \widetilde{X}^\mathsf{T}y$. After introducing new variable to problem (\ref{gSLOPE_ort}), namely $c\in\mathbb{R}^m$, we get the equivalent formulation
\begin{equation}
\label{07010812}
\argmin{b,c}\ \ \left\{\frac 12\big\|\widetilde{y}-b\big\|_2^2+\sigma J_{\lambda}(c):\ c=W\iI{b}\right\}.
\end{equation}
\begin{proposition}
\label{07061804}
Let $f(b,c):\mathbb{R}^p\times\mathbb{R}^m\longrightarrow\mathbb{R}$ be any function and consider optimization problem ${\argmin{b,c}\big\{f(b,c):\ (b,c)\in\mathcal{D}\big\}}$ with unique solution $(b^*,c^*)$ and feasible set $\mathcal{D}\subset\mathbb{R}^p\times\mathbb{R}^m$. Define $\mathcal{D}^c:=\big\{c\in\mathbb{R}^m|\ \exists b\in\mathbb{R}^p: (b,c)\in \mathcal{D}\big\}$. Suppose that for any $c\in\mathcal{D}^c$, there exists unique solution, $b^c$, to problem ${\argmin{b}\big\{f(b,c):\ (b,c)\in\mathcal{D}\big\}}$. Moreover, assume that the solution to ${\argmin{c}\big\{f(b^c,c):\ c\in\mathcal{D}^c\big\}}$ is unique. Then, it occurs
\begin{equation}
\label{07061532}
\left\{
\begin{array}{l}
c^*= \argmin{c}\big\{f(b^c,c):\ c\in\mathcal{D}^c\big\}\\
b^* = b^{c^*}
\end{array}
\right..
\end{equation}
\end{proposition} 
\begin{proof}
Suppose that there exists $(b^0,c^0)\in\mathcal{D}$, such that $f(b^0,c^0)<f(b^*,c^*)$, where $b^*$ and $c^*$ are defined as in (\ref{07061532}). We have
\begin{equation}
f(b^{c^0},c^0)\leq f(b^0,c^0)<f(b^*,c^*)=f(b^{c^*},c^*),
\end{equation}
which leads to the contradiction with definition of $c^*$.
\end{proof}
We will apply the above proposition to (\ref{07010812}). Let $(b^*,c^*)$ be solution to (\ref{07010812}). Then $b^*$ is also solution to convex problem (\ref{gSLOPE_ort}) with strictly convex objective function and therefore is unique. Since $c^*=W\iI{b^*}$, $c^*$ is unique as well. In considered situation ${\mathcal{D}^c=\big\{c:\ c\succeq0\big\}}$. We will start with solving the problem $b^c=\argmin{b}\ \left\{\frac 12\big\|\widetilde{y}-b\big\|_2^2+\sigma J_{\lambda}(c):\ c=W\iI{b}\right\}.$ The additive constant in the objective could be omitted. Moreover, for each $i\in\{1,\ldots,m\}$ we have
\begin{equation}
\label{07011714}
b^c_{\II_i}=\argmin{b_{\II_i}}\left\{\big\|\widetilde{y}_{\II_i}-b_{\II_i}\big\|^2_2:\ w^2_i\|b_{\II_i}\|^2_2-c^2_i=0\right\}.
\end{equation}
The Lagrange Multipliers method quickly yields $b^c_{\II_i}=(w_i\|\widetilde{y}_{\II_i}\|_2)^{-1}c_i\widetilde{y}_{\II_i}$ and, consequently, it holds $\left\|\widetilde{y}_{\II_i}-b^c_{\II_i}\right\|_2^2=\left(\|\widetilde{y}_{\II_i}\|_2-w_i^{-1}c_i\right)^2.$
From Proposition \ref{07061804}, we get the following procedure for solution, $b^*$, to problem (\ref{gSLOPE_ort})
\begin{equation}
\left\{
\begin{array}{l}
c^*=\argmin{c}\left\{\frac12\sum_{i=1}^m\big(\|\widetilde{y}_{\II_i}\|_2-w^{-1}_ic_i\big)^2+J_{\sigma\lambda}(c)\right\}\\
b^*_{\II_i}=c^*_i \big(w_i\|\widetilde{y}_{\II_i}\|_2\big)^{-1}\widetilde{y}_{\II_i} ,\quad i=1,\ldots,m
\end{array}
\right.
\end{equation}
(notice that we applied Proposition \ref{cnonneg} to omit the constraints $c\succeq0$ and that the objective function in definition of $c^*$ is strictly feasible, which guarantees the unique solution. The above procedure yields conclusion, that indices of groups estimated by gSLOPE as relevant coincide with the support of solution to SLOPE problem with diagonal matrix having inverses of weights $w_1,\ldots,w_m$ on diagonal. Moreover, after defining $\widetilde{\beta}\in\mathbb{R}^{\widetilde{p}}$ by conditions $\widetilde{\beta}_{\II_i}:=R_i\beta_{I_i}$, $i=1,\ldots,m$, we simply have $\iI{\widetilde{\beta}} = \XI{\beta}$ and
\begin{equation}
\widetilde{y}=\widetilde{X}^Ty=\widetilde{X}^T\bigg(\sum_{i=1}^mU_iR_i\beta_{I_i}+z\bigg) = \widetilde{X}^T\big(\widetilde{X}\widetilde{\beta}+z\big) = \widetilde{\beta}+\widetilde{X}^Tz,\quad \textrm{hence}\ \widetilde{y}\sim \mathcal{N}\big(\widetilde{\beta},\ \sigma \mathbf{I}_{\widetilde{p}}\big).
\end{equation}
Summarizing, if the assumption about the orthogonality at groups level is in use, one can consider the statistically equivalent model $\widetilde{y}\sim \mathcal{N}\big(\widetilde{\beta},\ \sigma \mathbf{I}_{\widetilde{p}}\big)$, define truly relevant groups via the support of $\iI{\widetilde{\beta}}$ and  treat the vector $\iI{b^*}=\big(\frac{c^*_1}{w_1},\ldots, \frac{c^*_m}{w_m}\big)$ as an gSLOPE estimate of group effect sizes, where $b^*$ and $c^*$ are defined in \eqref{17022353}, i.e. it holds $\iI{b^*}=\XI{\beta^\ES{gS}}$ for any solution $\beta^\ES{gS}$ to problem $\ref{gSLOPE}$.

\section{SLOPE with diagonal experiment matrix}

Let $y\in\mathbb{R}^p$ be fixed vector and $d_1,\ldots,d_p$ be positive numbers. We will use notation $diag(d_1,\ldots,d_p)$ to define the diagonal matrix $D$ such as $D_{i,i}=d_i$ for $i=1,\ldots, p$. Denote $d:=(d_1,\ldots, d_p)^\mathsf{T}$ and let $b^*$ be the solution to SLOPE optimization problem with diagonal experiment matrix, i.e. the solution to
\begin{equation}
\label{diagSLOPE}
\minimize{b}\,f(b): = \frac 12\big\|y-Db\big\|_2^2+J_{\lambda}\big( b\big).
\end{equation} 
Since $f$ is strictly convex function, the solution to (\ref{diagSLOPE}) is unique. It is easy to observe, that changing sign of $y_i$ corresponds to changing sign at $i$th coefficient of solution as well as permuting coefficients of $y$ together with $d_i's$ permutes coefficients of $b^*$. We will summarize this observations below without proofs.
\begin{proposition}
\label{PropPrem}
Let $\pi:\{1,\ldots,p\}\longrightarrow\{1,\ldots,p\}$ be given permutation with $P_{\pi}$ as corresponding matrix. Then:
\newline i) $P_{\pi}DP_{\pi}^\mathsf{T}=diag(d_{\pi(1)},\ldots, d_{\pi(p)})$;
\newline ii) $b_{\pi}:=P_{\pi}b^*$ is solution to 
$\minimize{b}\,f_{\pi}(b): = \frac12\Big\|P_{\pi}y-P_{\pi}DP_{\pi}^\mathsf{T}b\Big\|_2^2+ J_\lambda(b);$
\newline iii) $b_S:=Sb^*$ is solution to
$\minimize{b}\,f_S(b): = \frac12\Big\|Sy-Db\Big\|_2^2+ J_\lambda(b),$
\newline where $S$ is diagonal matrix with entries on diagonal coming from set $\{-1,1\}$.
\end{proposition}
\begin{proposition}
\label{cnonneg}
If $y\succeq 0$, then $b^*\succeq 0$.
\end{proposition}
\begin{proof}
Suppose that for some $r$ it occurs $b_r<0$ for any $b\in \mathbb{R}^p$. If $y_r=0$, then taking $\widehat{b}$ defined as $\widehat{b}_i:=\left\{\begin{array}{ll}0,&i=r\\ b_i,&\textrm{otherwise}\end{array}\right.$, we get $|\widehat{b}|\preceq |b|$ and Corollary \ref{17021046} gives $J_{\lambda}(\widehat{b})\leq J_{\lambda}(b)$. Consequently,
\begin{equation*}
f(b)-f(\widehat{b})\geq\frac 12\big\|y-Db\big\|_2^2 - \frac 12\big\|y-D\widehat{b}\big\|_2^2= \frac12(y_r-d_rb_r)^2-\frac12(y_r+d_r\widehat{b}_r)^2=\frac12d_r^2b_r^2>0.
\end{equation*}
Hence $b$ could not be the solution. Now consider case when $y_r>0$ and define $\widehat{b}$ by putting $\widehat{b}_i:=\left\{\begin{array}{ll}-b_r,&i=r\\ b_i,&\textrm{otherwise}\end{array}\right..$
Then we have $J_{\lambda}(b)=J_{\lambda}(\widehat{b})$ and
\begin{equation*}
f(b)-f(\widehat{b})=\frac12(y_r-d_rb_r)^2-\frac12(y_r+d_rb_r)^2=-2y_rd_rb_r>0.
\end{equation*}
and, as before, $b$ could not be optimal.
\end{proof}
\begin{proposition}
\label{06021048}
Let $b^*$ be the solution to problem (\ref{diagSLOPE}), $\{y_i\}_{i=1}^p$ be nonnegative sequence, $\{d_i\}_{i=1}^p$ be the sequence of positive numbers and assume that
\begin{equation}
d_1y_1\geq\ldots\geq d_py_p.
\end{equation}
If $b^*$ has exactly $r$ nonzero entries for $r>0$, then the set $\{1,\ldots,r\}$ corresponds to the support of $b^*$. 
\end{proposition}
\begin{proof}
It is enough to show that $$\big(j\in \{2,\ldots,m\},\ \ b^*_j\neq 0\big)\ \Longrightarrow\ b^*_{j-1} \neq 0.$$ Suppose that this is not true. From Proposition \ref{cnonneg} we know that $b^*$ is nonnegative, hence we can find $i$ from $\{2,\ldots,m\}$ such as $b_j^*>0$ and $b_{j-1}^*=0$. For $\varepsilon \in \big(0,b_j^*/2]$ define vector $b_{\varepsilon}$ by putting $(b_\varepsilon)_{j-1}:=\varepsilon$, $(b_\varepsilon)_j:=b_j^*-\varepsilon$ and $(b_\varepsilon)_i:=b^*_i$ for $i\notin\{j,l\}$. From Proposition \ref{Props 3} we have that $J_{\lambda}(b_{\varepsilon})\leq J_{\lambda}(b^*)$, which gives
\begin{equation}
\begin{split}
f(b^*)-f(b_{\varepsilon})&\geq \frac12\big(y_{j-1}-d_{j-1}b^*_{j-1}\big)^2 + \frac12\big(y_j-d_jb^*_j\big)^2 \\
&- \frac12\big(y_{j-1}-d_{j-1}b_{\varepsilon}(j-1)\big)^2 - \frac12\big(y_j-d_jb_{\varepsilon}(j)\big)^2=\\
&\varepsilon\left(A - \frac{d_{j-1}^2+d_j^2}2\cdot\varepsilon\right),\\
&\textrm{for }A:=(y_{j-1}d_{j-1}-y_jd_j)+d_j^2b_j^*>0.
\end{split}
\end{equation}
Therefore, we can find $\varepsilon>0$ such as $f(b^*)>f(b_{\varepsilon})$, which contradicts the optimality of $b^*$.
\end{proof}
Consider now problem (\ref{diagSLOPE}) with arbitrary sequence $\{y_i\}_{i=1}^p$. Suppose that $b^*$ has exactly $r>0$ nonzero coefficients and that $\pi:\{1,\ldots,p\}\longrightarrow\{1,\ldots,p\}$ is permutation which gives the order of magnitudes for $Dy$, i.e. $d_{\pi(1)}|y|_{\pi(1)}\geq\ldots\geq d_{\pi(p)}|y|_{\pi(p)}$. Basing on our previous observations, we get important
\begin{corollary}
\label{06021230}
If $b^*$ is the solution to (\ref{diagSLOPE}) having exactly $r>0$ nonzero coefficients and $\pi$ is permutation which places components of $D|y|$ in a nonincreasing order, i.e. $d_{\pi(i)}|y|_{\pi(i)}=|Dy|_{(i)}$ for $i=1,\ldots, p$, then the support of $b^*$ is composed of the set $\{\pi(1),\ldots,\pi(r)\}$.
\end{corollary}
The next three lemmas were proven in \cite{SLOPE} in situation when $d_1=\ldots=d_p=1$. We will follow the reasoning from this paper to prove the generalized claims. The main difference is that in general case the solution to considered problem (\ref{diagSLOPE}) does not have to be nonincreasingly ordered, under assumption that $d_1y_1\geq\ldots\geq d_py_p\geq0$ (which is the case for $d_1=\ldots=d_p=1$). This makes that generalizations of proofs presented in \cite{SLOPE} are not straightforward. 
\begin{lemma}
\label{lemma1}
Consider nonnegative sequence $\{y_i\}_{i=1}^p$ and sequence of positive numbers $\{d_i\}_{i=1}^p$ such as $d_1y_1\geq\ldots\geq d_py_p.$ If $b^*$ is solution to problem (\ref{diagSLOPE}) having exactly $r$ nonzero entries, then for every $j\leq r$ it holds that
\begin{equation}
\label{06291805}
\sum_{i=j}^r(d_iy_i-\lambda_i)>0
\end{equation}
and for every $j\geq r+1$
\begin{equation}
\label{06021637}
\sum_{i=r+1}^j(d_iy_i-\lambda_i)\leq0.
\end{equation}
\end{lemma}
\begin{proof}
From Proposition \ref{06021048} we know that $b_i^*>0$ for $i\in\{1,\ldots,r\}$. Let us define $$\widetilde{b}_i:=\left\{\begin{array}{rl}b_i^*-h,&i\in\{j,\ldots,r\}\\b^*_i,&\textrm{otherwise}.\end{array}\right.,$$ where we restrict only to sufficiently small values of $h$, so as to the condition $\widetilde{b}_i>0$ is met for all $i$ from $\{j,\ldots,r\}$. For such $h$ we have $b^*_{(r+1)}=\ldots=b^*_{(p)}=\widetilde{b}_{(r+1)}=\ldots=\widetilde{b}_{(p)}=0$. Therefore there exists permutation $\pi:\{1,\ldots,r\}\longrightarrow\{1,\ldots,r\}$ such as $\sum_{i=1}^r\lambda_i\widetilde{b}_{(i)}=\sum_{i=1}^r\lambda_{\pi(i)}\widetilde{b}_i$. For such permutation we have
\begin{equation}
\label{06021616}
\begin{split}
J_{\lambda}(b^*)-J_{\lambda}(\widetilde{b})=&\sum_{i=1}^r\lambda_ib^*_{(i)}-\sum_{i=1}^r\lambda_i\widetilde{b}_{(i)}=\sum_{i=1}^r\lambda_ib^*_{(i)}-\sum_{i=1}^r\lambda_{\pi(i)}\widetilde{b}_i\geq\\
&\sum_{i=1}^r\lambda_{\pi(i)}b^*_i-\sum_{i=1}^r\lambda_{\pi(i)}\widetilde{b}_i=h\sum_{i=j}^r\lambda_{\pi(i)}\geq h\sum_{i=j}^r\lambda_i,
\end{split}
\end{equation}
where the first inequality follows from the rearrangement inequality and second is the consequence of monotonicity of $\{\lambda_i\}_{i=1}^p$. We also have
\begin{equation}
\label{06021617}
\begin{split}
\|y-Db^*\|_2^2-\|y-D\widetilde{b}\|_2^2=&\sum_{i=j}^r(y_i-d_ib_i^*)^2-\sum_{i=j}^r(y_i-d_ib_i^*+d_ih)^2=\\
&2h\sum_{i=j}^r(d_i^2b_i^*-d_iy_i)-h^2\sum_{i=j}^rd_i^2.
\end{split}
\end{equation}
Optimality of $b^*$, (\ref{06021616}) and (\ref{06021617}) yield
\begin{equation}
0\geq f(b^*)-f(\widetilde{b})\geq h\sum_{i=j}^r(d_i^2b_i^*-d_iy_i+\lambda_i)-\frac12h^2\sum_{i=j}^rd_i^2,
\end{equation}
for each $h$ from the interval $[0,\varepsilon]$, where $\varepsilon>0$ is some (sufficiently small) value. This gives ${\sum_{i=j}^r(d_i^2b_i^*-d_iy_i+\lambda_i)\leq 0}$ and consequently
\begin{equation}
\sum_{i=j}^r(d_iy_i-\lambda_i)\geq \sum_{i=j}^rd_i^2b_i^*>0.
\end{equation}
To prove claim (\ref{06021637}), consider a new sequence defined as $\widetilde{b}_i:=\left\{\begin{array}{rl}h,&i\in\{r+1,\ldots,j\}\\b^*_i,&\textrm{otherwise}.\end{array}\right..$ We will restrict our attention only to $0<h<\min\{b_i^*:\ i\leq	 r\},$ so as to $b^*_{(\cdot)}$ and $\widetilde{b}_{(\cdot)}$ are given by applying the same permutation to $b^*$ and $\widetilde{b}$, respectively. Moreover, for each $i$ from $\{r+1,\ldots,j\}$ it holds $\widetilde{b}_{(i)}=\widetilde{b}_i=h$. From optimality of $b^*$
\begin{equation*}
0\geq f(b^*)-f(\widetilde{b})=\frac12\sum_{i=r+1}^j\left(y_i^2-(y_i-d_ih)^2\right)-\sum_{i=r+1}^j\lambda_ih=h\sum_{i=r+1}^j(d_iy_i-\lambda_i)-\frac12h^2\sum_{i=r+1}^jd_i^2,
\end{equation*}
for all considered $h$, which leads to (\ref{06021637}). 
\end{proof}
\begin{lemma}
\label{lemma2}
Let $b^*$ be solution to problem (\ref{diagSLOPE}) with nonnegative, nonincreasing sequence $\{\lambda_i\}_{i=1}^p$. Let $R(b^*)$ be number of all nonzeros in $b^*$ and $r\geq1$. Then, for any $i\in\{1,\ldots,p\}$
$$\big\{y:\ b^*_i\neq 0\textrm{ and }R(b^*)=r\big\}=\big\{y:\ d_i|y_i|>\lambda_r\textrm{ and }R(b^*)=r\big\}.$$
\end{lemma} 
\begin{proof} 
Suppose that $b^*$ has $r>0$ nonzero coefficients and let $\pi$ be permutation which places components of $D|y|$ in a nonincreasing order. From Corollary \ref{06021230} it holds that $\{i:\ b^*_i\neq0\} = \{\pi(1),\ldots, \pi(r)\}$. Define $\widetilde{y}:=P_{\pi}Sy$ and $\widetilde{D}:=P_{\pi}DP_{\pi}\T{T}$, for $S$ being the diagonal matrix such as $S_{i,i}=sgn(y_i)$. Then $P_{\pi}Sb^*$ is solution to problem
\begin{equation}
\label{06291803}
\argmin{b} \frac12\left\|\widetilde{y}-\widetilde{D}b\right\|^2_2+ J_{\lambda}(b),
\end{equation} 
which satisfies the assumptions of Lemma \ref{lemma1}. Taking $j=r$ in (\ref{06291805}) and $j=r+1$ in $(\ref{06021637})$ we immediately get
\begin{equation}
\label{06291817}
d_{\pi(r)}|y|_{\pi(r)}>\lambda_r\ \ \textrm{and}\ \ d_{\pi(r+1)}|y|_{\pi(r+1)}\leq \lambda_{r+1}.
\end{equation}
We will now show that $\big\{y:\ b^*_i\neq 0\textrm{ and }R(b^*)=r\big\}\subset\big\{y:\ d_i|y_i|>\lambda_r\textrm{ and }R(b^*)=r\big\}.$ Fix $i\in\{1,\ldots,p\}$ and suppose that $b^*_i$ is nonzero coefficient. Then $i\in\{\pi(1),\ldots,\pi(r)\}$ and therefore $d_i|y_i|\geq d_{\pi(r)}|y|_{\pi(r)}>\lambda_r$, thanks to first inequality from (\ref{06291817}).
To show the second inclusion assume that $d_i|y_i|>\lambda_r$. Then, from the second inequality in (\ref{06291817}), $d_i|y_i|>\lambda_{r+1}\geq d_{\pi(r+1)}|y|_{\pi(r+1)}$, which gives $i\in\{\pi(1),\ldots,\pi(r)\}.$
\end{proof}
\begin{lemma}
\label{lemma3}
For given sequence $\{y_i\}_{i=1}^p$, sequence of positive numbers $\{d_i\}_{i=1}^p$, nonincreasing, nonnegative sequence $\{\lambda_i\}_{i=1}^p$ and fixed $j\in \{1,\ldots,p\}$, consider a following procedure
\begin{itemize}
\item define ${\widetilde{y}:=(y_1,\ldots,y_{j-1},y_{j+1},\ldots,y_p)^\mathsf{T}}$, $\widetilde{D}:=diag(d_1,\ldots,d_{j-1},d_{j+1},\ldots,d_p)$, $\widetilde{d}_i:=\widetilde{D}_{i,i}$ for $i=1,\ldots,p-1$ and ${\widetilde{\lambda}:=(\lambda_2,\ldots,\lambda_p)^\mathsf{T}}$;
\item find $\widetilde{b}^*:=\argmin{b\in\mathbb{R}^{p-1}}\frac{1}{2}\big\|\widetilde{y}-\widetilde{D}b\big\|_2^2+J_{\widetilde{\lambda}}(b);$
\item define $\widetilde{R}^j(\widetilde{b}^*):=|\{i:\ \widetilde{b}^*_i\neq 0\}|.$
\end{itemize}
Then for $r\geq1$ it holds $\big\{y:\ d_j|y_j|>\lambda_r\textrm{ and }R(b^*)=r\big\}\subset\big\{y:\ d_j|y_j|>\lambda_r\textrm{ and }\widetilde{R}^j(\widetilde{b}^*)=r-1\big\}.$
\end{lemma}
\begin{proof}
We have to show that solution $\widetilde{b}^*$ to problem
\begin{equation}
\label{07022239}
\minimize{b}\ F(b): = \frac{1}{2}\sum_{i=1}^{p-1}\left(\widetilde{y}_i-\widetilde{d}_ib_i\right)^2+\sum_{i=1}^{p-1}\widetilde{\lambda}_ib_{(i)}
\end{equation}
has exactly $r-1$ nonzero coefficients. From Proposition \ref{PropPrem} we know that the change of signs of $y_i$'s does not affect the support, hence without loss of generality we can assume that $\widetilde{y}\succeq 0$, and $\widetilde{b}^*\succeq 0$ as a result (from Proposition \ref{cnonneg}). We will start with situation when $d_1y_1\geq\ldots\geq d_py_p$ and consequently $\widetilde{d}_1\widetilde{y}_1\geq\ldots\geq \widetilde{d}_{p-1}\widetilde{y}_p$. If $j$ is fixed index such as $d_j|y_j|>\lambda_r$ and $R(b^*)=r$, this gives
\begin{equation}
\label{08020010}
j\in\{1,\ldots,r\}.
\end{equation}
To show that solution to (\ref{07022239}) has at least $r-1$ nonzero entries, suppose by contradiction that $\widetilde{b}^*$ has exactly $k-1$ nonzero entries with $k<r$. Let us define $\widehat{b}\in\mathbb{R}^{p-1}$ as
$$\widehat{b}_i:=\left\{\begin{array}{ll}h,&i\in\{k,\ldots,r-1\}\\ \widetilde{b}^*_i,&\textrm{otherwise}\end{array}\right.,$$
where $0<h<\min\{\widetilde{b}_1^*,\ldots,\widetilde{b}_{k-1}^*\}$. Then 
\begin{equation}
F(\widetilde{b}^*)-F(\widehat{b})=h\sum_{i=k}^{r-1}(\widetilde{d}_i\widetilde{y}_i-\widetilde{\lambda}_i)-h^2\sum_{i=k}^{r-1}\frac12\widetilde{d}_i^2.
\end{equation}
Now
\begin{equation}
\sum_{i=k}^{r-1}(\widetilde{d}_i\widetilde{y}_i-\widetilde{\lambda}_i)=\sum_{i=k+1}^r(\widetilde{d}_{i-1}\widetilde{y}_{i-1}-\lambda_i)\geq \sum_{i=k+1}^r(d_iy_i-\lambda_i)>0,
\end{equation}
where the first equality follows from $\widetilde{\lambda}_i=\lambda_{i+1}$, the first inequality from $\widetilde{d}_{i-1}\widetilde{y}_{i-1}\geq d_iy_i$ and the second from Lemma \ref{lemma1}. If $h$ is small enough, we get $F(\widehat{b})<F(\widetilde{b}^*)$ which leads to contradiction.

Suppose now by contradiction that $\widetilde{b}^*$ has $k$ nonzero entries with $k\geq r$ and define
$$\widehat{b}_i:=\left\{\begin{array}{ll}\widetilde{b}^*_i-h,&i\in\{r,\ldots,k\}\\ \widetilde{b}^*_i,&\textrm{otherwise}\end{array}\right..$$
Analogously to (\ref{06021616}), we get
$J_{\widetilde{\lambda}}(\widetilde{b}^*)-J_{\widetilde{\lambda}}(\widehat{b})\geq h\sum_{i=r}^k\widetilde{\lambda}_i$ and consequently
\begin{equation}
F(\widetilde{b}^*)-F(\widehat{b})\geq h\left[\sum_{i=r}^k(\widetilde{\lambda}_i-\widetilde{d}_i\widetilde{y}_i)+\sum_{i=r}^k\widetilde{d}_i^2\widetilde{b}^*_i\right]-\frac12h^2\sum_{i=r}^k\widetilde{d}^2_i.
\end{equation}
Now
\begin{equation}
\sum_{i=r}^k(\widetilde{\lambda}_i-\widetilde{d}_i\widetilde{y}_i)=\sum_{i=r+1}^{k+1}(\lambda_i-d_iy_i)\geq0,
\end{equation}
where the first equality follows from definition of $\widetilde{\lambda}$ and ($\ref{08020010}$), while the inequality follows from Lemma \ref{lemma1}. If $h$ is small enough, we get $F(\widehat{b})<F(\widetilde{b}^*)$, which contradicts the optimality of $\widetilde{b}^*$.

Consider now general situation, i.e. without assumption concerning the order of $D|y|$. Suppose that $\pi$, with corresponding matrix $P_{\pi}$, is permutation which orders $D|y|$. Define $y_{\pi}: = P_{\pi}y$ and $D_{\pi}: =P_{\pi}DP_{\pi}^\mathsf{T}$. Applying the procedure described in the statement of Lemma simultaneously to $(y, D, \lambda)$ for $j$, and to $(y_{\pi}, D_{\pi}, \lambda)$ for $\pi(j)$ we end with $\big(\widetilde{y},\widetilde{D},\widetilde{\lambda},\widetilde{R}_1^j(\widetilde{b}^*)\big)$ and $\big(\widetilde{y_{\pi}},\widetilde{D_{\pi}},\widetilde{\lambda},\widetilde{R}_2^{\pi(j)}(\widetilde{b}_{\pi}^*)\big)$. It is straightforward to see, that there exists permutation $\widetilde{\pi}:\{1,\ldots,p-1\}\longrightarrow \{1,\ldots,p-1\}$ such that $\widetilde{y_{\pi}}=P_{\widetilde{\pi}}\widetilde{y}$ and $\widetilde{D_{\pi}} = P_{\widetilde{\pi}}\widetilde{D}P_{\widetilde{\pi}}^\mathsf{T}$. From Proposition \ref{PropPrem} we have that $\widetilde{b}_{\pi}^*=P_{\widetilde{\pi}}\widetilde{b}^*$ and $\widetilde{R}_1^j(\widetilde{b}^*)=\widetilde{R}_2^{\pi(j)}(\widetilde{b}_{\pi}^*)$. Moreover, from the first part of proof $\widetilde{R}_2^{\pi(j)}(\widetilde{b}_{\pi}^*)=r-1,$ which gives the claim. 
\end{proof}

\newcommand{\E}{\mathbb{E}}
\renewcommand{\P}{\mathbb{P}}
\newcommand{\goto}{\rightarrow}
\newcommand{\var}{\operatorname{Var}}
\newcommand{\iid}{i.i.d. }
\newcommand{\floor}[1]{\lfloor #1 \rfloor}
\renewcommand{\d}{\mathrm{d}}
\section{Minimax estimation of gSLOPE}
\label{sec:minim-estim-gslope}
\begin{proof}[Proof of Theorem~\ref{minimax}]
Once again we will employ the equivalent formulation of gSLOPE under assumption about orthogonality at groups level, i.e. problem \eqref{17022353}, and we will consider statistically equivalent model $\widetilde{y}\sim \mathcal{N}\big(\widetilde{\beta},\ \sigma \mathbf{I}_{\widetilde{p}}\big)$, with $\widetilde{\beta}_{\II_i}=R_i\beta_{I_i}$, $i=1,\ldots,m$. Then $\XI{\beta}=\iI{\Beta}$ and for solution $b^*$ to \eqref{17022353} it holds $\iI{b^*}=\XI{\beta^\ES{gS}}$ for any solution $\beta^\ES{gS}$ to problem $\eqref{gSLOPE}$.
Without loss of generality, assume $\sigma = 1$. Note that $\|\y_{I_i}\|_2^2$ is distributed as the noncentral $\chi^2_{l_i}(\|\Beta_{\II_i}\|_2^2)$, where $\|\Beta_{\II_i}\|_2^2$ is the noncentrality. 

The lower bound of the minimax risk can be obtained as follows. For each $\II_{i}$, only $\Beta_j$ with the smallest index $j \in \II_i$ is \textit{possibly} nonzero and the rest $l_i-1$ components of $\Beta_{\II_i}$ are fixed to be zero. Then, this is reduced to a simple Gaussian sequence model with length $m$ and sparsity at most $k$. Given the condition $k/m \goto 0$, this classical sequence model has minimax risk $(1+o(1)) 2k\log(m/k)$ (see e.g. \cite{donoho1994minimax}).

Our next step is to evaluate the worst risk of gSLOPE over the nearly black project. We would completes the proof if we show this worst risk is bounded above by $(1+o(1)) 2k\log(m/k)$. For simplicity, assume that $\|\Beta_{\II_i}\|_2 = 0$ for all $i \ge k+1$ and write $\mu_i = \|\Beta_{\II_i}\|_2, \zeta_i = \|\y_{\II_i}\|_2 \sim \chi_{l_i}(\mu_i^2)$. Denote by $\widehat\zeta$ the SLOPE solution. Then, the risk is
\[
\E\|\widehat\zeta - \mu\|_2^2 = \E \sum_{i=1}^k (\widehat\zeta_i - \mu_i)_2^2 + \E\sum_{i=k+1}^m \widehat\zeta_i^2.
\]
Then, it suffices to show
\begin{equation}\label{eq:on_risk}
\E \left[ \sum_{i=1}^k (\widehat\zeta_i - \mu_i)^2 \right] \le (1+o(1)) 2k\log(m/k)
\end{equation}
and
\begin{equation}\label{eq:off_risk}
\E \left[ \sum_{i=k+1}^m \widehat\zeta_i^2 \right] = o(1)2k\log(m/k).
\end{equation}
Below, Lemmas~\ref{lm:off_supp1}, \ref{lm:off_supp2}, and \ref{lm:off_supp3} together give \eqref{eq:off_risk}. The remaining part of this proof serves to validate \eqref{eq:on_risk}. To start with, we employ the representation $\zeta_i^2 = (\xi_{i1} + \mu_i)^2 + \xi_{i2}^2 + \cdots + \xi_{i l_i}^2$ for \iid $\xi_{ij} \sim \mathcal{N}(0,1)$ (we can assume this representation without loss of generality, since the distribution of $(\xi_{i1} + a_1)^2 + (\xi_{i2}+a_2)^2 + \cdots + (\xi_{i l_i}+a_{l_i})^2$ depends only on the non-centrality $a_1^2 + \cdots + a_{l_i}^2$). As in the proof of Lemma 3.2 in \cite{WE}, we get
\begin{equation}\label{eq:tri_risk}
\begin{aligned}
\sum_{i=1}^k (\widehat\zeta_i - \mu_i)^2  &\le  \left( \|\widehat\zeta_{[1:k]} - \zeta_{[1:k]}\|_2 + \|\zeta_{[1:k]} - \mu_{[1:k]}\|_2 \right)^2\\
& \le \left( \|\lambda_{[1:k]}\|_2 + \|\zeta_{[1:k]} - \mu_{[1:k]}\|_2 \right)^2.
\end{aligned}
\end{equation}
As $l$ is fixed and $k/m \goto 0$, \cite{chisquareQ} gives $\lambda_i \sim \sqrt{2\log\frac{m}{qi}}$ for all $i \le k$. From this we know
\begin{equation}\label{eq:lambda_sum}
\|\lambda_{[1:k]}\|_2^2 = \sum_{i=1}^k \lambda_i^2 \sim 2k\log\frac{m}{k}.
\end{equation}
Next, we see
\begin{multline}\nonumber
\left| \sqrt{(\xi_{i1} + \mu_i)^2 + \xi_{i2}^2 + \cdots + \xi_{i l_i}^2} - \mu_i \right| \le \sqrt{\xi_{i2}^2 + \cdots + \xi_{i l_i}^2} + |\xi_{i1}| \\
\le 2\sqrt{\xi_{i1}^2 + \xi_{i2}^2 + \cdots + \xi_{i l_i}^2} \equiv 2\|\xi_i\|_2,
\end{multline}
which yields
\begin{equation}\label{eq:chi_sum}
\|\zeta_{[1:k]} - \mu_{[1:k]}\|_2^2 \le 4\sum_{i=1}^k \|\xi_i\|_2^2
\end{equation}
Note that $\sum_{i=1}^k \|\xi_i\|_2^2$ is distributed as the chi-square with $l_1 + \cdots + l_k \le lk$ degrees of freedom. Taking \eqref{eq:lambda_sum} and \eqref{eq:chi_sum} together, from \eqref{eq:tri_risk} we get
\begin{align*}
\E \left[ \sum_{i=1}^k (\widehat\zeta_i - \mu_i)^2 \right]  &\le \|\lambda_{[1:k]}\|_2^2 +\E \|\zeta_{[1:k]} - \mu_{[1:k]}\|_2^2 + 2\|\lambda_{[1:k]}\|_2 \E \|\zeta_{[1:k]} - \mu_{[1:k]}\|_2\\
&\le (1+o(1))2k\log\frac{m}{k} + 4lk + 2\sqrt{(1+o(1))2k\log\frac{m}{k}} \cdot \sqrt{4lk}\\
& \sim (1+o(1))2k\log\frac{m}{k},
\end{align*}
where the last step makes use of $m/k \goto \infty$. This establishes \eqref{eq:on_risk} and consequently completes the proof.

\end{proof}

The following three lemmas aim to prove \eqref{eq:off_risk}. Denote by $\zeta_{(1)} \ge \cdots \ge \zeta_{(m-k)}$ the order statistics of $\zeta_{k+1}, \ldots, \zeta_m$. Recall that $\zeta_i \sim \chi_{l_i}$ for $i \ge k+1$. As in the proof of Lemma 3.3 in \cite{WE}, we have
\[
\sum_{i=k+1}^m \widehat\zeta_i^2 \le \sum_{i=1}^{m-k} (\zeta_{(i)} - \lambda_{k+i})_+^2,
\]
where $x_+ = \max\{x, 0\}$. For a sufficiently large constant $A > 0$ and sufficiently small constant $\alpha > 0$ both to be specified later, we partition the sum into three parts:
\[
\sum_{i=1}^{m-k} \E (\zeta_{(i)} - \lambda_{k+i})_+^2 = \sum_{i=1}^{\floor{Ak}} \E (\zeta_{(i)} - \lambda_{k+i})_+^2 + \sum_{i=\lceil Ak \rceil }^{\floor{\alpha m}} \E (\zeta_{(i)} - \lambda_{k+i})_+^2 + \sum_{i=\lceil \alpha m \rceil }^{m-k} \E (\zeta_{(i)} - \lambda_{k+i})_+^2
\]
The three lemmas, respectively, show that each part is negligible compared with $2k\log(m/k)$. We indeed prove a stronger version in which the order statistics $\zeta_{(1)} \ge \cdots \ge \zeta_{(m-k)} \ge \zeta_{(m-k+1)} \ge \cdots \ge \zeta_{(m)}$ come from $m$ \iid $\chi_l$. Let $U_1, \ldots, U_m$ be \iid uniform random variables on $(0, 1)$, and $U_{(1)} \le U_{(2)} \le \cdots \le U_{(m)}$ be the \textit{increasing} order statistics. So we have the representation $\zeta_{(i)} = F^{-1}_{\chi_l}(1 - U_{(i)})$

\begin{lemma}\label{lm:off_supp1}
Under the preceding conditions, for any $A > 0$ we have
\[
\frac1{2k\log(m/k)} \sum_{i=1}^{\lfloor Ak \rfloor} \E(\zeta_{(i)} - \lambda_{k+i})_+^2 \rightarrow 0.
\]
\end{lemma}
\begin{proof}[Proof of Lemma~\ref{lm:off_supp1}]
Recognizing that $l$ is fixed, from \cite{chisquareQ} it follows that
\[
F^{-1}_{\chi_l}(1 - q_1) - F^{-1}_{\chi_l}(1 - q_2) \sim \sqrt{2\log\frac1{q_1}} - \sqrt{2\log\frac1{q_2}}
\]
for $q_1, q_2 \goto 0$. We also know that $\zeta_i$ is distributed as $F^{-1}_{\chi_l}(1 - U_{(i)})$. Making use of these facts, we get
\begin{equation}\nonumber
\begin{aligned}
\E (\zeta_{(i)} - \lambda_{k+i})_+^2 &= \E (F^{-1}_{\chi_l}(1 - U_{(i)}) - F^{-1}_{\chi_l}(1 - q(k+i)/m))_+^2 \\
&\sim \E \left( \sqrt{2\log\frac1{U_{(i)}}} - \sqrt{2\log\frac{m}{q(k+i)}} \right)_+^2\\
&\le \E \left( \sqrt{2\log\frac1{U_{(i)}}} - \sqrt{2\log\frac{m}{q(k+i)}} \right)^2\\
&\lesssim \E \left( \frac{\log^2(q(k+i)/m U_{(i)})}{\log(m/q(k+i))} \right).
\end{aligned}
\end{equation}
Now, we proceed to evaluate 
\[
\E \left[ \log^2\frac{q(k+i)}{m U_{(i)}} \right]= \log^2\frac{q(k+i)}{m} + \E \log^2 U_{(i)} - 2\log\frac{q(k+i)}{m} \E \log U_{(i)}.
\]
Observing that $U_{(i)}$ follows $\mathrm{Beta}(i, m+i-i)$, we get (see e.g. \cite{abramowitz1964})
\[
\begin{aligned}
&\E \log U_{(i)} = -\log\frac{m+1}{i} + \delta_1,\\
&\E \log^2 U_{(i)} = \left( \log\frac{m+1}{i} - \delta_1 \right)^2 + \frac1i - \frac1{m+1} + \delta_2
\end{aligned}
\]
for some $\delta_1 = O(1/i)$ and $\delta_2 = O(1/i^2)$. Thus we can evaluate $\E\log^2\frac{q(k+i)}{mU_{(i)}}$ as
\begin{multline}\nonumber
\E\log^2\frac{q(k+i)}{mU_{(i)}} = \log^2\frac{q(k+i)}{m} - 2\log\frac{q(k+i)}{m}\E \log U_{(i)} + \E\log^2U_{(i)}\\
= \log^2\frac{q(k+i)}{m} + 2\log\frac{q(k+i)}{m}\left(\log\frac{m+1}{i} - \delta_1\right) + \left(\log\frac{m+1}{i} - \delta_1\right)^2 + \frac{1}{i} - \frac{1}{m+1} + \delta_2\\
 = \log^2\frac{q(k+i)(m+1)}{im} - 2\delta_1\log\frac{q(k+i)(m+1)}{im} + \frac1i  - \frac1{m+1} + \delta_1^2 + \delta_2.
\end{multline}
Hence, we get
\begin{align*}
&\sum_{i=1}^{\floor{Ak}}\E(\zeta_{(i)} - \lambda_{k+i})_+^2 \\
&\lesssim \frac1{\log\frac{m}{q(A+1)k}}\left(\underbrace{\sum_{i=1}^{\floor{Ak}}\log^2\frac{q(k+i)(m+1)}{im}}_{\mbox{I}} - \underbrace{\sum_{i=1}^{\floor{Ak}}2\delta_1\log\frac{q(k+i)(m+1)}{im}}_{\mbox{II}} + \underbrace{\sum_{i=1}^{\floor{Ak}}\big(\frac1{i} - \frac1{m+1} + \delta_1^2 + \delta_2\big)}_{\mbox{III}} \right)\\
&= \frac1{\log\frac{m}{q(A+1)k}} (\mbox{I} + |\mbox{II}| + |\mbox{III}|).
\end{align*}
Since $\frac{m}{q(A+1)k} \goto \infty$. The proof would be completed once we show $\mbox{I}, |\mbox{II}|$, and $|\mbox{III}|$ are bounded. To this end, first note that
\begin{equation}\nonumber
\begin{aligned}
\mbox{I} &= \sum_{i=1}^{\floor{Ak}} \log^2\frac{q(k+i)(m+1)}{im} \\&\le \sum_{i=1}^{\floor{Ak}}\max\left\{k\int^{i/k}_{(i-1)/k}\log^2\frac{q(m+1)(1+x)}{mx}\d x, k\int^{(i+1)/k}_{i/k}\log^2\frac{q(m+1)(1+x)}{mx}\d x \right\}\\
& \le 2k\int^{A+1}_0 \log^2\frac{q(m+1)(1+x)}{mx}\d x \asymp k = o\left(2k\log\frac{m}{k}\right).
\end{aligned}
\end{equation}
The second term $\mbox{II}$ obeys
\begin{equation}\nonumber
\begin{aligned}
|\mbox{II}| &\le \sum_{i=1}^{\floor{Ak}}2\Big|\delta_1 \log\frac{q(k+i)(m+1)}{im}\Big| \lesssim \sum_{i=1}^{\floor{Ak}}\frac1i \Big|\log\frac{q(k+i)(m+1)}{im}\Big| \\
&\le \sum_{i=1}^{\floor{Ak}}\max\left\{ k\int^{i/k}_{(i-1)/k}\Big|\log\frac{q(m+1)(1+x)}{mx}\Big|\d x, k\int^{(i+1)/k}_{i/k}\Big|\log\frac{q(m+1)(1+x)}{mx}\Big|\d x \right\}\\
&\le 2k\int^{A+1}_{0}\Big|\log\frac{q(m+1)(1+x)}{mx}\Big|\d x \asymp k = o\left(2k\log\frac{m}{k} \right),
\end{aligned}  
\end{equation}
where we use the fact that $\int^{A+1}_{0}\Big|\log\frac{q(m+1)(1+x)}{mx}\Big|\d x$ is bounded by some constant.
The last term is simply bounded as
\begin{equation}\nonumber
|\mbox{III}| \le \sum_{i=1}^{\floor{Ak}}\Big|\frac1i  - \frac1{m+1} + \delta_1^2 + \delta_2\Big| \lesssim \sum_{i=1}^{\floor{Ak}}\frac{1}{i} \lesssim \log (Ak) = o\left(2k\log\frac{m}{k}\right).
\end{equation}
Combining these established bounds on $\mbox{I}, \mbox{II}$, and $\mbox{III}$ finishes proof.

\end{proof}


\begin{lemma}\label{lm:off_supp2}
Under the preceding conditions, let $A$ be any constant satisfying $q(1+A)/A < 1$ and $\alpha$ be sufficiently small such that $l/\lambda_{k + \floor{\alpha m}} < 1/2$. Then, 
\[
\frac1{2k\log(m/k)} \sum_{i = \lceil Ak \rceil}^{\lfloor \alpha m \rfloor} \E(\zeta_{(i)} - \lambda_{k+i})_+^2 \rightarrow 0.
\]  
\end{lemma}
\begin{proof}[Proof of Lemma~\ref{lm:off_supp2}]
Note that $\lambda_{k + \floor{\alpha m}} \sim \sqrt{2\log\frac{m}{q(k + \floor{\alpha m})}} \sim \sqrt{2\log\frac1{q\alpha}}$. So it is clear that such $\alpha$ exists. Pick any fixed $i$ between $\lceil Ak \rceil$ and$\lfloor \alpha m \rfloor$. As in the proof of Lemma A.4 in \cite{WE}, denote by $\alpha_u = \P(\chi_l > \lambda_{k+i} + u)$. Note that
\begin{equation}\nonumber
\begin{aligned}
\alpha_u &= \P(\chi_l > \lambda_{k+i} + u) = \int^{\infty}_{(\lambda_{k+i} + u)^2} \frac1{\mathrm{e}^{l/2}\Gamma(l/2)} x^{l/2-1} \mathrm{e}^{-x/2} \mathrm{d} x\\
&= \int^{\infty}_{\lambda_{k+i}^2} \frac1{\mathrm{e}^{l/2}\Gamma(l/2)} \left(\frac{(\lambda_{k+i}+u)^2}{\lambda_{k+i}^2} y \right)^{l/2-1} \exp\left(- \frac{(\lambda_{k+i}+u)^2}{2\lambda^2_{k+i}}y \right) \mathrm{d} \frac{(\lambda_{k+i}+u)^2}{\lambda_{k+i}^2} y\\
&= \left(1 + \frac{u}{\lambda_{k+i}}\right)^l \int^{\infty}_{\lambda_{k+i}^2} \frac1{\mathrm{e}^{l/2}\Gamma(l/2)} y^{l/2-1} \exp\left(- \frac{(\lambda_{k+i}+u)^2}{2\lambda^2_{k+i}}y \right) \mathrm{d} y\\
&\le \left(1 + \frac{u}{\lambda_{k+i}}\right)^l \mathrm{e}^{-\lambda_{k+i} u} \int^{\infty}_{\lambda_{k+i}^2} \frac1{e^{l/2}\Gamma(l/2)} y^{l/2-1} \mathrm{e}^{-y/2}\mathrm{d} y\\
&= \left(1 + \frac{u}{\lambda_{k+i}}\right)^l \mathrm{e}^{-\lambda_{k+i} u} \alpha_0\\
&\le \exp \left( \frac{l}{\lambda_{k+i}} u-\lambda_{k+i}u \right) \alpha_0.
\end{aligned}
\end{equation}
With the proviso that $l/\lambda_{k+\lfloor \alpha m \rfloor} < 1/2 < \lambda_{k+\lfloor \alpha m \rfloor}/2$, it follows that
\[
\alpha_u \le \mathrm{e}^{-\lambda_{k+i}u/2 }\alpha_0.
\]
The remaining proof follows from exactly the same reasoning as that of Lemma A.4 in \cite{WE}.

\end{proof}


\begin{lemma}\label{lm:off_supp3}
Under the preceding conditions, for any constant $\alpha > 0$ we have
\[
\frac1{2k\log(m/k)} \sum_{i = \lceil \alpha m \rceil}^{m - k} \E(\zeta_{(i)} - \lambda_{k+i})_+^2 \rightarrow 0.
\]
\end{lemma}

\begin{proof}[Proof of Lemma~\ref{lm:off_supp3}]
Recognizing that the value of the summation increases as $\alpha$ decreases, we only prove the lemma for sufficiently small $\alpha$. In the case of $U_{(i)} \ge \alpha/3$, we get
\begin{equation}\nonumber
\begin{aligned}
(\zeta_{(i)} - \lambda_{k+i})_+ &= \left( F^{-1}_{\chi_l}(1 - U_{(i)}) - F^{-1}_{\chi_l}(1 - q(k+i)/m) \right)_+\\
&\asymp \left(1 - U_{(i)} - (1 - q(k+i)/m)\right)_+\\
&= (q(k+i)/m - U_{(i)})_+,
\end{aligned}
\end{equation}
since both $U_{(i)}$ and $q(k+i)/m$ are bounded below away from zero. Otherwise, we use the trivial inequality $(\zeta_{(i)} - \lambda_{k+i})_+ \le \zeta_{(i)}$. In either case, we get
\[
\begin{aligned}
(\zeta_{(i)} - \lambda_{k+i})_+^2  &\lesssim \zeta_{(i)}^2 \bm{1}_{U_{(i)} < \frac{\alpha}{3}} + \left(\frac{q(k+i)}{m} - U_{(i)} \right)_+^2\\
&= \Big(F^{-1}_{\chi_l}(1 - U_{(i)}) \Big)^2 \bm{1}_{U_{(i)} < \frac{\alpha}{3}} + \left(\frac{q(k+i)}{m} - U_{(i)} \right)_+^2\\
&\asymp 2\log \left( \frac1{ U_{(i)}} \right) \bm{1}_{U_{(i)} < \frac{\alpha}{3}} + \left(\frac{q(k+i)}{m} - U_{(i)} \right)_+^2\\
&\lesssim \log \left( \frac1{ U_{(i)}} \right) \bm{1}_{U_{(i)} < \frac{\alpha}{3}} + \bm{1}_{U_{(i)} \le \frac{q(k+i)}{m}}.\\
\end{aligned}
\]
Hence,
\[
\sum_{i = \lceil \alpha m \rceil}^{m - k} \E(\zeta_{(i)} - \lambda_{k+i})_+^2 \lesssim \sum_{i = \lceil \alpha m \rceil}^{m - k} \E\left(\log \left( \frac1{ U_{(i)}} \right); U_{(i)} < \frac{\alpha}{3} \right) + \sum_{i = \lceil \alpha m \rceil}^{m - k} \P\left(U_{(i)} \le \frac{q(k+i)}{m}\right)
\]
In the remaining proof we aim to show
\begin{equation}\label{eq:third_alpha_sum}
\sum_{i = \lceil \alpha m \rceil}^{m - k} \E\left(\log \left( \frac1{ U_{(i)}} \right); U_{(i)} < \frac{\alpha}{3} \right) \goto 0
\end{equation}
and
\begin{equation}\label{eq:count_alpha}  
\sum_{i = \lceil \alpha m \rceil}^{m - k} \P\left(U_{(i)} \le \frac{q(k+i)}{m}\right) \goto 0.
\end{equation}
This is more than we need since $2k\log(m/k) \goto \infty$.

Each summand of \eqref{eq:third_alpha_sum} is bounded above by
\begin{equation}\nonumber
\begin{aligned}
\E\left(\log \left( \frac1{ U_{(\lceil \alpha m \rceil)}} \right); U_{(\lceil \alpha m \rceil)} < \frac{\alpha}{3} \right) &= \displaystyle\int^{\frac{\alpha}{3}}_0 \frac{x^{\lceil \alpha m \rceil-1}(1-x)^{m-\lceil \alpha m \rceil} \log \frac1x}{\operatorname{B}(\lceil \alpha m \rceil, m+1-\lceil \alpha m \rceil)} \d x\\
&\le \int^{\frac{\alpha}{3}}_0 \frac{x^{\lceil \alpha m \rceil-1}\log \frac1x}{\operatorname{B}(\lceil \alpha m \rceil, m+1-\lceil \alpha m \rceil)} \d x  \\
&= \frac1{\lceil \alpha m \rceil^2 \operatorname{B}(\lceil \alpha m \rceil, m+1-\lceil \alpha m \rceil)}\int^{(\frac{\alpha}{3})^{\lceil \alpha m \rceil}}_0 \log\frac1{y} ~ \d y\\
&\sim \frac{(\alpha/3)^{\lceil \alpha m \rceil} \log\frac3{\alpha}}{\lceil \alpha m \rceil \operatorname{B}(\lceil \alpha m \rceil, m+1-\lceil \alpha m \rceil)}.
\end{aligned}
\end{equation}
The last line obeys
\[
\begin{aligned}
\log\left[ \frac{(\alpha/3)^{\lceil \alpha m \rceil} }{\operatorname{B}(\lceil \alpha m \rceil, m+1-\lceil \alpha m \rceil)} \right] &\sim -\alpha m \log\frac3{\alpha} + \alpha m \log\frac1{\alpha}  + (1-\alpha)m \log \frac1{1-\alpha}\\
& = -\alpha m \log 3  + (1-\alpha)m \log \frac1{1-\alpha}.\\
\end{aligned}
\]
For small $\alpha$, we get $-\alpha \log 3  + (1-\alpha) \log \frac1{1-\alpha} = -\alpha\log 3 + (1+o(1))(1-\alpha)\alpha = -(\log 3  -1 + o(1))\alpha$. (Note that $\log 3 -1 = 0.0986\ldots > 0$.) This immediately yields
\[
\E\left(\log \left( \frac1{ U_{(\lceil \alpha m \rceil)}} \right); U_{(\lceil \alpha m \rceil)} < \frac{\alpha}{3} \right) \sim \mathrm{e}^{-(\log 3  -1 + o(1))\alpha m},
\]
which implies \eqref{eq:third_alpha_sum} since $m\mathrm{e}^{-(\log 3  -1 + o(1))\alpha m} \goto 0$.

Next, we turn to show \eqref{eq:count_alpha}. Note that $\P\left(U_{(i)} \le \frac{q(k+i)}{m}\right)$ actually is the tail probability of the binomial distribution with $m$ trials and success probability $\frac{q(k+i)}{m}$. Hence, by the Chernoff bound, this
probability is bounded as
\[
\P\left(U_{(i)} \le \frac{q(k+i)}{m}\right) \le \exp\left( -m \operatorname{KL}(i/m || q(k+i)/m) \right),
\]
where $\operatorname{KL}(a||b) := a \log\frac{a}{b} + (1-a)\log\frac{1-a}{1-b}$ is the Kullback-Leibler divergence. Thanks to $i \ge \lceil \alpha m \rceil \gg k$, simple analysis reveals that
\[
\operatorname{KL}(i/m || q(k+i)/m) \ge (1+o(1))i \left( \log\frac1q - 1 + q \right) / m.
\]
Combining the last two displays gives
\[
\P\left(U_{(i)} \le \frac{q(k+i)}{m}\right) \le \mathrm{e}^{-(1+o(1)) \left( \log\frac1q - 1 + q \right)i}.
\]
Plugging the above inequality into \eqref{eq:count_alpha} yields
\[
\sum_{i = \lceil \alpha m \rceil}^{m - k} \P\left(U_{(i)} \le \frac{q(k+i)}{m}\right) \le  \sum_{i = \lceil \alpha m \rceil}^{m - k} \mathrm{e}^{-(1+o(1)) \left( \log\frac1q - 1 + q \right)i}   \goto 0,
\]
where the last step follows from $\log\frac1q - 1 + q  > 0$ and $\lceil \alpha m \rceil \goto \infty$.
\end{proof}

\section{Strength of signals}
Consider the case when all submatrices $X_{I_i}$ have the same rank, $l>0$, $w>0$ is used as the universal weight and $X$ is orthogonal at groups level. From the interpretation of gSLOPE estimate coming from \eqref{17022353}, we see that the identification of the relevant groups could be summarized as follows: $\lambda$ decides on the number, $R$, of groups labeled as relevant, which correspond to indices of the $R$ largest values among $w^{-1}\|\y_{\II_1}\|_2,\ldots,w^{-1}\|\y_{\II_m}\|_2$. The random variables $w^{-1} \|\y_{\II_i}\|_2$ have a (possibly) non-central $\Chi$ distributions with $l$ degrees of freedom and noncentrality parameters given by the entries of $\iI{\Beta}$. Now, the nonzero $\|\Beta_{\II_i}\|_2$ could be perceived as a strong signal, if with the high probability the random variable having the noncentral $\Chi$ distribution with the noncentrality parameter $\|\Beta_{\II_i}\|_2$ is large comparing to the background composed of the independent random variables with the $\Chi_l$ distributions (then signal is likely to be identified by gSLOPE; otherwise, the signal could be easily covered by random disturbances and its identification has more in common with good luck than with the usage of particular method). The important quantity, which could be treated as a breaking point, is the expected value of the maximum of the background noise. Group effects being close to this value, could be perceived as medium under the orthogonal case and weak under the occurrence of correlations between groups. The above reasoning applied to the considered case, yields the issue of approximation of the expected value of the maximum of $m$ independent $\Chi_l$-distributed variables. Suppose that $\Psi_i\sim \Chi_l$ for $i=\{1,\ldots,m\}$. From Jensen's inequality we have 
$$\mathbb{E}\left(\max_{i=1,\ldots,m}\{\Psi_i\}\right) = \mathbb{E}\left(\sqrt{\max_{i=1,\ldots,m}\{\Psi_i^2\}}\right)\leq\sqrt{\mathbb{E}\left(\max_{i=1,\ldots,m}\{\Psi^2_i\}\right)},$$ 
hence we will replace the last problem by the problem of finding the reasonable upper bound on the expected value of the maximum of $m$ independent, $\Chi_l^2$-distributed variables.
\begin{theorem}
Let $\Psi_1,\ldots,\Psi_m$ be independent variables, $\Psi_i\sim\Chi_l^2$ for all $i$. Then
\begin{equation}
\label{05031231}
\mathbb{E}\left(\max_{i=1,\ldots,m}\{\Psi_i\}\right)\leq \frac{4\ln(m)}{1-m^{-\frac2l}}.
\end{equation}
\end{theorem}
\begin{proof}
Denote $M_m:=\max_{i=1,\ldots,m}\{\Psi_i\}$. From the Jensen's inequality applied to $e^{tM_m}$ we have
\begin{equation}
e^{t\mathbb{E}[M_m]}\leq \mathbb{E}\left[e^{tM_m}\right]=\mathbb{E}\left[\max_{i=1,\ldots,m}e^{t\Psi_i}\right]\leq\sum_{i=1}^m\mathbb{E}\left[e^{t\Psi_i}\right].
\end{equation}
We will consider only $t\in[0,\frac12)$. Since the moment generating function for $\Chi^2_l$ distribution is given by $MGF:=(1-2t)^{-\frac l2}$, for each $i$ it holds $\mathbb{E}\left[e^{t\Psi_i}\right]=(1-2t)^{-\frac l2}$ and we get $e^{t\mathbb{E}[M_m]}\leq m (1-2t)^{-\frac l2}$. Applying the natural logarithm to both sides yields
\begin{equation}
\label{05031230}
\mathbb{E}[M_m]\leq\frac{\ln(m)+\ln\left((1-2t)^{-\frac l2}\right)}{t},\ \ t\in\left[0, 1/2\right).
\end{equation}
Define $t_{m,l}: = \frac{1-m^{-\frac 2l}}2$. Then for all positive, natural numbers $l$ and $m$ we have $t_{m,l}\in[0,\frac12)$. Plugging $t_{m,l}$ to the right side of (\ref{05031230}) gives inequality (\ref{05031231}) and finishes the proof.
\end{proof}
\noindent The above theorem gives us the motivation to use the quantity $\sqrt{4\ln(m)/(1-m^{-2/l})}$ as the upper bound on the expected value of maximum over $m$ independent $\Chi_l$-distributed variables. In all simulations, which we have performed to investigate the performance of gSLOPE, we have generated the effects for truly relevant groups basing on these upper bounds. In particular, in experiments where $l_i$'s as well as weights were identical, we aimed at $\mathbb{E}\left(\|\y_{\II_i}\|_2\right)=\sqrt{4\ln(m)/(1-m^{-2/l})}$, for the truly relevant group $i$. Since $\mathbb{E}\left(\|\y_{\II_i}\|_2\right)\approx \sqrt{\|\Beta_{\II_i}\|_2^2+l}$, this yields the setting
\begin{equation}
\label{05031641}
\|\Beta_{\II_i}\|_2 = B(m,l), \qquad\textrm{for}\qquad B(m,l):=\sqrt{4\ln(m)(1-m^{-2/l})-l}
\end{equation}
for groups chosen to be truly relevant.

\section{Dual norm and conjugate of grouped sorted $\ell_1$ norm}
\label{subsec:app26112110}
Let $f:\mathbb{R}^p\rightarrow \mathbb{R}$ be a norm. We will use notation $f^D$ to refer to the dual norm to $f$, i.e function defined as $f^D(x):=\underset{b}\max\big\{x^\mathsf{T}b: f(b)\leq1\big\}$. It could be shown (see \cite{SLOPE2}), that the set $C_{\lambda}$, defined as
$C_{\lambda}:=\Big\{x\in\mathbb{R}^p:\ \sum_{i=1}^k|x|_{(i)}\leq\sum_{i=1}^k\lambda_i,\ k=1,\ldots,p\Big\}$, is unit ball of the dual norm to $J_{\lambda}$ for any nonnegative, nonincreasing sequence $\{\lambda_i\}_{i=1}^p$ with at least one nonzero element. We will now consider the dual norm to $J_{\lambda,I,W}(b)=J_{\lambda}\big(W\I{b}\big)$. It holds
\begin{equation}
\begin{split}
J_{\lambda,I,W}^D(x)=\ &\underset{b}\max\big\{x^\mathsf{T}b:\ J_{\lambda,I,W}(b)\leq 1\big\}=\underset{b}\max\big\{x^\mathsf{T}b:\ J_{\lambda}(W\I{b})\leq1\big\}=\\
&\underset{b,c}\max\big\{x^\mathsf{T}b:\ J_{\lambda}(c)\leq1,\ c=W\I{b}\big\}=\underset{c}\max\big\{x^\mathsf{T}b^c:\ J_{\lambda}(c)\leq1,\ c\succeq0\big\},
\end{split}
\end{equation}
where $b^c$ is defined as $b^c:=\argmax{b}\left\{x^\mathsf{T}b:\ c=W\I{b}\right\}.$ This problem is separable and for each $i$ we have $b^c_{I_i}=\argmax{}\big\{x_{I_i}^\mathsf{T}b_{I_i}:\ c^2_i=w_i^2\|b_{I_i}\|^2_2\big\}$. Applying the Lagrange multiplier method quickly yields $x_{I_i}^\mathsf{T}b_{I_i}^c=c_iw_i^{-1}\|x_{I_i}\|_2$. Consequently,
\begin{equation}
\begin{split}
J_{\lambda,I,W}^D(x) =&\ \underset{c}\max\big\{(W^{-1}\I{x})^\mathsf{T}c:\ J_{\lambda}(c)\leq1,\ c\succeq0\big\} =\\ &\ \underset{c}\max\big\{(W^{-1}\I{x})^\mathsf{T}c:\ J_{\lambda}(c)\leq1\big\} =
\ J_{\lambda}^D\big(W^{-1}\I{x}\big).
\end{split}
\end{equation}
Therefore, $\big\{x:\ J_{\lambda,I, W}^D(x)\leq 1\big\}=\big\{x:\ J_{\lambda}^D(W^{-1}\I{x})\leq 1\big\}=\big\{x:\ W^{-1}\I{x}\in C_{\lambda}\big\}.$ Since the conjugate of norm is equal to zero for arguments from unit ball of dual norm, and equal to infinity otherwise, we immediately get
\begin{corollary}
\label{07071055}
The conjugate function for $J_{\lambda,I,W}$ is the function $J^*_{\lambda,I,W}$ defined as
\begin{equation}
J^*_{\lambda,I,W}(x)=\left\{\begin{array}{cl}
0,&W^{-1}\I{x}\in C_{\lambda}\\
\infty,&\textrm{otherwise}
\end{array}
\right..
\end{equation}
\end{corollary}

\section{Proof of Theorem \ref{27111155}}
We will start with introducing an alternative definition of subgradient for convex function. Suppose that $f$ is convex function and for some $b, g\in \mathbb{R}^p$ it occurs $f(b+h)\geq f(b)+g^\mathsf{T}h$ for $h\in H$, where $H$ is open set containing zero. Let $h_0\in \mathbb{R}^p$ be arbitrary vector. Function $F:\mathbb{R}\rightarrow\mathbb{R}$, defined as ${F(t): = f(b+th_0)-tg^\mathsf{T}h_0}$, is convex. There exists $t_0\in(0,1)$ such that $t_0h_0\in H$, what gives
\begin{equation}
f(b)\leq F(t_0)=F\big((1-t_0)\cdot 0+t_0\cdot 1\big)\leq (1-t_0)f(b)+t_0F(1)
\end{equation}
and $f(b+h_0)\geq f(b)+g^\mathsf{T}h_0$ as a result. Above reasoning leads to
\begin{corollary}
\label{06182058}
For any open set $H$ containing zero the subgradient of convex function $f$ at $b$ could be equivalently defined as a vector $g$ satisfying $f(b+h)\geq f(b)+g^\mathsf{T}h,\ \textrm{for all}\ h\in H.$
\end{corollary}
\noindent We are now ready to prove Theorem \ref{27111155}. For $b\in \mathbb{R}^p$ and $J_{\lambda,I}(b):=J_{\lambda}\big(\I{b}\big)$ define the set
$H:=\big\{b\in\mathbb{R}^p:\ \|(b+h)_{I_1}\|_2>\ldots>\|(b+h)_{I_s}\|_2,\ \|(b+h)_{I_s}\|_2>\|h_{I_j}\|_2,\ j>s\big\}.$
If $g\in\partial J_{\lambda,I}(b)$, then for all $h\in H$ from definition of subgradient
\begin{equation}
\label{06191517}
\sum_{i=1}^s\lambda_i\|(b+h)_{I_i}\|_2+\sum_{i=s+1}^m\lambda_i\big(\I{(b+h)}\big)_{(i)}\geq\sum_{i=1}^s\lambda_i\|b_{I_i}\|_2 + \sum_{i=1}^sg_{I_i}^\mathsf{T}h_{I_i}+(g^c)^\mathsf{T}h^c,
\end{equation}
for $g^c:=(g_{I_{s+1}}^\mathsf{T},\ldots,g_{I_m}^\mathsf{T})^\mathsf{T}$ and $h^c:=(h_{I_{s+1}}^\mathsf{T},\ldots,h_{I_m}^\mathsf{T})^\mathsf{T}$. Define $\widetilde{I}:=\big\{\widetilde{I}_1,\ldots, \widetilde{I}_{m-s}\big\}$, with $\widetilde{I}_i:=\big\{(i-1)\cdot l + 1,\ldots, i\cdot l\big\}$. Then $\big\|(g^c)_{\widetilde{I}}\big\|_2 = \big\|g_{I^c}\big\|_2$. Consider first case, when $h$ belongs to the set $H^c:=\{h\in H:\ h_{I_i}\equiv0,\ i\leq s\}$. This yields
\begin{equation}
\label{06191400}
\sum_{i=1}^{m-s}\lambda_{s+i}\Big(\big(\|h_{I_{s+1}}\|_2,\ldots, \|h_{I_m}\|_2\big)^\mathsf{T}\Big)_{(i)}\geq (g^c)^\mathsf{T}h^c.
\end{equation}
Since $\{h^c:\ h\in H^c\}$ is open in $\mathbb{R}^{l(m-s)}$ and contains zero, from Corollary \ref{06182058} we have that $g^c\in \partial J_{\lambda^c,\widetilde{I}}(0)$ and the inequality (\ref{06191400}) is true for any $h^c\in\mathbb{R}^{l(m-s)}$ yielding  
\begin{equation}
0\geq \sup_{h^c}\Big\{(g^c)^\mathsf{T}h^c -J_{\lambda^c,\widetilde{I}}(h^c)\Big\} = J_{\lambda^c,\widetilde{I}}^*(g^c).
\end{equation}
The conjugate of $J_{\lambda^c,\widetilde{I}}$ is given by Corollary \ref{07071055}). This formulation immediately gives condition $\big\|(g^c)_{\widetilde{I}}\big\|_2\in C_{\lambda^c}$, which is equivalent with $\big\|g_{I^c}\big\|_2\in C_{\lambda^c}$. To find conditions for $g_{I_i}$ with $i\leq s$, define sets $H_i:=\{h\in H:\ h_{I_j}\equiv0,\ j\neq i\}$. For $h\in H_i$, (\ref{06191517}) reduces to $\lambda_i\|b_{I_i}+h_{I_i}\|_2\geq \lambda_i\|b_{I_i}\|_2+g_{I_i}^\mathsf{T}h_{I_i}$. Since the set $\{h_{I_i}:\ h\in H_i\}$ is open in $\mathbb{R}^l$ and contains zero, from Corollary \ref{06182058} we have $g_{I_i}\in \partial f_i(b_{I_i})$ for $f_i:\mathbb{R}^l\longrightarrow\mathbb{R}$, $f_i(x):=\lambda_i\|x\|_2$. For $i\leq s$, $f_i$ is differentiable in $b_{I_i}$, which gives $g_{I_i}=\lambda_i\frac{b_{I_i}}{\|b_{I_i}\|_2}$ for $i\in\{1,\ldots,s\}$. The statement follows from the fact that $\partial J_{\lambda,I}(wb)=w\cdot\partial J_{\lambda,I}(b)$.

\section{Expected values of random matrices}
\subsection{The proof of Lemma \ref{25031835}}

The claim is obvious for $n=1$ and we will assume that $n>1$. First, we will list some basic properties of $A_X$. It could be easily noticed that: $A_X$ is symmetric matrix, $A_X$ is idempotent matrix (meaning that $A_XA_X=A_X$) and that $\operatorname{trace}(A_X)=\operatorname{trace}(X^\mathsf{T}X(X^\mathsf{T}X)^{-1})=r$. We will now show that for each $i\in\{1,\ldots,n\}$, $j\in\{1,\ldots,r\}$ the support of a $A_X(i,j)$ distribution is bounded, which will give us the existence of the expected value. Let $\|A\|_F$ be the Frobenius norm. Then 
\begin{equation}
\big|(A_X)_{i,j}\big| \leq \|A\|_F = \sqrt{\operatorname{trace}(A_X^\mathsf{T}A_X)}=\sqrt{\operatorname{trace}(A_X)}=\sqrt{r}.
\end{equation}
We will use notation $E_X: = \mathbb{E}\left(A_X\right)$. Since entries of matrix $X$ are randomized independently with the same distribution, $E_X$ is invariant under permutation applied to rows, i.e. $E_X=E_{PX}$ for any permutation matrix $P$. This gives $E_X = PE_XP^\mathsf{T}$, which means that applying the same permutation to rows and columns has no impact on expected value. We will show that
\begin{equation}
\label{24032202}
(E_X)_{i,j}=(E_X)_{1,n},\textrm{ for }i<j.
\end{equation}
Consider first the case when $i=1$ and $1<j<n$. Denoting by $P_{j\leftrightarrow n}$ matrix corresponding to transposition which replaces elements $j$ and $n$, we have $(E_X)_{1,j} = \big(P_{j\leftrightarrow n}E_XP_{j\leftrightarrow n}^\mathsf{T}\big)_{1,j}= (E_X)_{1,n}$. When $j=n$ and $1<i<n$, the same reasoning works with $P_{1\leftrightarrow i}$. Suppose now, that $1<i<n$ and $1<j<n$. We get $(E_X)_{i,j}=(E_X)_{1,n}$ analogously by using arbitrary permutation matrix $P$ which replaces element $j$ with $n$ and element $i$ with $1$. Since $A_X$ is symmetric, (\ref{24032202}) is true also for $i>j$. On the other hand, for all $i,j \in\{1,\ldots, n\}$, we have $(E_X)_{i,i} = \big(P_{j\leftrightarrow i}E_XP_{j\leftrightarrow i}^\mathsf{T}\big)_{i,i}= (E_X)_{j,j}$. Consequently, all off-diagonal entries of $E_X$ are equal to some $t$ and all diagonal entries have the same value $d$. Since
\begin{equation}
nd = \operatorname{trace}(E_X) = \sum_{i=1}^n\mathbb{E}(A_X(i,i)) = \mathbb{E}\left(\sum_{i=1}^n(A_X)_{i,i}\right) = r,
\end{equation}
we have $d=\frac rn$ and it remains to show that $t=0$. Define $\Sigma: = \left[\begin{BMAT}{c0c}{c0c}-1&\mathbf{0}^\mathsf{T}\\ \mathbf{0}&\mathbf{I}_{n-1}\end{BMAT}\right]$. Then $\Sigma X_S$ differs from $X_S$ only by signs of the first row. Since entries of matrix $X_S$ have zero-symmetric distribution, we have $E_X = E_{\Sigma X}$. Now
\begin{equation}
\left[\begin{BMAT}{c0c}{c0c}d&\mathbf{1}_{n-1}^\mathsf{T}t\\ \mathbf{1}_{n-1}t& \ddots\end{BMAT}\right] = E_X = \Sigma E_X \Sigma = \left[\begin{BMAT}{c0c}{c0c}d&-\mathbf{1}_{n-1}^\mathsf{T}t\\ -\mathbf{1}_{n-1}t& \ddots\end{BMAT}\right], 
\end{equation}
which implies $t=0$ and proves the statement.
\subsection{The proof of Lemma \ref{31032152}}
It is easy to see that $M_{X,\lambda}$ is symmetric, positive semi-definite matrix. Denote by $\|M_{X,\lambda}\|_*$ the nuclear (trace) norm of matrix $M_{X,\lambda}$. We have
\begin{equation}
\label{01040935}
\begin{split}
&\mathbb{E}\big|(M_{X,\lambda})_{i,j}\big|\leq \mathbb{E}\big(\|M_{X,\lambda}\|_*\big)=\mathbb{E}\big(\operatorname{trace}(M_{X,\lambda})\big)=\mathbb{E}\big(\operatorname{trace}(H_{\lambda,\beta}^\mathsf{T}B_X^\mathsf{T}B_XH_{\lambda,\beta})\big)=\\
&\mathbb{E}\big(H_{\lambda,\beta}^\mathsf{T}(X^\mathsf{T}X)^{-1}H_{\lambda,\beta}\big)=\frac {n} {(n-r-1)}\,H_{\lambda,\beta}^\mathsf{T}H_{\lambda,\beta}=\frac {n\,\|\lambda^S\|_2^2} {n-r-1},
\end{split}
\end{equation}
since $X^\mathsf{T}X$ follows an inverse Wishart distribution. This gives the existence of $E_X:=\mathbb{E}(M_{X,\lambda})$. Analogously to situation in Lemma \ref{25031835}, $E_X$ is invariant under permutation or signs changes applied to rows of $X$, i.e. $E_X=E_{PX}$ for any permutation matrix $P$, and $E_X = E_{\Sigma X}$ for diagonal matrix $\Sigma$ with entries on diagonal coming from set $\{-1,1\}$. Since $E_{PX} = PE_XP^\mathsf{T}$ and $E_{\Sigma X} = \Sigma E_X\Sigma$, as before we have that $E_X$ is diagonal matrix with all diagonal entries having the same value $d$. The value $d$ could be easy found using (\ref{01040935}) since we have
\begin{equation}
n d = \operatorname{trace}\big(E_X\big) = \frac {n\,\|\lambda^S\|_2^2} {n-r-1}. 
\end{equation}

\section{Stopping criteria for numerical algorithm}
Without loss of generality assume that $\sigma=1$. We will start with optimization problem in (\ref{07031414}), namely
\begin{equation}
\label{24022120}
\minimize{\eta}\ \ f(\eta)=\frac12\|y-\X M\eta\|_2^2+J_{\lambda}\big(\iI{\eta}\big)
\end{equation}
for $\iI{\eta}=\Big(\|\eta_{\II_1}\|_2,\ldots,\|\eta_{\II_m}\|_2\Big)^\mathsf{T}$ and ${M_{\II_i,\II_i}=\frac1{w_i}\mathbf{I}_{l_i}}$, $i=1,\ldots,m$. This problem could be written in equivalent form
\begin{equation}
\label{07071412}
\begin{array}{cl}
\minimize{\eta, r, c}&\frac12\|r\|^2_2+c\\[.1cm]
\textrm{s.t.}&\left\{\begin{array}{l}J_{\lambda,\II}(\eta)-c\leq0\\y-r-\X M\eta=0\end{array}\right.
\end{array}
\end{equation}
$\big($notice that for $(\eta^*, r^*, c^*)$ being solution, it must occurs $c^*=J_{\lambda,\II}(\eta^*)\big)$. Since ($\ref{07071412}$) is convex and $(\eta_0, r_0, c_0)$, for $\eta_0=0$, $r_0=y$ and $c_0=1$, is strictly feasible, the strong duality holds. Lagrange dual function for this problem is given by
\begin{equation}
\begin{split}
g(\mu,\nu)=&\inf_{\eta,r,c}\ \bigg\{\frac12\|r\|_2^2+c+\mu^\mathsf{T}\big(y-r-\X M\eta\big)+\nu\big(J_{\lambda,\II}(\eta)-c\big)\bigg\}=\\
&\mu^\mathsf{T}y+\inf_r\ \bigg\{\frac12\|r\|^2_2-\mu^\mathsf{T}r\bigg\}+\inf_c\ \big\{c-\nu c\big\}+\inf_{\eta}\ \big\{-\mu^\mathsf{T}\X M\eta+\nu J_{\lambda,\II}(\eta)\big\}.
\end{split}
\end{equation}
Now, since the minimum of $\frac12\|r\|^2_2-\mu^\mathsf{T}r$ is taken for $r=\mu$, we have
\begin{equation}
g(\mu,\nu)=\mu^\mathsf{T}y-\frac12\|\mu\|^2_2+\inf_c\ \big\{c-\nu c\big\}-J_{\nu\lambda,\II}^*\big((\X M)^\mathsf{T}\mu\big).
\end{equation} 
Then $\nu^*=1$ and from Corollary \ref{07071055}, the dual problem to (\ref{07071412}) is equivalent to
\begin{equation}
\label{07071448}
\begin{array}{cl}
\maximize{\mu}&\mu^\mathsf{T}y-\frac12\|\mu\|_2^2\\[.1cm]
\textrm{s.t.}&\iI{M\X^\mathsf{T}\mu}\in C_{\lambda}
\end{array}.
\end{equation}
Let $(\eta^*,r^*,c^*)$ be primal and $(\mu^*, 1)$ be dual solution to (\ref{07071412}). Obviously, $\mu^*=r^*=y-\X M\eta^*$ and $c^*=J_{\lambda,\II}(\eta^*)$. Furthermore, from strong duality we have
\begin{equation}
\frac12\|y-\X M\eta^*\|^2_2+J_{\lambda,\II}(\eta^*) = (y-\X M\eta^*)^\mathsf{T}y-\frac12\|y-\X M\eta^*\|_2^2,
\end{equation}
which gives $(\X M\eta^*)^\mathsf{T}(y-\X M\eta^*)=J_{\lambda,\II}\big(\eta^*\big)$. Now, for current approximate $\eta^{[k]}$ of solution to (\ref{24022120}), achieved after applying proximal gradient method, we define the current duality gap for $k$ step as
\begin{equation}
\rho(\eta^{[k]})=(\X M\eta^{[k]})^\mathsf{T}(y-\X M\eta^{[k]})-J_{\lambda,\II}\big(\eta^{[k]}\big)
\end{equation} 
and we will determine the infeasibility of $\mu^{[k]}:=y-\X M\eta^{[k]}$ by using the measure
\begin{equation}
\textrm{infeas}\big(\mu^{[k]}\big): = \max\left\{J^D_{\lambda,\II}\big(M \X^\mathsf{T}\mu^{[k]}\big)-1,0\right\}
\end{equation}
To define the stopping criteria we have applied the widely used procedure: treat $\rho(\eta^{[k]})$ as indicator telling how far $\eta^{[k]}$ is from true solution and terminate the algorithm when this difference and infeasibility measure are sufficiently small. Summarizing, we have derived algorithm according to scheme
\begin{algorithm}
  \caption{Algorithm for group SLOPE}
  \begin{algorithmic}[!]
    \State \textbf{input:} infeas.tol:  {\itshape positive number determining the tolerance for infeasibility};
		\State \ \ \ \ \ \ \ \ \ \ dual.tol:\ \ \ {\itshape positive number determining the tolerance for duality gap};
		\State \ \ \ \ \ \ \ \ \ \ \ \ $k:=0$,\ \  $\eta^{[0]}$,\ \ $\mu^{[0]}:=\mu(\eta^{[0]})$,\ \ $\textrm{infeas}^{[0]}:=\textrm{infeas}\big(\mu^{[0]}\big)$,\ \ $\rho^{[0]}:=\rho(\eta^{[0]})$;
    \While{ ( $\textrm{infeas}^{[k]}>\textrm{infeas.tol}$\ \  or\ \  $\rho^{[k]}>\textrm{dual.tol}$) }
		\State 1. $k\gets k+1$;
		\State 2. get $\eta^{[k]}$ from Algorithm \ref{24021312};
		\State 3. $\mu^{[k]}:=\mu(\eta^{[k]})$;
		\State 4. $\textrm{infeas}^{[k]}:=\textrm{infeas}\big(\mu^{[k]}\big)$, $\rho^{[k]}:=\rho(\eta^{[k]})$; 
    \EndWhile
		\State $\beta_{gS}: = M\eta^{[k]}$.
  \end{algorithmic}
\end{algorithm}

\end{document}